\documentclass[9pt]{amsart}
\usepackage[T1]{fontenc}
\usepackage[english]{babel}
\usepackage{amsmath,amsthm}
\usepackage{amsfonts}
\usepackage{comment}
\usepackage{graphicx}
\usepackage[dvipsnames]{xcolor}
\usepackage[pagebackref,hypertexnames=false,colorlinks=true,
            linkcolor=NavyBlue,
            citecolor=NavyBlue,
            urlcolor=NavyBlue]{hyperref} 
\usepackage{cleveref}
\usepackage{subcaption}
\usepackage[all]{xy}
\usepackage{thmtools, thm-restate}
\usepackage{overpic}
\usepackage{amssymb,tikz}
\usepackage{enumitem}
\usepackage{multicol}
\usepackage{alltt}

\usepackage{float}

\usetikzlibrary{tikzmark}
\usetikzlibrary{decorations.pathreplacing}

\newcommand*\circled[1]{\tikz[baseline=(char.base)]{
            \node[shape=circle,draw,inner sep=1pt] (char) {#1};}}

\newcommand{\red}{\textcolor{red}}

\usepackage[top=1.23in, bottom=1.22in, left=1.23in, right=1.23in]{geometry}

\newtheorem{thm}{Theorem}[section]
\newtheorem{cor}[thm]{Corollary}
\newtheorem{lem}[thm]{Lemma}

\newtheorem{maintheorem}{Theorem}

\theoremstyle{definition}
\newtheorem{defn}[thm]{Definition}
\newtheorem*{conditions}{Working Conditions}
\newtheorem{notation}[thm]{Notation}

\theoremstyle{remark}
\newtheorem{rem}[thm]{Remark}

\numberwithin{equation}{section}
\newcommand{\spin}{\ifmmode{\rm Spin}\else{${\rm spin}$\ }\fi}
\newcommand{\spinc}{\ifmmode{{\rm Spin}^c}\else{${\rm spin}^c$}\fi}

\newcommand{\Z}{\mathbb{Z}}

\DeclareMathOperator{\tors}{tors}

\newcommand{\Lsum}{\#_i L(p_i,q_i)}
\newcommand{\Xsum}{\natural_i X(p_i,q_i)}

\newcommand{\plumb}{\entrymodifiers={+[o][F-]} \xymatrix@C=8pt}
\newcommand{\el}{\ar@{-}[r]}
\newcommand{\ed}{\ar@{..}[r] }

\DeclareMathOperator{\supp}{Supp}

\date{\today}%
\begin{document}

\title[Forbidden Configurations and  Definite Fillings of Lens Spaces]{Forbidden Configurations and Definite Fillings of Lens Spaces}%

\author{Antony T.H. Fung}%
\address {Korea Advanced Institute of Science and Technology}
\email{antonyfung@kaist.ac.kr}

\author{JungHwan Park}%
\address {Korea Advanced Institute of Science and Technology}
\email{jungpark0817@kaist.ac.kr}

\begin{abstract}
We study definite fillings of lens spaces. We classify the lens spaces $L(p,q)$ for which every smooth negative-definite filling $X$ satisfies
\[
b_2(X)\ge b_2(X(p,q))-1,
\]
where $X(p,q)$ denotes the canonical negative-definite plumbing. The classification is given by $17$ ``forbidden configurations'' that cannot appear as induced subgraphs of the canonical plumbing graph. More generally, we introduce a combinatorial framework that encodes the lattice embedding information coming from the dual plumbing of $X(p,q)$, and we prove that it is governed by a finite set of minimal forbidden configurations. We also discuss consequences for symplectic fillings of lens spaces and for smoothings of cyclic quotient singularities.
\end{abstract}

\maketitle
% ----------------------------------------------------------------
%\tableofcontents

\section{Introduction}
A fundamental problem in low-dimensional topology is to understand the topology of $4$-manifolds bounded by a given closed oriented $3$-manifold $Y$. In this article, we focus on
\emph{definite fillings}: smooth, compact, oriented $4$-manifolds $X$ with $\partial X=Y$ whose
intersection form $Q_X$ is definite.  If a compact $4$-manifold $X$ is a filling of a closed $3$-manifold $Y$, then we say that the intersection form $Q_X$ \emph{fills} $Y$, or that $Q_X$ is \emph{bounded} by $Y$. The guiding question of this article is:
\begin{quote}
Understand the definite intersection forms bounded by a closed oriented $3$-manifold $Y$.
\end{quote}
Note that Donaldson's diagonalization theorem~\cite{Donaldson:1987-1} gives a complete answer when $Y=S^3$: any definite intersection
form bounded by $S^3$ is the standard diagonal lattice. For more recent developments and related studies, see for example~\cite{Froyshov:1995-1, Lisca:2007-1,
Lisca:2007-2, Owens-Strle:2012-1, Choe-Park:2018-2, Scaduto:2018-1, Golla-Scaduto:2019-1, Aceto:2020, Choe-Park:2021, Froyshov:2023, definite}.

In~\cite{definite}, Aceto--McCoy--Park initiated a systematic study of this question for \emph{lens spaces}. One of their main results is a classification of the connected sums of lens spaces that behave like $S^3$ in the following sense: up to stabilization by $\langle -1\rangle$ summands, they admit a unique negative-definite intersection form among negative-definite fillings.\footnote{Note that if $Y$ bounds a negative-definite $4$--manifold $X$ with intersection form $Q_X$, then by blowing up $X$ we see that $Y$ also bounds the stabilized form $Q_X \oplus \langle -1\rangle^{\oplus n}$ for any $n>0$.} They proved that there exist $10$ ``forbidden configurations'' of linear weighted graphs such that a connected sum of lens spaces has this unique-intersection-form property if and only if the canonical plumbing graph associated to the connected sum does not contain any of these configurations. In this article, we show that analogous ``forbidden configurations'' arise in a more general setting. To state this precisely, we first introduce some terminology.

Recall that the \emph{lens space} $L(p,q)$ is obtained by $-p/q$--Dehn surgery on the unknot in $S^3$, where $p$ and $q$ are coprime integers with $p>q>0$. It bounds a \emph{canonical negative-definite plumbing} $X(p,q)$, namely the smooth negative-definite $4$--manifold obtained by plumbing disk bundles over spheres along the \emph{linear} weighted graph
\[
\begin{tikzpicture}[xscale=1.0,yscale=1,baseline={(0,0)}]
    \node at (1-0.1, .4) {$-a_1$};
    \node at (2-0.1, .4) {$-a_2$};
    \node at (3-0.1, .4) {$-a_3$};
    \node at (5-0.1, .4) {$-a_{n}$};
    \node (A1_1) at (1, 0) {$\bullet$};
    \node (A1_2) at (2, 0) {$\bullet$};
    \node (A1_3) at (3, 0) {$\bullet$};
    \node (A1_4) at (4, 0) {$\cdots$};
    \node (A1_5) at (5, 0) {$\bullet$};
    \path (A1_2) edge [-] node [auto] {$\scriptstyle{}$} (A1_3);
    \path (A1_3) edge [-] node [auto] {$\scriptstyle{}$} (A1_4);
    \path (A1_4) edge [-] node [auto] {$\scriptstyle{}$} (A1_5);
    \path (A1_1) edge [-] node [auto] {$\scriptstyle{}$} (A1_2);
\end{tikzpicture}
\]
where the integers $a_i\ge 2$ are uniquely determined by the continued fraction expansion
\[
\frac{p}{q} = a_1 - \cfrac{1}{a_2 - \cfrac{1}{\ddots - \cfrac{1}{a_n}}}.
\]
More generally, a connected sum $\#_i L(p_i,q_i)$ bounds the boundary connected sum of the corresponding canonical negative-definite plumbings $\natural_i X(p_i,q_i)$, and the associated plumbing graph is the disjoint union of the plumbing graphs for the individual summands. We also recall that $-L(p,q)$, the lens space $L(p,q)$ with reversed orientation, is orientation-preservingly diffeomorphic to $L(p,p-q)$. Finally, the intersection form of a compact oriented $4$--manifold $X$ defines an integral lattice $(H_2(X)/\tors, Q_X)$, which we will often denote simply by $Q_X$.

Given a connected sum of lens spaces $L=\#_i L(p_i,q_i)$, we may glue
$\natural_i X(p_i,q_i)$ and $\natural_i X(p_i,p_i-q_i)$ along their common boundary to obtain a closed, smooth, negative-definite $4$--manifold $Z$.\footnote{In fact, it is diffeomorphic to a connected sum of copies of the complex projective plane; see, e.g., \cite[Lemma~3.1]{Ownes-Swenton:2024}.}
By Donaldson's theorem, the intersection form of $Z$ is isomorphic to $(\Z^{n},-\mathrm{Id})$, where
$n=b_2(\natural_i X(p_i,q_i))+b_2(\natural_i X(p_i,p_i-q_i))$.
In particular, the integral lattice $Q_{\natural_i X(p_i,p_i-q_i)}$ admits an embedding into $(\Z^{n},-\mathrm{Id})$.
In general, $Q_{\natural_i X(p_i,p_i-q_i)}$ may embed into a standard lattice of smaller rank.

% Let $Z$ be the closed smooth negative-definite 4-manifold obtained by gluing $X$ and $\natural_i X(p_i,p_i-q_i)$  Let $M = b_2(X)+b_2(\natural_i X(p_i,p_i-q_i))$. By Donaldson's Theorem \cite{Donaldson:1987-1}, the intersection form of $Z$ is isomorphic to $(\Z^M, -\mathrm{Id})$. Consequently, $Q_{\natural_i X(p_i,p_i-q_i)}$ embeds into $(\Z^M, -\mathrm{Id})$.
% Since $L$ does not satisfy Property $X_k$, the rank of the target lattice must satisfy
% $$M > N.$$
% Substituting the expressions for $M$ and $N$, we obtain
% $$b_2(X)+b_2(\natural_i X(p_i,p_i-q_i)) > b_2(\natural_i X(p_i,q_i))+b_2(\natural_i X(p_i,p_i-q_i))-k.$$
% Simplifying this inequality yields the desired result.

% In this case we say that $Q_X$ \emph{fills} $Y$. 

\begin{defn}\label{def:Xk}
Let $L=\#_i L(p_i,q_i)$ be a connected sum of lens spaces, and let $k$ be a positive integer. Set
\[
n_k = b_2(\natural_i X(p_i,q_i)) + b_2(\natural_i X(p_i,p_i-q_i)) - k.
\]
We say that $L$ satisfies \textit{Property $X_k$} if the integral lattice $Q_{\natural_i X(p_i,p_i-q_i)}$ does \emph{not} admit an embedding into the standard lattice $(\Z^{n_k},-\mathrm{Id})$.
\end{defn}

% \footnote{\JP{I am changing the definition of Property $X_k$; hopefully this does not create too much chaos. Okay, I think I made all the changes so that everything makes sense. I think it is important that you check that if I made all the corrections in Section 7..}}

% \textcolor{blue}{Note that every connected sum of lens spaces satisfies Property $X_k$ for some positive integer $k$.}
 
 The following elementary lemma motivates this definition. Its proof is a standard application
of Donaldson's diagonalization theorem, so we omit it.

\begin{lem}\label{lem:Xk}\pushQED{\qed}
Let $L = \#_i L(p_i, q_i)$ be a connected sum of lens spaces, and let $\natural_i X(p_i, q_i)$ denote the corresponding boundary connected sum of the canonical negative-definite plumbings.
If $L$ satisfies Property $X_k$ with a positive integer $k$, then every smooth negative-definite filling $X$ of $L$ satisfies
\[
b_2(X) \ge b_2(\natural_i X(p_i,q_i)) - (k-1). \qedhere
\]
\popQED
\end{lem}

% \begin{lem}\label{lem:Xk}
% Let $L=\#_i L(p_i,q_i)$ be a connected sum of lens spaces, and let $k\ge 1$.  If there exists a smooth negative-definite filling $X$ of
% $L$ such that
% \[
% b_2(X) \le b_2(\natural_i X(p_i,q_i)) - k,
% \]
% then $L$ satisfies Property $X_k$. \qed
% \end{lem}

% With this terminology, the main theorem of~\cite{definite} can be rewritten as follows.

With this terminology, the main theorem of~\cite[Theorem 1.1]{definite} can be summarized as follows. Let $L=\#_i L(p_i,q_i)$ be a connected sum of lens spaces. Then there exists a finite set of configurations $S_1$ with $|S_1|=10$ such that the following are equivalent:
\begin{enumerate}[label=(\alph*)]
    \item\label{a} $L$ satisfies Property $X_1$.
    \item\label{b} The plumbing graph associated to $L$ does not contain a configuration from $S_1$  as an induced subgraph.
    \item\label{c} Every smooth negative-definite filling $X$ of $L$ satisfies
    $b_2(X) \ge b_2(\natural_i X(p_i,q_i))$.
\end{enumerate}
As a consequence, $L$ satisfying any of the above conditions is equivalent to admitting, up to stabilization by $\langle -1\rangle$ summands, a unique negative-definite intersection form among negative-definite fillings.

A striking feature of this result is that the set of forbidden configurations $S_1$ can be chosen to have only $10$ elements, and moreover this choice is minimal: any set of configurations that characterizes the equivalence of the above conditions must necessarily contain $S_1$. The core of the proof of the main theorem of~\cite{definite} is to show, for this fixed set $S_1$, that \ref{b} implies \ref{a}.

Our first main theorem is that, for each positive integer $k$, there exists a \emph{finite} set of forbidden configurations $S_k$ that makes the analogous statements to \ref{a} and \ref{b} equivalent for each $k$:

\begin{maintheorem}\label{thm:all k}
Let $k$ be a positive integer. There exists a finite set of minimal configurations $S_k$ satisfying the following conditions:
\begin{enumerate}[font=\upshape]
    \item A connected sum of lens spaces $L$ satisfies Property $X_k$ if and only if the plumbing graph associated to $L$ does not contain a configuration from $S_k$ as an induced subgraph;
    \item the absolute value of the weight of any vertex in any configuration in $S_k$ is at most $2k+7$;
    \item the number of vertices in any configuration in $S_k$ is at most $13k$.
\end{enumerate}
\end{maintheorem}

While this theorem guarantees the existence of a finite list of forbidden configurations, the resulting bounds are far from optimal. For instance, when $k=1$, the estimate suggests on the order of $800$ million candidates even after restricting to the connected elements of $S_1$, whereas~\cite{definite} showed that the actual list consists of only $10$ configurations. 

Moreover, by Lemma~\ref{lem:Xk} we have \ref{a} implies \ref{c}, so in order to obtain the equivalence, it remains to prove that \ref{c} implies \ref{b}. When $k=1$, this does turn out to be true, but it fails in general. We invite the reader to check that for $k=8$, the analogue of \ref{c} does not imply the analogue of \ref{b}.\footnote{Note that $S_8$ must contain the plumbing graph associated to $L(9,8)$, since the integral lattice $Q_{X(9,1)}$ embeds into $(\Z,-\mathrm{Id})$, whereas $L(9,8)$ does not admit a smooth negative-definite filling with $b_2=0$.}

Our next main theorem has three key points. First, we extend~\cite{definite} to the case $k=2$ by explicitly determining $S_2$, where $|S_2|=17$. Second, in this case Donaldson's obstruction still coincides with the smooth topology, in the sense that the analogues of \ref{a}, \ref{b}, and \ref{c} remain equivalent. Lastly, this allows us to classify all negative-definite intersection forms bounded negative-definite fillings for such lens spaces:

\begin{maintheorem}\label{thm:main}
If $L = \#_i L(p_i, q_i)$ is a connected sum of lens spaces, and $\natural_i X(p_i, q_i)$ denotes the corresponding boundary connected sum of the canonical negative-definite plumbings, then the following statements are equivalent:
\begin{enumerate}[label=(\roman*), font=\upshape]
\item\label{it:min_filling} every smooth negative-definite filling $X$ of $L$ satisfies $$b_2(X) \geq b_2(\natural_i X(p_i,q_i))-1;$$

\item\label{it:intersectionform}
every smooth negative-definite filling $X$ of $L$ satisfies either
\[
Q_X \cong Q_{\natural_i X(p_i,q_i)} \oplus \langle -1 \rangle^m 
\qquad \text{or} \qquad 
Q_X \cong Q_{\natural_i \widetilde{X}(p_i,q_i)} \oplus \langle -1 \rangle^m
\]
for some nonnegative integer $m$.
Here, $\natural_i \widetilde{X}(p_i,q_i)$ denotes the 4–manifold obtained when the canonical plumbing graph associated to $L$ contains one of the following configurations as an induced subgraph:
\[
\begin{tikzpicture}[xscale=1.0,yscale=1,baseline={(0,0)}]
    \node at (1-0.1,0.4) {$-4$};
    \node (A1) at (1,0) {$\bullet$};
\end{tikzpicture}
\qquad \qquad \qquad
\begin{tikzpicture}[xscale=1.0,yscale=1,baseline={(0,0)}]
    \node at (1-0.1, .4) {$-3$};
    \node at (2-0.1, .4) {$-3$};
    \node (A_1) at (1, 0) {$\bullet$};
    \node (A_2) at (2, 0) {$\bullet$};
    \path (A_1) edge [-] node [auto] {$\scriptstyle{}$} (A_2);
\end{tikzpicture}
\qquad  \qquad \qquad
\begin{tikzpicture}[xscale=1.0,yscale=1,baseline={(0,0)}]
    \node at (1-0.1, .4) {$-3$};
    \node at (2-0.1, .4) {$-2$};
    \node at (3-0.1, .4) {$-3$};
    \node (A_1) at (1, 0) {$\bullet$};
    \node (A_2) at (2, 0) {$\bullet$};
    \node (A_3) at (3, 0) {$\bullet$};
    \path (A_1) edge [-] node [auto] {$\scriptstyle{}$} (A_2);
    \path (A_2) edge [-] node [auto] {$\scriptstyle{}$} (A_3);
\end{tikzpicture}
\]
In this case, $\natural_i \widetilde{X}(p_i,q_i)$ is obtained by performing a rational blowdown on $\natural_i X(p_i,q_i)$ along an embedded $-4$-sphere obtained by smoothing the transverse intersection points among the spheres corresponding to the vertices of the above configuration in the plumbing graph;\footnote{Blowing down along different embedded $-4$-spheres may yield different integral lattices. See Remark~\ref{rem:5521} for an example.}

% Moreover, among these configurations, at most one may occur in the canonical plumbing graph associated to $L$.

\item\label{it:combinatorial} the canonical plumbing graph associated to $L$ does not contain any of the following configurations as an induced subgraph:
\begin{multicols}{2}
\begin{enumerate}[label=(\alph*),font=\upshape]
\item\label{it:52} $\begin{tikzpicture}[xscale=1.0,yscale=1,baseline={(0,0)}]
    \node at (1-0.1,0.4) {$-5$};
    \node at (2-0.1,0.4) {$-2$};
    \node (A1) at (1,0) {$\bullet$};
    \node (A2) at (2,0) {$\bullet$};
        \path (A1) edge [-] node [auto] {$\scriptstyle{}$} (A2);
  \end{tikzpicture}$
\item\label{it:622} $\begin{tikzpicture}[xscale=1.0,yscale=1,baseline={(0,0)}]
    \node at (1-0.1, .4) {$-6$};
    \node at (2-0.1, .4) {$-2$};
    \node at (3-0.1, .4) {$-2$};
    \node (A_1) at (1, 0) {$\bullet$};
    \node (A_2) at (2, 0) {$\bullet$};
    \node (A_3) at (3, 0) {$\bullet$};
        \path (A_1) edge [-] node [auto] {$\scriptstyle{}$} (A_2);
    \path (A_2) edge [-] node [auto] {$\scriptstyle{}$} (A_3);
  \end{tikzpicture}$
\item\label{it:2-2} $\begin{tikzpicture}[xscale=1.0,yscale=1,baseline={(0,0)}]
    \node at (1-0.1, .4) {$-2$};
    \node at (2-0.1, .4) {$-2$};
    \node (A_1) at (1, 0) {$\bullet$};
    \node (A_2) at (2, 0) {$\bullet$};
  \end{tikzpicture}$
\item\label{it:3-22}  $\begin{tikzpicture}[xscale=1.0,yscale=1,baseline={(0,0)}]
    \node at (1-0.1, .4) {$-3$};
    \node at (2-0.1, .4) {$-2$};
    \node at (3-0.1, .4) {$-2$};
    \node (A_1) at (1, 0) {$\bullet$};
    \node (A_2) at (2, 0) {$\bullet$};
    \node (A_3) at (3, 0) {$\bullet$};
    \path (A_2) edge [-] node [auto] {$\scriptstyle{}$} (A_3);
  \end{tikzpicture}$
\item\label{it:3223}
$\begin{tikzpicture}[xscale=1.0,yscale=1,baseline={(0,0)}]
    \node at (1-0.1, .4) {$-3$};
    \node at (2-0.1, .4) {$-2$};
    \node at (3-0.1, .4) {$-2$};
    \node at (4-0.1, .4) {$-3$};
    \node (A1_1) at (1, 0) {$\bullet$};
    \node (A1_2) at (2, 0) {$\bullet$};
    \node (A1_3) at (3, 0) {$\bullet$};
    \node (A1_4) at (4, 0) {$\bullet$};
    \path (A1_2) edge [-] node [auto] {$\scriptstyle{}$} (A1_3);
    \path (A1_3) edge [-] node [auto] {$\scriptstyle{}$} (A1_4);
    \path (A1_1) edge [-] node [auto] {$\scriptstyle{}$} (A1_2);
  \end{tikzpicture}$
\item\label{it:3532} 
$\begin{tikzpicture}[xscale=1.0,yscale=1,baseline={(0,0)}]
    \node at (1-0.1, .4) {$-3$};
    \node at (2-0.1, .4) {$-5$};
    \node at (3-0.1, .4) {$-3$};
    \node at (4-0.1, .4) {$-2$};
    \node (A1_1) at (1, 0) {$\bullet$};
    \node (A1_2) at (2, 0) {$\bullet$};
    \node (A1_3) at (3, 0) {$\bullet$};
    \node (A1_4) at (4, 0) {$\bullet$};
    \path (A1_2) edge [-] node [auto] {$\scriptstyle{}$} (A1_3);
    \path (A1_3) edge [-] node [auto] {$\scriptstyle{}$} (A1_4);
    \path (A1_1) edge [-] node [auto] {$\scriptstyle{}$} (A1_2);
  \end{tikzpicture}$
\item\label{it:2235} $\begin{tikzpicture}[xscale=1.0,yscale=1,baseline={(0,0)}]
    \node at (1-0.1, .4) {$-2$};
    \node at (2-0.1, .4) {$-2$};
    \node at (3-0.1, .4) {$-3$};
    \node at (4-0.1, .4) {$-5$};
    \node (A1_1) at (1, 0) {$\bullet$};
    \node (A1_2) at (2, 0) {$\bullet$};
    \node (A1_3) at (3, 0) {$\bullet$};
    \node (A1_4) at (4, 0) {$\bullet$};
    \path (A1_2) edge [-] node [auto] {$\scriptstyle{}$} (A1_3);
    \path (A1_3) edge [-] node [auto] {$\scriptstyle{}$} (A1_4);
    \path (A1_1) edge [-] node [auto] {$\scriptstyle{}$} (A1_2);
  \end{tikzpicture}$
\item\label{it:432} $\begin{tikzpicture}[xscale=1.0,yscale=1,baseline={(0,0)}]
    \node at (1-0.1, .4) {$-4$};
    \node at (2-0.1, .4) {$-3$};
    \node at (3-0.1, .4) {$-2$};
    \node (A1_1) at (1, 0) {$\bullet$};
    \node (A1_2) at (2, 0) {$\bullet$};
    \node (A1_3) at (3, 0) {$\bullet$};
    \path (A1_2) edge [-] node [auto] {$\scriptstyle{}$} (A1_3);
    \path (A1_1) edge [-] node [auto] {$\scriptstyle{}$} (A1_2);
  \end{tikzpicture}$
\item\label{it:342} $\begin{tikzpicture}[xscale=1.0,yscale=1,baseline={(0,0)}]
    \node at (1-0.1, .4) {$-3$};
    \node at (2-0.1, .4) {$-4$};
    \node at (3-0.1, .4) {$-2$};
    \node (A1_1) at (1, 0) {$\bullet$};
    \node (A1_2) at (2, 0) {$\bullet$};
    \node (A1_3) at (3, 0) {$\bullet$};
    \path (A1_2) edge [-] node [auto] {$\scriptstyle{}$} (A1_3);
    \path (A1_1) edge [-] node [auto] {$\scriptstyle{}$} (A1_2);
  \end{tikzpicture}$
\item\label{it:4422} $\begin{tikzpicture}[xscale=1.0,yscale=1,baseline={(0,0)}]
    \node at (1-0.1, .4) {$-4$};
    \node at (2-0.1, .4) {$-4$};
    \node at (3-0.1, .4) {$-2$};
    \node at (4-0.1, .4) {$-2$};
    \node (A1_1) at (1, 0) {$\bullet$};
    \node (A1_2) at (2, 0) {$\bullet$};
    \node (A1_3) at (3, 0) {$\bullet$};
    \node (A1_4) at (4, 0) {$\bullet$};
    \path (A1_1) edge [-] node [auto] {$\scriptstyle{}$} (A1_2);
    \path (A1_2) edge [-] node [auto] {$\scriptstyle{}$} (A1_3);
    \path (A1_3) edge [-] node [auto] {$\scriptstyle{}$} (A1_4);
  \end{tikzpicture}$
\item\label{it:34332} $\begin{tikzpicture}[xscale=1.0,yscale=1,baseline={(0,0)}]
    \node at (1-0.1, .4) {$-3$};
    \node at (2-0.1, .4) {$-4$};
    \node at (3-0.1, .4) {$-3$};
    \node at (4-0.1, .4) {$-3$};
		\node at (5-0.1, .4) {$-2$};
    \node (A1_1) at (1, 0) {$\bullet$};
    \node (A1_2) at (2, 0) {$\bullet$};
    \node (A1_3) at (3, 0) {$\bullet$};
    \node (A1_4) at (4, 0) {$\bullet$};
		\node (A1_5) at (5, 0) {$\bullet$};
    \path (A1_1) edge [-] node [auto] {$\scriptstyle{}$} (A1_2);
    \path (A1_2) edge [-] node [auto] {$\scriptstyle{}$} (A1_3);
    \path (A1_3) edge [-] node [auto] {$\scriptstyle{}$} (A1_4);
    \path (A1_4) edge [-] node [auto] {$\scriptstyle{}$} (A1_5);
  \end{tikzpicture}$
\item\label{it:33-2}  $\begin{tikzpicture}[xscale=1.0,yscale=1,baseline={(0,0)}]
    \node at (1-0.1, .4) {$-3$};
    \node at (2-0.1, .4) {$-3$};
    \node at (3-0.1, .4) {$-2$};
    \node (A_1) at (1, 0) {$\bullet$};
    \node (A_2) at (2, 0) {$\bullet$};
    \node (A_3) at (3, 0) {$\bullet$};
    \path (A_1) edge [-] node [auto] {$\scriptstyle{}$} (A_2);
  \end{tikzpicture}$
\item\label{it:323-3} $\begin{tikzpicture}[xscale=1.0,yscale=1,baseline={(0,0)}]
    \node at (1-0.1, .4) {$-3$};
    \node at (2-0.1, .4) {$-2$};
    \node at (3-0.1, .4) {$-3$};
    \node at (4-0.1, .4) {$-3$};
    \node (A1_1) at (1, 0) {$\bullet$};
    \node (A1_2) at (2, 0) {$\bullet$};
    \node (A1_3) at (3, 0) {$\bullet$};
    \node (A1_4) at (4, 0) {$\bullet$};
    \path (A1_1) edge [-] node [auto] {$\scriptstyle{}$} (A1_2);
    \path (A1_2) edge [-] node [auto] {$\scriptstyle{}$} (A1_3);
  \end{tikzpicture}$
\item\label{it:4-4} $\begin{tikzpicture}[xscale=1.0,yscale=1,baseline={(0,0)}]
    \node at (1-0.1, .4) {$-4$};
    \node at (2-0.1, .4) {$-4$};
    \node (A1_1) at (1, 0) {$\bullet$};
    \node (A1_2) at (2, 0) {$\bullet$};
  \end{tikzpicture}$
\item\label{it:4-33}  $\begin{tikzpicture}[xscale=1.0,yscale=1,baseline={(0,0)}]
    \node at (1-0.1, .4) {$-4$};
    \node at (2-0.1, .4) {$-3$};
    \node at (3-0.1, .4) {$-3$};
    \node (A_1) at (1, 0) {$\bullet$};
    \node (A_2) at (2, 0) {$\bullet$};
    \node (A_3) at (3, 0) {$\bullet$};
    \path (A_2) edge [-] node [auto] {$\scriptstyle{}$} (A_3);
  \end{tikzpicture}$
\item\label{it:4-323} $\begin{tikzpicture}[xscale=1.0,yscale=1,baseline={(0,0)}]
    \node at (1-0.1, .4) {$-4$};
    \node at (2-0.1, .4) {$-3$};
    \node at (3-0.1, .4) {$-2$};
    \node at (4-0.1, .4) {$-3$};
    \node (A1_1) at (1, 0) {$\bullet$};
    \node (A1_2) at (2, 0) {$\bullet$};
    \node (A1_3) at (3, 0) {$\bullet$};
    \node (A1_4) at (4, 0) {$\bullet$};
    \path (A1_2) edge [-] node [auto] {$\scriptstyle{}$} (A1_3);
    \path (A1_3) edge [-] node [auto] {$\scriptstyle{}$} (A1_4);
  \end{tikzpicture}$
\item\label{it:33-33} $\begin{tikzpicture}[xscale=1.0,yscale=1,baseline={(0,0)}]
    \node at (1-0.1, .4) {$-3$};
    \node at (2-0.1, .4) {$-3$};
    \node at (3-0.1, .4) {$-3$};
    \node at (4-0.1, .4) {$-3$};
    \node (A1_1) at (1, 0) {$\bullet$};
    \node (A1_2) at (2, 0) {$\bullet$};
    \node (A1_3) at (3, 0) {$\bullet$};
    \node (A1_4) at (4, 0) {$\bullet$};
    \path (A1_1) edge [-] node [auto] {$\scriptstyle{}$} (A1_2);
    \path (A1_3) edge [-] node [auto] {$\scriptstyle{}$} (A1_4);
  \end{tikzpicture}$;
\end{enumerate}
\end{multicols}\vspace{-.2cm}
% \item\label{it:submanifold}
% $\natural_i X(p_i,q_i)$ contains no smoothly embedded submanifold diffeomorphic to 
% $X(9,2)$, $X(16,3)$, $X(16,7)$, $X(64,23)$, $X(2,1)\natural X(2,1)$,  $X(8,3)\natural X(2,1)$, $X(12,5)\natural X(3,1)$, 
% or $X(4,1)\sqcup X(4,1)$.

\item\label{it:submanifold}
the manifold $\natural_i X(p_i,q_i)$ contains no smoothly embedded submanifold 
diffeomorphic to any of the following: 
$X(9,2)$, $X(16,3)$, $X(16,7)$, $X(64,23)$, 
$X(2,1)\sqcup X(2,1)$, $X(8,3)\sqcup X(2,1)$, 
$X(12,5)\sqcup X(3,1)$, or $X(4,1)\sqcup X(4,1)$.

% \begin{multicols}{3}
% \begin{itemize}[label=]
    
% \item $X(9,2)$
% \item $X(16,3)$
% \item $X(16,7)$
% \item $X(64,23)$
% \item $X(2,1)\,\sqcup\,X(2,1)$
% \item $X(8,3)\,\sqcup\,X(2,1)$
% \item $X(12,5)\,\sqcup\,X(3,1)$
% \item $X(4,1)\,\sqcup\,X(4,1)$.

% \end{itemize}
% \end{multicols}

\end{enumerate}
\end{maintheorem}

% $X(3,1)\natural X(3,2)$,

Combined with \cite[Theorem~1.1]{definite}, we immediately obtain a complete combinatorial and geometric classification of connected sums of lens spaces whose $b_2$-minimal negative-definite filling $X$ satisfies $b_2(X)= b_2(\natural_i X(p_i,q_i)) - 1$. More precisely, we have the following:

\begin{cor}\label{cor:main}
Let $L = \#_i L(p_i, q_i)$ be a connected sum of lens spaces, and let $\natural_i X(p_i, q_i)$ denote the corresponding boundary connected sum of the canonical negative-definite plumbings. Suppose moreover that every smooth negative-definite filling $X$ of $L$ satisfies $b_2(X) \geq b_2(\natural_i X(p_i, q_i)) - 1$, then the following statements are equivalent:

\begin{enumerate}[label=(\roman*), font=\upshape]
\item\label{it:cor_min_filling} there exists a smooth negative-definite filling $X$ of $L$ such that 
\[
b_2(X) = b_2(\natural_i X(p_i, q_i)) - 1;
\]

\item there exists a smooth negative-definite filling $X$ of $L$ such that 
\[
Q_X \cong Q_{\natural_i \widetilde{X}(p_i,q_i)} \oplus \langle -1 \rangle^m,
\]
where $\natural_i \widetilde{X}(p_i,q_i)$ denotes the 4–manifold obtained as in Theorem~\ref{thm:main}\ref{it:intersectionform};

\item the canonical plumbing graph associated to $L$ contains at least one of the following configurations as an induced subgraph:
\[ \begin{tikzpicture}[xscale=1.0,yscale=1,baseline={(0,0)}]
    \node at (1-0.1,0.4) {$-4$};
    \node (A1) at (1,0) {$\bullet$};
\end{tikzpicture}\qquad\qquad \qquad  \begin{tikzpicture}[xscale=1.0,yscale=1,baseline={(0,0)}]
    \node at (1-0.1, .4) {$-3$};
    \node at (2-0.1, .4) {$-3$};
    \node (A_1) at (1, 0) {$\bullet$};
    \node (A_2) at (2, 0) {$\bullet$};
    \path (A_1) edge [-] node [auto] {$\scriptstyle{}$} (A_2);
\end{tikzpicture}\qquad \qquad  \qquad \begin{tikzpicture}[xscale=1.0,yscale=1,baseline={(0,0)}]
    \node at (1-0.1, .4) {$-3$};
    \node at (2-0.1, .4) {$-2$};
    \node at (3-0.1, .4) {$-3$};
    \node (A_1) at (1, 0) {$\bullet$};
    \node (A_2) at (2, 0) {$\bullet$};
    \node (A_3) at (3, 0) {$\bullet$};
    \path (A_1) edge [-] node [auto] {$\scriptstyle{}$} (A_2);
    \path (A_2) edge [-] node [auto] {$\scriptstyle{}$} (A_3);
\end{tikzpicture};\]

\item\label{it:submanifold_L41}
the manifold $\natural_i X(p_i,q_i)$ contains a smoothly embedded submanifold diffeomorphic to $X(4,1)$. \qed

\end{enumerate}

\end{cor}

\begin{rem}\label{rem:5521}
Consider $L=L(55,21)$, whose canonical plumbing graph is
\[
\begin{tikzpicture}[xscale=1.0,yscale=1,baseline={(0,0)}]
    \node at (1-0.1, .4) {$-3$};
    \node at (2-0.1, .4) {$-3$};
    \node at (3-0.1, .4) {$-3$};
    \node at (4-0.1, .4) {$-3$};
    \node (A1_1) at (1, 0) {$\bullet$};
    \node (A1_2) at (2, 0) {$\bullet$};
    \node (A1_3) at (3, 0) {$\bullet$};
    \node (A1_4) at (4, 0) {$\bullet$};
    \path (A1_1) edge [-] node [auto] {$\scriptstyle{}$} (A1_2);
    \path (A1_2) edge [-] node [auto] {$\scriptstyle{}$} (A1_3);
    \path (A1_3) edge [-] node [auto] {$\scriptstyle{}$} (A1_4);
\end{tikzpicture}.
\]
Then, $L$ satisfies Theorem~\ref{thm:main}\ref{it:combinatorial}. Since its plumbing graph contains
\[
\begin{tikzpicture}[xscale=1.0,yscale=1,baseline={(0,0)}]
    \node at (1-0.1, .4) {$-3$};
    \node at (2-0.1, .4) {$-3$};
    \node (A_1) at (1, 0) {$\bullet$};
    \node (A_2) at (2, 0) {$\bullet$};
    \path (A_1) edge [-] node [auto] {$\scriptstyle{}$} (A_2);
\end{tikzpicture}
\]
as an induced subgraph, by Corollary~\ref{cor:main}, the $b_2$-minimal negative-definite filling $X$ of $L$ satisfies
$b_2(X)=b_2(X(55,21))-1=3$. Moreover, such an $X$ can be obtained by performing a rational blowdown on the embedded $-4$-sphere obtained by smoothing the transverse intersection point between the two spheres corresponding to any pair of adjacent vertices. Note that blowing down on the embedded $-4$-sphere obtained from the two vertices in the middle produces a manifold whose intersection form is not equivalent to the intersection form of the manifold produced by blowing down on the embedded $-4$-sphere obtained from the two vertices on the left side.\footnote{See Figure~\ref{fig:c3} for an explicit calculation of the intersection forms. One of the resulting intersection forms has an element with self-pairing $-3$ while the other one does not, and therefore they cannot be equivalent.}
\end{rem}

% Due to symmetry, blowing down on the embedded $-4$-sphere obtained from the two vertices on the left side produces a manifold with an intersection form that is equivalent to the intersection form of the manifold obtained by blowing down on the embedded $-4$-sphere obtained from the two vertices on the right side.

% \begin{cor}\label{cor:symplecticunique}
% Let $L(p, q)$ be a lens space, and let $X(p, q)$ denote its canonical negative-definite plumbing. Suppose moreover that every smooth negative-definite filling $X$ of $L$ satisfies $b_2(X) \geq b_2(X(p, q)) - 1$. If $\xi$ is a tight contact structure on $L(p,q)$, then any minimal symplectic filling of $(L(p,q),\xi)$ is diffeomorphic to either $X(p,q)$ or $\widetilde{X}(p,q)$ from Theorem~\ref{thm:main}\ref{it:intersectionform}.
% \end{cor}
Lastly, as in \cite[Section~1.2]{definite}, we also have applications toward symplectic fillings of lens spaces and the smoothings of cyclic quotient singularities. The connection comes from the fact that every symplectic filling of a lens space is negative-definite~\cite{Schonenberger, Etnyre:2004-1}. Recall that the symplectic fillings of the standard (i.e., universally tight) contact structures on lens spaces were classified by Lisca~\cite{Lisca:2008-1}, and that there is a bijection between minimal symplectic fillings (i.e., those containing no symplectically embedded sphere of self-intersection $-1$)\footnote{Such a filling is not necessarily $b_2$-minimal across all symplectic fillings.} and the smoothings of cyclic quotient singularities~\cite{NemethiPopescu}. Combined with our main results, this leads to the following corollaries:

The first corollary states that, for every lens space $L$ satisfying any of the conditions in Theorem~\ref{thm:main}, the $b_2$-minimal negative-definite filling $X$ can be realized as a symplectic filling. Moreover, except for the three exceptional cases $L \in \{L(4,1), L(8,3), L(12,5)\}$, $X$ is simply connected. This is not the case in general; for instance, although $L(4,3)$ bounds a smooth rational homology ball, it has a unique minimal symplectic filling $W$ with $b_2(W)=3$:

\begin{cor}\label{cor:symplectic}
Let $L = L(p,q)$ be a lens space, and let $X(p,q)$ denote the corresponding canonical negative-definite plumbing. Suppose that every smooth negative-definite filling $X$ of $L$ satisfies $b_2(X) \ge b_2(X(p,q)) - 1$. If $\xi_{\mathrm{st}}$ is the standard contact structure on $L$, then there exists a minimal symplectic filling $W$ of $(L,\xi_{\mathrm{st}})$ such that
\[
b_2(W)=\min_X b_2(X),
\]
where the minimum is taken across all smooth negative-definite fillings $X$ of $L$.

Furthermore, every such $W$ is simply connected unless $L \in \{L(4,1), L(8,3), L(12,5)\}$, in which case it has fundamental group $\Z/2$.
\end{cor}

% \begin{cor}\label{cor:symplecticunique}
% Let $L = L(p,q)$ be a lens space as in Theorem~\ref{thm:main}\ref{it:combinatorial}, and let $\xi_{\mathrm{st}}$ denote the standard contact structure on $L$. Then the number of 
% minimal symplectic fillings of $(L,\xi_{\mathrm{st}})$, up to orientation-preserving diffeomorphism, is at most one plus the number of bad vertices of the dual plumbing graph of~$L$. Moreover, an analogous statement holds for the number of Milnor fibers arising from the irreducible components of the reduced miniversal base space of the cyclic quotient singularity associated with $L$.
% \end{cor}

Lisca proves in \cite[Corollary~1.2(b)]{Lisca:2008-1} that if a lens space $L=L(p,q)$ with the standard contact structure $\xi_{\mathrm{st}}$ has canonical plumbing graph whose vertex weights are $-a_i$ with $a_i\ge 5$ for all $i$, then $(L,\xi_{\mathrm{st}})$ admits a unique symplectic filling up to orientation-preserving diffeomorphism. The same uniqueness conclusion holds for a larger family of lens spaces, namely those satisfying Property $X_1$, by \cite[Corollary~1.7]{definite} (a direct consequence of \cite[Theorem~1.1]{definite} and Lisca's classification~\cite[Theorem~1.1]{Lisca:2008-1}). In what follows, we extend this perspective to lens spaces with Property $X_2$, obtaining a complete count of symplectic fillings up to orientation-preserving diffeomorphism:

\begin{cor}\label{cor:symplecticunique}
Let $L = L(p,q)$ be a lens space satisfying one of the items in Theorem~\ref{thm:main}, and let $\xi_{\mathrm{st}}$ denote the standard contact structure on $L$. Let $n(L)$ be the number of induced subgraphs in the canonical plumbing graph associated to $L$ that are isomorphic to one of the following configurations:
\[
\begin{tikzpicture}[xscale=1.0,yscale=1,baseline={(0,0)}]
    \node at (1-0.1,0.4) {$-4$};
    \node (A1) at (1,0) {$\bullet$};
\end{tikzpicture}
\qquad \qquad \qquad
\begin{tikzpicture}[xscale=1.0,yscale=1,baseline={(0,0)}]
    \node at (1-0.1, .4) {$-3$};
    \node at (2-0.1, .4) {$-3$};
    \node (A_1) at (1, 0) {$\bullet$};
    \node (A_2) at (2, 0) {$\bullet$};
    \path (A_1) edge [-] node [auto] {$\scriptstyle{}$} (A_2);
\end{tikzpicture}
\qquad  \qquad \qquad
\begin{tikzpicture}[xscale=1.0,yscale=1,baseline={(0,0)}]
    \node at (1-0.1, .4) {$-3$};
    \node at (2-0.1, .4) {$-2$};
    \node at (3-0.1, .4) {$-3$};
    \node (A_1) at (1, 0) {$\bullet$};
    \node (A_2) at (2, 0) {$\bullet$};
    \node (A_3) at (3, 0) {$\bullet$};
    \path (A_1) edge [-] node [auto] {$\scriptstyle{}$} (A_2);
    \path (A_2) edge [-] node [auto] {$\scriptstyle{}$} (A_3);
\end{tikzpicture}.
\]
Then the number of minimal symplectic fillings of $(L,\xi_{\mathrm{st}})$, up to orientation-preserving diffeomorphism, is either $n(L)+1$ or $n(L)$. Moreover, this number is $n(L)$ if and only if $q\equiv 1 \pmod p$ and the canonical plumbing graph associated to $L$ contains one of the following as an induced subgraph:
\[
\begin{tikzpicture}[xscale=1.0,yscale=1,baseline={(0,0)}]
    \node at (1-0.1, .4) {$-4$};
    \node at (2-0.1, .4) {$-4$};
    \node (A_1) at (1, 0) {$\bullet$};
    \node (A_2) at (2, 0) {$\bullet$};
    \path (A_1) edge [-] node [auto] {$\scriptstyle{}$} (A_2);
\end{tikzpicture}
\qquad  \qquad \qquad
\begin{tikzpicture}[xscale=1.0,yscale=1,baseline={(0,0)}]
    \node at (1-0.1, .4) {$-3$};
    \node at (2-0.1, .4) {$-3$};
    \node at (3-0.1, .4) {$-3$};
    \node (A_1) at (1, 0) {$\bullet$};
    \node (A_2) at (2, 0) {$\bullet$};
    \node (A_3) at (3, 0) {$\bullet$};
    \path (A_1) edge [-] node [auto] {$\scriptstyle{}$} (A_2);
    \path (A_2) edge [-] node [auto] {$\scriptstyle{}$} (A_3);
\end{tikzpicture}.
\]
Furthermore, an analogous statement holds for the number of Milnor fibers arising from the irreducible components of the reduced miniversal base space of the cyclic quotient singularity corresponding to $L$.
\end{cor}

\begin{rem}
When counting $n(L)$, the induced subgraphs are not required to be disjoint. For example, $n(L(55,21))=3$. Since $21^2\equiv 1 \pmod{55}$ and the canonical plumbing graph associated to $L(55,21)$ contains
\[
\begin{tikzpicture}[xscale=1.0,yscale=1,baseline={(0,0)}]
    \node at (1-0.1, .4) {$-3$};
    \node at (2-0.1, .4) {$-3$};
    \node at (3-0.1, .4) {$-3$};
    \node (A_1) at (1, 0) {$\bullet$};
    \node (A_2) at (2, 0) {$\bullet$};
    \node (A_3) at (3, 0) {$\bullet$};
    \path (A_1) edge [-] node [auto] {$\scriptstyle{}$} (A_2);
    \path (A_2) edge [-] node [auto] {$\scriptstyle{}$} (A_3);
\end{tikzpicture}
\]
as an induced subgraph, it follows that $(L(55,21),\xi_{\mathrm{st}})$ admits $n(L)=3$ distinct minimal symplectic fillings (equivalently, Milnor fibers) up to orientation-preserving diffeomorphism.
\end{rem}

\subsection*{Proof outline for Theorem~\ref{thm:all k}}

To define $S_k$, we start with the set of all configurations that do not satisfy Property $X_k$, and then remove all configurations that contain another configuration in the set as an induced subgraph. The resulting set being finite is implied by Theorem~\ref{thm:all k}(\ref{mainbodyA2})(\ref{mainbodyA3}).

To prove Theorem~\ref{thm:all k}(\ref{mainbodyA1}), we show that if a configuration satisfies Property $X_k$, then all its induced subgraphs also satisfy Property $X_k$.

To prove Theorem~\ref{thm:all k}(\ref{mainbodyA2}), we show that if a configuration consisting of a vertex with a very large weight does not satisfy Property $X_k$, then the induced subgraph formed by removing this vertex also does not satisfy Property $X_k$. Intuitively, vertices of large weight do not contribute to violating Property $X_k$.

To prove Theorem~\ref{thm:all k}(\ref{mainbodyA3}), we first observe that not allowing certain configurations as induced subgraphs in $S_k$ imposes upper bounds on the numbers of vertices of weights $-2$, $-3$, and $-4$ in each  configuration in $S_k$. We then show that if a configuration violating Property $X_k$ has many vertices that are not adjacent to any vertex with weight $-2$, $-3$, or $-4$, then at least one of the induced subgraphs formed by removing one of those vertices violates Property $X_k$.

The detailed proof is given in Section~\ref{sec:forbidden_config}.

\subsection*{Proof outline for Theorem~\ref{thm:main}}

Similar to \cite{definite}, Theorem~\ref{thm:main} is proven via the chain of implications 
\[\mathrm{\ref{it:intersectionform}} \Rightarrow \mathrm{\ref{it:min_filling}} \Rightarrow \mathrm{\ref{it:submanifold}} \Rightarrow \mathrm{\ref{it:combinatorial}} \Rightarrow \mathrm{\ref{it:intersectionform}}.\]
The implication \ref{it:intersectionform} $\Rightarrow$ \ref{it:min_filling} holds because $Q_X$ is defined on the second homology group. The implication \ref{it:min_filling} $\Rightarrow$ \ref{it:submanifold} is obtained by noting that for every manifold listed in \ref{it:submanifold} has second Betti number at least 2, while its boundary is a lens space or a connected sum of lens spaces that bounds a smooth rational homology ball. Hence, if there is such an embedded submanifold, we can cut it out and replace it with a rational homology ball, reducing the second Betti number by at least 2. To obtain \ref{it:submanifold} $\Rightarrow$ \ref{it:combinatorial}, from each configuration listed in \ref{it:combinatorial}, we show that we can use the corresponding spherical generators to construct an embedded copy of a manifold listed under \ref{it:submanifold}.

The implication \ref{it:combinatorial} $\Rightarrow$ \ref{it:intersectionform} is the difficult step. The proof of this step spans through Section~\ref{Section:preliminaries}, \ref{sec:definitions}, and \ref{sec:lattice_analysis}. By \cite[Lemma~2.4]{definite} and Donaldson's diagonalization theorem, all possibilities of $Q_X$ arise as the orthogonal complement of an embedding of $Q_{\natural_i X(p_i,p_i-q_i)}$ into a diagonal lattice. Hence, it suffices to classify all such embeddings when $L$ satisfies \ref{it:combinatorial}, and then compare the corresponding orthogonal complements with the lattices listed in \ref{it:intersectionform}.

\subsection*{Structure of the article}
In Section~\ref{sec:topology}, we prove \ref{it:min_filling} $\Rightarrow$ \ref{it:submanifold} and \ref{it:submanifold} $\Rightarrow$ \ref{it:combinatorial} of Theorem~\ref{thm:main}. In Section~\ref{Section:preliminaries}, we translate \ref{it:combinatorial} of Theorem~\ref{thm:main} into conditions about the plumbing diagram of $\natural_i X(p_i,p_i-q_i)$. In Section~\ref{sec:definitions}, we make some technical definitions to help with the lattice embedding analysis that occurs in Section~\ref{sec:lattice_analysis}. At the end of Section~\ref{sec:lattice_analysis}, we compute the intersection form $Q_{\natural_i \widetilde{X}(p_i,q_i)}$ stated in Theorem~\ref{thm:main}\ref{it:intersectionform} and compare it with the results from the lattice embedding analysis, thus completing the proof of Theorem~\ref{thm:main}.

In Section~\ref{sec:forbidden_config}, we prove Theorem~\ref{thm:all k}. The proof of Theorem~\ref{thm:all k} is independent of the previous sections except for a few definitions that can also be found in \cite{definite}. However, we put it after the proof of Theorem~\ref{thm:main} because it may be easier for some readers to go through those technical arguments in Section~\ref{sec:forbidden_config} after getting used to them when reading the previous sections. In Section~\ref{sec:symplectic}, we prove Corollary~\ref{cor:symplectic} and Corollary~\ref{cor:symplecticunique}.

\subsection*{Acknowledgements} 

The authors thank Marco Golla, Duncan McCoy, and Brendan Owens for helpful comments on an earlier draft of this article. Both authors are partially supported by the Samsung Science and Technology Foundation (SSTF-BA2102-02) and by the NRF grant RS-2025-00542968.

% The authors would like to thank Brendan Owens for pointing out that the manifold $Z$ mentioned before Definition~\ref{def:Xk} is diffeomorphic to a connected sum of copies of the complex projective plane, and suggesting an edit on Remark~\ref{rem:5521} to make it less confusing for the readers.

\section{Rational homology balls and rational blowdowns}\label{sec:topology}

In this section, we first consider connected sums of lens spaces arising as boundaries in Theorem~\ref{thm:main}\ref{it:submanifold}. Using the fact that each bounds a rational homology ball, we show that Theorem~\ref{thm:main}\ref{it:min_filling} implies Theorem~\ref{thm:main}\ref{it:submanifold}. We then show that each canonical negative-definite plumbing corresponding to a plumbing graph in Theorem~\ref{thm:main}\ref{it:combinatorial} contains a smoothly embedded submanifold diffeomorphic to one of the manifolds listed in Theorem~\ref{thm:main}\ref{it:submanifold}. This proves that Theorem~\ref{thm:main}\ref{it:submanifold} implies Theorem~\ref{thm:main}\ref{it:combinatorial}.

\begin{lem}\label{thm:1implies4}
Theorem \ref{thm:main}\ref{it:min_filling} implies Theorem \ref{thm:main}\ref{it:submanifold}.
\end{lem}

\begin{proof}
Let $X = \natural_i X(p_i,q_i)$ and suppose that $X$ contains an embedded submanifold
diffeomorphic to one of the manifolds listed in Theorem~\ref{thm:main}\ref{it:submanifold}.
We claim that the boundary
\[
L = \#_i L(p_i,q_i)
\]
then bounds a negative-definite manifold $X'$ with $b_2(X') \le b_2(X) - 2$.

This follows from the cut-and-paste argument in \cite[Lemma~2.2]{definite}.  Indeed, every
manifold in Theorem~\ref{thm:main}\ref{it:submanifold} is negative-definite with $b_2 \ge 2$,
and its boundary is a rational homology sphere that bounds a rational homology ball.  For
convenience, we record the relevant instances of Lisca's classification:

\begin{itemize}
\item $L(9,2)$ bounds a rational homology ball
  \cite[family $1_-$, $m=3$, $k=1$]{Lisca:2007-1}.
\item $L(16,3)$ bounds a rational homology ball
  \cite[family $1_-$, $m=4$, $k=1$]{Lisca:2007-1}.
\item $L(16,7)$ bounds a rational homology ball
  \cite[family $f(3_-)$, $m=4$, $d=3$]{Lisca:2007-1}.
\item $L(64,23)$ bounds a rational homology ball
  \cite[family $1_-$, $m=8$, $k=3$]{Lisca:2007-1}.
\item $L(2,1)\#L(2,1)$ bounds a rational homology ball
  \cite[family~2]{Lisca:2007-2}.
\item $L(3,1)\#L(3,2)$ bounds a rational homology ball
  \cite[family~2]{Lisca:2007-2}.
\item $L(8,3)\#L(2,1)$ bounds a rational homology ball (the reverse orientation of
  \cite[family~4, $m=2$, $n=2$, $k=1$]{Lisca:2007-2}).
\item $L(12,5)\#L(3,1)$ bounds a rational homology ball (the reverse orientation of
  \cite[family~4, $m=2$, $n=3$, $k=1$]{Lisca:2007-2}).
\item $L(4,1)\sqcup L(4,1)$ bounds a rational homology ball since $L(4,1)$ does
  \cite[family $1_-$, $m=2$, $k=1$]{Lisca:2007-1}.
\end{itemize}
In each case, by cutting out the given submanifold and gluing in the corresponding rational homology ball, we obtain a negative-definite filling $X'$ of $L$ satisfying $b_2(X') \leq b_2(X) - 2$, as claimed.
\end{proof}

\begin{lem}\label{lem:4implies3}
Theorem \ref{thm:main}\ref{it:submanifold} implies Theorem \ref{thm:main}\ref{it:combinatorial}.
\end{lem}

\begin{proof}
We examine each configuration in Theorem \ref{thm:main}\ref{it:combinatorial} and show that, whenever such a configuration occurs in $X$, it contains a smoothly embedded submanifold diffeomorphic to one of the manifolds listed in Theorem \ref{thm:main}\ref{it:submanifold}.

Configuration \ref{it:52} is $X(9,2)$. Configuration \ref{it:622} is $X(16,3)$. Configuration \ref{it:2-2} is $X(2,1)\natural X(2,1)$. As described in the proof of \cite[Lemma~2.3]{definite}, in configuration \ref{it:3-22}, tubing together the $-3$-sphere and a $-2$-sphere yields an $X(9,2)$. Configuration \ref{it:3223} is $X(16,7)$. Configuration \ref{it:3532} is $X(64,23)$. Again by \cite[Lemma~2.3]{definite}, in configuration \ref{it:2235}, resolving the intersection point between the $-3$-sphere and the $-5$-sphere yields an $X(16,3)$. In configurations \ref{it:432} and \ref{it:342}, resolving the intersection point between the $-4$-sphere and the $-3$-sphere yields an $X(9,2)$. In configuration \ref{it:4422}, resolving the intersection point between the two $-4$-spheres yields an $X(16,3)$. In configuration \ref{it:34332}, resolving the intersection point between the $-4$-sphere and the middle $-3$-sphere yields an $X(64,23)$. Configuration \ref{it:33-2} is $X(8,3)\natural X(2,1)$, and configuration \ref{it:323-3} is $X(12,5)\natural X(3,1)$.

As mentioned in the proof of \cite[Lemma~2.3]{definite}, any occurrence of one of the following subgraphs
\[
\begin{tikzpicture}[xscale=1.0,yscale=1,baseline={(0,0)}]
    \node at (1-0.1,0.4) {$-4$};
    \node (A1) at (1,0) {$\bullet$};
\end{tikzpicture}
\qquad \qquad \qquad
\begin{tikzpicture}[xscale=1.0,yscale=1,baseline={(0,0)}]
    \node at (1-0.1, .4) {$-3$};
    \node at (2-0.1, .4) {$-3$};
    \node (A_1) at (1, 0) {$\bullet$};
    \node (A_2) at (2, 0) {$\bullet$};
    \path (A_1) edge [-] node [auto] {$\scriptstyle{}$} (A_2);
\end{tikzpicture}
\qquad  \qquad \qquad
\begin{tikzpicture}[xscale=1.0,yscale=1,baseline={(0,0)}]
    \node at (1-0.1, .4) {$-3$};
    \node at (2-0.1, .4) {$-2$};
    \node at (3-0.1, .4) {$-3$};
    \node (A_1) at (1, 0) {$\bullet$};
    \node (A_2) at (2, 0) {$\bullet$};
    \node (A_3) at (3, 0) {$\bullet$};
    \path (A_1) edge [-] node [auto] {$\scriptstyle{}$} (A_2);
    \path (A_2) edge [-] node [auto] {$\scriptstyle{}$} (A_3);
\end{tikzpicture}
\]
contains an $X(4,1)$. Since configurations \ref{it:4-4}, \ref{it:4-33}, \ref{it:4-323}, and \ref{it:33-33} each contain two disjoint copies of one of these subgraphs, each of them contains an $X(4,1)\sqcup X(4,1)$.

This verifies that every configuration in Theorem \ref{thm:main}\ref{it:combinatorial} gives rise to an embedded submanifold diffeomorphic to a manifold in the list of Theorem \ref{thm:main}\ref{it:submanifold}, and the lemma follows.
\end{proof}

\begin{rem}
The purpose of the tubing and intersection-point resolving arguments is to reduce the number of manifolds appearing in Theorem~\ref{thm:main}\ref{it:submanifold}, thereby making our main theorem look more streamlined. A similar effect could be achieved without these arguments by separately showing, for each configuration in Theorem~\ref{thm:main}\ref{it:combinatorial}, that whenever such a configuration occurs, one can reduce $b_2$ by at least $2$ by choosing suitable entries from Lisca's lists~\cite{Lisca:2007-1,Lisca:2007-2} or from the Berge lens space lists~\cite[\S 6.2]{Rasmussen:2007-1}, \cite[\S 1.2]{Greene:2013-1}.
\end{rem}

\section{Working Conditions for dual plumbing graphs}\label{Section:preliminaries}

Recall that a connected sum of lens spaces $L=\Lsum$ arises as the boundary of the plumbings $\Xsum$ and $-\natural_i X(p_i,p_i-q_i)$. If $\Xsum$ is obtained by plumbing along a canonical plumbing graph $P$ (coming from the negative continued fraction expansion) and $-\natural_i X(p_i,p_i-q_i)$ is obtained by plumbing along a canonical plumbing graph $P^*$, then we call $P^*$ the \emph{dual} of $P$. Following \cite[Section~3.1]{definite}, we describe the relationship between $P$ and $P^*$ explicitly, introduce the Working Conditions in our setting, and prove that if $L$ satisfies Theorem~\ref{thm:main}\ref{it:combinatorial}, then the dual graph $P^*$ satisfies the Working Conditions.

We first recall some notation and a lemma from \cite{definite} that we will use.

% We work with negative continued fractions and use the notation
% \[
% [a_1,\dots,a_n]^- = a_1 - \cfrac{1}{a_2 - \cfrac{1}{\ddots - \cfrac{1}{a_n}}}.
% \]
% Moreover, we write $[a]^m$ for the length-$m$ chain $a,\ldots,a$, with $[a]^0$ understood to be the empty chain.

\begin{defn}\label{def:adjustedweight}
Let $P$ be a disjoint union of linear plumbing graphs. We define the \emph{adjusted weight} of a vertex $v \in P$ to be
\[
w'(v) := w(v) - d(v),
\]
where $d(v)$ is the degree of $v$.
\end{defn}

\begin{notation}
When depicting a vertex in a plumbing graph, we use a solid dot to indicate that the label is the weight
\[
\begin{tikzpicture}[xscale=1.0,yscale=1,baseline={(0,0)}]
    \node at (1, .4) {$w(v)$};
    \node (A_1) at (1, 0) {$\bullet$};
\end{tikzpicture}
\]
and a hollow dot to indicate that the label is the adjusted weight
\[
\begin{tikzpicture}[xscale=1.0,yscale=1,baseline={(0,0)}]
    \node at (1, .4) {$w'(v)$};
    \node (A_1) at (1, 0) {$\circ$};
\end{tikzpicture}
\]
\end{notation} 
\noindent For example, with this notational convention
\[
\begin{tikzpicture}[xscale=1.0,yscale=1,baseline={(0,0)}]
    \node at (1, .4) {$1$};
    \node at (2, .4) {$1$};
    \node at (3, .4) {$3$};
    \node (A1_1) at (1, 0) {$\circ$};
    \node (A1_2) at (2, 0) {$\circ$};
    \node (A1_3) at (3, 0) {$\circ$};
    \path (A1_1) edge [-] node [auto] {$\scriptstyle{}$} (A1_2);
    \path (A1_2) edge [-] node [auto] {$\scriptstyle{}$} (A1_3);
\end{tikzpicture}
\qquad\qquad\text{ and }\qquad\qquad
\begin{tikzpicture}[xscale=1.0,yscale=1,baseline={(0,0)}]
    \node at (1, .4) {$2$};
    \node at (2, .4) {$3$};
    \node at (3, .4) {$4$};
    \node (A1_1) at (1, 0) {$\bullet$};
    \node (A1_2) at (2, 0) {$\bullet$};
    \node (A1_3) at (3, 0) {$\bullet$};
    \path (A1_1) edge [-] node [auto] {$\scriptstyle{}$} (A1_2);
    \path (A1_2) edge [-] node [auto] {$\scriptstyle{}$} (A1_3);
\end{tikzpicture}
\]
both depict the same plumbing graph.

For the lemma below, we write $[a]^m$ for the length-$m$ chain $a,\ldots,a$, with $[a]^0$ understood to be the empty chain.

\begin{lem}[{\cite[Lemma 3.5]{definite}\label{lem:duality}}]
Suppose that a plumbing $P$ contains a linear component of the form
$$\begin{tikzpicture}[xscale=1.0,yscale=1,baseline={(0,0)}]
    \node at (1-0.1, .4) {$-c_1$};
    \node at (2-0.1, .4) {$-c_2$};
        \node at (3-0.1, .4) {$-c_3$};
    \node at (5-0.1, .4) {$-c_{n}$};
    \node (A1_1) at (1, 0) {$\bullet$};
    \node (A1_2) at (2, 0) {$\bullet$};
    \node (A1_3) at (3, 0) {$\bullet$};
    \node (A1_4) at (4, 0) {$\cdots$};
    \node (A1_5) at (5, 0) {$\bullet$};
    \path (A1_2) edge [-] node [auto] {$\scriptstyle{}$} (A1_3);
    \path (A1_3) edge [-] node [auto] {$\scriptstyle{}$} (A1_4);
        \path (A1_4) edge [-] node [auto] {$\scriptstyle{}$} (A1_5);
    \path (A1_1) edge [-] node [auto] {$\scriptstyle{}$} (A1_2);
  \end{tikzpicture},$$
  where the $c_i$ are integers satisfying
  \[
  (c_1, \dots, c_n)= ([2]^{a_0}, b_1, [2]^{a_1}, \dots, b_k, [2]^{a_k}),
  \]
for $a_i$ and $b_i$ are integers satisfying $a_i\geq 0$ and $b_i\geq 3$. Then the corresponding component of $P^*$ is of the form
  $$\begin{tikzpicture}[xscale=1.0,yscale=1,baseline={(0,0)}]
    \node at (1-0, .4) {$d_1$};
    \node at (2-0, .4) {$d_2$};
        \node at (3-0, .4) {$d_3$};
    \node at (5-0, .4) {$d_{n'}$};
    \node (A1_1) at (1, 0) {$\circ$};
    \node (A1_2) at (2, 0) {$\circ$};
    \node (A1_3) at (3, 0) {$\circ$};
    \node (A1_4) at (4, 0) {$\cdots$};
    \node (A1_5) at (5, 0) {$\circ$};
    \path (A1_2) edge [-] node [auto] {$\scriptstyle{}$} (A1_3);
    \path (A1_3) edge [-] node [auto] {$\scriptstyle{}$} (A1_4);
        \path (A1_4) edge [-] node [auto] {$\scriptstyle{}$} (A1_5);
    \path (A1_1) edge [-] node [auto] {$\scriptstyle{}$} (A1_2);
  \end{tikzpicture},$$
  where
  \[
 (d_1, \dots, d_{n'}) =(a_0+1, [0]^{b_1-3}, a_1+1, \dots, [0]^{b_k-3}, a_k+1).
  \]
\end{lem}

As in \cite[Section~3.1]{Aceto-McCoy-Park:2020-1}, this allows us to recast the conditions of Theorem~\ref{thm:main}\ref{it:combinatorial} in terms of properties of the dual. More precisely, the conditions of Theorem~\ref{thm:main}\ref{it:combinatorial} translate into the following list of conditions on the dual, which we will refer to as the \emph{Working Conditions}.

\begin{defn}\label{def:bad}
A vertex is said to be \emph{bad} if it is adjacent to distinct vertices $u$ and $v$ with $w'(u) > 0$ and $w'(v) > 0$.
\end{defn}

\begin{rem}\label{rem:badvertex}
Each bad vertex with adjusted weight $0$ or $1$ or $2$ in the dual plumbing graph $P^*$ corresponds to a copy of
\[
\begin{tikzpicture}[xscale=1.0,yscale=1,baseline={(0,0)}]
    \node at (1-0.1,0.4) {$-4$};
    \node (A1) at (1,0) {$\bullet$};
\end{tikzpicture}
\qquad \qquad \qquad
\begin{tikzpicture}[xscale=1.0,yscale=1,baseline={(0,0)}]
    \node at (1-0.1, .4) {$-3$};
    \node at (2-0.1, .4) {$-3$};
    \node (A_1) at (1, 0) {$\bullet$};
    \node (A_2) at (2, 0) {$\bullet$};
    \path (A_1) edge [-] node [auto] {$\scriptstyle{}$} (A_2);
\end{tikzpicture}
\qquad  \qquad \qquad
\begin{tikzpicture}[xscale=1.0,yscale=1,baseline={(0,0)}]
    \node at (1-0.1, .4) {$-3$};
    \node at (2-0.1, .4) {$-2$};
    \node at (3-0.1, .4) {$-3$};
    \node (A_1) at (1, 0) {$\bullet$};
    \node (A_2) at (2, 0) {$\bullet$};
    \node (A_3) at (3, 0) {$\bullet$};
    \path (A_1) edge [-] node [auto] {$\scriptstyle{}$} (A_2);
    \path (A_2) edge [-] node [auto] {$\scriptstyle{}$} (A_3);
\end{tikzpicture}
\]
in the original graph. A bad vertex in $P^*$ with adjusted weight $0$ arises from a vertex from the original graph with weight $-4$. A bad vertex in $P^*$ with adjusted weight $1$ arises from two adjacent vertices in the original graph with weight $-3$. A bad vertex in $P^*$ with adjusted weight $2$ arises from a copy of $[-3,-2,-3]$ in the original graph.
\end{rem}

\begin{conditions} The Working Conditions for $\Gamma$, a disjoint union of weighted linear graphs, are the following:
\begin{enumerate}[label=\Roman*]
\item\label{it:positivity} every vertex of $\Gamma$ satisfies $w(v)\geq 2$;
\item\label{it:largeweight} every vertex of $\Gamma$ satisfies $w'(v)\leq 3$ and $\Gamma$ contains at most one vertex $v$ with $w'(v) >1$;
\item\label{it:bad1} if $w'(v)>1$, and $x$ is a bad vertex with $w'(x)=1$, then $x$ and $v$ must be adjacent to each other;
\item\label{it:weight3condition} if $\Gamma$ contains a vertex $v$ with $w'(v)=3$, then $\Gamma$ does not contain a subgraph of the form $$\begin{tikzpicture}[xscale=1.0,yscale=1,baseline={(0,0)}]
    \node at (1,0.4) {$1$};
    \node at (2,0.4) {$1$};
    \node (A1) at (1,0) {$\circ$};
    \node (A2) at (2,0) {$\circ$};
    \path (A1) edge [-] node [auto] {$\scriptstyle{}$} (A2);
  \end{tikzpicture};$$
\item\label{it:bad2} if $u,v$ are adjacent to each other, $w'(u)=w'(v)=1$, and $x$ is a bad vertex with $w'(x)=2$, then $x$ must be adjacent to either $u$ or $v$;
\item\label{it:forbidden_configs} the graph $\Gamma$ does not contain any subgraphs of the following forms:
\begin{multicols}{2}
\begin{enumerate}[label=(\alph*),font=\upshape]
\item\label{it:1002} $\begin{tikzpicture}[xscale=1.0,yscale=1,baseline={(0,0)}]
    \node at (0,0.4) {$1$};
    \node at (1,0.4) {$0$};
    \node at (2,0.4) {$0$};
    \node at (3,0.4) {$2$};
    \node (A1) at (0,0) {$\circ$};
    \node (A2) at (1,0) {$\circ$};
    \node (A3) at (2,0) {$\circ$};
    \node (A4) at (3,0) {$\circ$};
    \path (A1) edge [-] node [auto] {$\scriptstyle{}$} (A2);
    \path (A2) edge [-] node [auto] {$\scriptstyle{}$} (A3);
    \path (A3) edge [-] node [auto] {$\scriptstyle{}$} (A4);
  \end{tikzpicture}$
\item\label{it:1003} $\begin{tikzpicture}[xscale=1.0,yscale=1,baseline={(0,0)}]
    \node at (0,0.4) {$1$};
    \node at (1,0.4) {$0$};
    \node at (2,0.4) {$0$};
    \node at (3,0.4) {$3$};
    \node (A1) at (0,0) {$\circ$};
    \node (A2) at (1,0) {$\circ$};
    \node (A3) at (2,0) {$\circ$};
    \node (A4) at (3,0) {$\circ$};
    \path (A1) edge [-] node [auto] {$\scriptstyle{}$} (A2);
    \path (A2) edge [-] node [auto] {$\scriptstyle{}$} (A3);
    \path (A3) edge [-] node [auto] {$\scriptstyle{}$} (A4);
  \end{tikzpicture}$
\item\label{it:10003} $\begin{tikzpicture}[xscale=1.0,yscale=1,baseline={(0,0)}]
    \node at (0,0.4) {$1$};
    \node at (1,0.4) {$0$};
    \node at (2,0.4) {$0$};
    \node at (3,0.4) {$0$};
    \node at (4,0.4) {$3$};
    \node (A1) at (0,0) {$\circ$};
    \node (A2) at (1,0) {$\circ$};
    \node (A3) at (2,0) {$\circ$};
    \node (A4) at (3,0) {$\circ$};
    \node (A5) at (4,0) {$\circ$};
    \path (A1) edge [-] node [auto] {$\scriptstyle{}$} (A2);
    \path (A2) edge [-] node [auto] {$\scriptstyle{}$} (A3);
    \path (A3) edge [-] node [auto] {$\scriptstyle{}$} (A4);
    \path (A4) edge [-] node [auto] {$\scriptstyle{}$} (A5);
  \end{tikzpicture}$
\item\label{it:110012} $\begin{tikzpicture}[xscale=1.0,yscale=1,baseline={(0,0)}]
    \node at (0,0.4) {$1$};
    \node at (1,0.4) {$1$};
    \node at (2,0.4) {$0$};
    \node at (3,0.4) {$0$};
    \node at (4,0.4) {$1$};
    \node at (5,0.4) {$2$};
    \node (A1) at (0,0) {$\circ$};
    \node (A2) at (1,0) {$\circ$};
    \node (A3) at (2,0) {$\circ$};
    \node (A4) at (3,0) {$\circ$};
    \node (A5) at (4,0) {$\circ$};
    \node (A6) at (5,0) {$\circ$};
    \path (A1) edge [-] node [auto] {$\scriptstyle{}$} (A2);
    \path (A2) edge [-] node [auto] {$\scriptstyle{}$} (A3);
    \path (A3) edge [-] node [auto] {$\scriptstyle{}$} (A4);
    \path (A4) edge [-] node [auto] {$\scriptstyle{}$} (A5);
    \path (A5) edge [-] node [auto] {$\scriptstyle{}$} (A6);
  \end{tikzpicture}$
\item\label{it:31001} $\begin{tikzpicture}[xscale=1.0,yscale=1,baseline={(0,0)}]
    \node at (0,0.4) {$3$};
    \node at (1,0.4) {$1$};
    \node at (2,0.4) {$0$};
    \node at (3,0.4) {$0$};
    \node at (4,0.4) {$1$};
    \node (A1) at (0,0) {$\circ$};
    \node (A2) at (1,0) {$\circ$};
    \node (A3) at (2,0) {$\circ$};
    \node (A4) at (3,0) {$\circ$};
    \node (A5) at (4,0) {$\circ$};
    \path (A1) edge [-] node [auto] {$\scriptstyle{}$} (A2);
    \path (A2) edge [-] node [auto] {$\scriptstyle{}$} (A3);
    \path (A3) edge [-] node [auto] {$\scriptstyle{}$} (A4);
    \path (A4) edge [-] node [auto] {$\scriptstyle{}$} (A5);
  \end{tikzpicture}$
\item\label{it:1012} $\begin{tikzpicture}[xscale=1.0,yscale=1,baseline={(0,0)}]
    \node at (0,0.4) {$1$};
    \node at (1,0.4) {$0$};
    \node at (2,0.4) {$1$};
    \node at (3,0.4) {$2$};
    \node (A1) at (0,0) {$\circ$};
    \node (A2) at (1,0) {$\circ$};
    \node (A3) at (2,0) {$\circ$};
    \node (A4) at (3,0) {$\circ$};
    \path (A1) edge [-] node [auto] {$\scriptstyle{}$} (A2);
    \path (A2) edge [-] node [auto] {$\scriptstyle{}$} (A3);
    \path (A3) edge [-] node [auto] {$\scriptstyle{}$} (A4);
  \end{tikzpicture}$
\item\label{it:1013} $\begin{tikzpicture}[xscale=1.0,yscale=1,baseline={(0,0)}]
    \node at (0,0.4) {$1$};
    \node at (1,0.4) {$0$};
    \node at (2,0.4) {$1$};
    \node at (3,0.4) {$3$};
    \node (A1) at (0,0) {$\circ$};
    \node (A2) at (1,0) {$\circ$};
    \node (A3) at (2,0) {$\circ$};
    \node (A4) at (3,0) {$\circ$};
    \path (A1) edge [-] node [auto] {$\scriptstyle{}$} (A2);
    \path (A2) edge [-] node [auto] {$\scriptstyle{}$} (A3);
    \path (A3) edge [-] node [auto] {$\scriptstyle{}$} (A4);
  \end{tikzpicture}$
\item\label{it:1102} $\begin{tikzpicture}[xscale=1.0,yscale=1,baseline={(0,0)}]
    \node at (0,0.4) {$1$};
    \node at (1,0.4) {$1$};
    \node at (2,0.4) {$0$};
    \node at (3,0.4) {$2$};
    \node (A1) at (0,0) {$\circ$};
    \node (A2) at (1,0) {$\circ$};
    \node (A3) at (2,0) {$\circ$};
    \node (A4) at (3,0) {$\circ$};
    \path (A1) edge [-] node [auto] {$\scriptstyle{}$} (A2);
    \path (A2) edge [-] node [auto] {$\scriptstyle{}$} (A3);
    \path (A3) edge [-] node [auto] {$\scriptstyle{}$} (A4);
  \end{tikzpicture}$
\item\label{it:110112} $\begin{tikzpicture}[xscale=1.0,yscale=1,baseline={(0,0)}]
    \node at (0,0.4) {$1$};
    \node at (1,0.4) {$1$};
    \node at (2,0.4) {$0$};
    \node at (3,0.4) {$1$};
    \node at (4,0.4) {$1$};
    \node at (5,0.4) {$2$};
    \node (A1) at (0,0) {$\circ$};
    \node (A2) at (1,0) {$\circ$};
    \node (A3) at (2,0) {$\circ$};
    \node (A4) at (3,0) {$\circ$};
    \node (A5) at (4,0) {$\circ$};
    \node (A6) at (5,0) {$\circ$};
    \path (A1) edge [-] node [auto] {$\scriptstyle{}$} (A2);
    \path (A2) edge [-] node [auto] {$\scriptstyle{}$} (A3);
    \path (A3) edge [-] node [auto] {$\scriptstyle{}$} (A4);
    \path (A4) edge [-] node [auto] {$\scriptstyle{}$} (A5);
    \path (A5) edge [-] node [auto] {$\scriptstyle{}$} (A6);
  \end{tikzpicture}$;
\end{enumerate}
\end{multicols}
\item\label{it:bad_part} if $\Gamma$ contains at least one bad vertex, then all bad vertices must be contained in one copy of one of the following subgraphs, with the adjusted weights of all bad vertices being circled:
\begin{multicols}{2}
\begin{enumerate}[label=(\alph*),font=\upshape]
\item\label{it:101} $\begin{tikzpicture}[xscale=1.0,yscale=1,baseline={(0,0)}]
    \node at (0,0.4) {$1$};
    \node at (1,0.4) {{$\red{\circled{0}}$}};
    \node at (2,0.4) {$1$};
    \node (A1) at (0,0) {$\circ$};
    \node (A2) at (1,0) {$\circ$};
    \node (A3) at (2,0) {$\circ$};
    \path (A1) edge [-] node [auto] {$\scriptstyle{}$} (A2);
    \path (A2) edge [-] node [auto] {$\scriptstyle{}$} (A3);
  \end{tikzpicture}$
\item\label{it:102} $\begin{tikzpicture}[xscale=1.0,yscale=1,baseline={(0,0)}]
    \node at (0,0.4) {$1$};
    \node at (1,0.4) {{$\red{\circled{0}}$}};
    \node at (2,0.4) {$2$};
    \node (A1) at (0,0) {$\circ$};
    \node (A2) at (1,0) {$\circ$};
    \node (A3) at (2,0) {$\circ$};
    \path (A1) edge [-] node [auto] {$\scriptstyle{}$} (A2);
    \path (A2) edge [-] node [auto] {$\scriptstyle{}$} (A3);
  \end{tikzpicture}$
\item\label{it:103} $\begin{tikzpicture}[xscale=1.0,yscale=1,baseline={(0,0)}]
    \node at (0,0.4) {$1$};
    \node at (1,0.4) {{$\red{\circled{0}}$}};
    \node at (2,0.4) {$3$};
    \node (A1) at (0,0) {$\circ$};
    \node (A2) at (1,0) {$\circ$};
    \node (A3) at (2,0) {$\circ$};
    \path (A1) edge [-] node [auto] {$\scriptstyle{}$} (A2);
    \path (A2) edge [-] node [auto] {$\scriptstyle{}$} (A3);
  \end{tikzpicture}$
\item\label{it:111} $\begin{tikzpicture}[xscale=1.0,yscale=1,baseline={(0,0)}]
    \node at (0,0.4) {$1$};
    \node at (1,0.4) {{$\red{\circled{1}}$}};
    \node at (2,0.4) {$1$};
    \node (A1) at (0,0) {$\circ$};
    \node (A2) at (1,0) {$\circ$};
    \node (A3) at (2,0) {$\circ$};
    \path (A1) edge [-] node [auto] {$\scriptstyle{}$} (A2);
    \path (A2) edge [-] node [auto] {$\scriptstyle{}$} (A3);
  \end{tikzpicture}$
\item\label{it:112} $\begin{tikzpicture}[xscale=1.0,yscale=1,baseline={(0,0)}]
    \node at (0,0.4) {$1$};
    \node at (1,0.4) {{$\red{\circled{1}}$}};
    \node at (2,0.4) {$2$};
    \node (A1) at (0,0) {$\circ$};
    \node (A2) at (1,0) {$\circ$};
    \node (A3) at (2,0) {$\circ$};
    \path (A1) edge [-] node [auto] {$\scriptstyle{}$} (A2);
    \path (A2) edge [-] node [auto] {$\scriptstyle{}$} (A3);
  \end{tikzpicture}$
\item\label{it:121} $\begin{tikzpicture}[xscale=1.0,yscale=1,baseline={(0,0)}]
    \node at (0,0.4) {$1$};
    \node at (1,0.4) {{$\red{\circled{2}}$}};
    \node at (2,0.4) {$1$};
    \node (A1) at (0,0) {$\circ$};
    \node (A2) at (1,0) {$\circ$};
    \node (A3) at (2,0) {$\circ$};
    \path (A1) edge [-] node [auto] {$\scriptstyle{}$} (A2);
    \path (A2) edge [-] node [auto] {$\scriptstyle{}$} (A3);
  \end{tikzpicture}$
\item\label{it:10101} $\begin{tikzpicture}[xscale=1.0,yscale=1,baseline={(0,0)}]
    \node at (0,0.4) {$1$};
    \node at (1,0.4) {{$\red{\circled{0}}$}};
    \node at (2,0.4) {$1$};
    \node at (3,0.4) {{$\red{\circled{0}}$}};
    \node at (4,0.4) {$1$};
    \node (A1) at (0,0) {$\circ$};
    \node (A2) at (1,0) {$\circ$};
    \node (A3) at (2,0) {$\circ$};
    \node (A4) at (3,0) {$\circ$};
    \node (A5) at (4,0) {$\circ$};
    \path (A1) edge [-] node [auto] {$\scriptstyle{}$} (A2);
    \path (A2) edge [-] node [auto] {$\scriptstyle{}$} (A3);
    \path (A3) edge [-] node [auto] {$\scriptstyle{}$} (A4);
    \path (A4) edge [-] node [auto] {$\scriptstyle{}$} (A5);
  \end{tikzpicture}$
\item\label{it:10102} $\begin{tikzpicture}[xscale=1.0,yscale=1,baseline={(0,0)}]
    \node at (0,0.4) {$1$};
    \node at (1,0.4) {{$\red{\circled{0}}$}};
    \node at (2,0.4) {$1$};
    \node at (3,0.4) {{$\red{\circled{0}}$}};
    \node at (4,0.4) {$2$};
    \node (A1) at (0,0) {$\circ$};
    \node (A2) at (1,0) {$\circ$};
    \node (A3) at (2,0) {$\circ$};
    \node (A4) at (3,0) {$\circ$};
    \node (A5) at (4,0) {$\circ$};
    \path (A1) edge [-] node [auto] {$\scriptstyle{}$} (A2);
    \path (A2) edge [-] node [auto] {$\scriptstyle{}$} (A3);
    \path (A3) edge [-] node [auto] {$\scriptstyle{}$} (A4);
    \path (A4) edge [-] node [auto] {$\scriptstyle{}$} (A5);
  \end{tikzpicture}$
\item\label{it:10111} $\begin{tikzpicture}[xscale=1.0,yscale=1,baseline={(0,0)}]
    \node at (0,0.4) {$1$};
    \node at (1,0.4) {{$\red{\circled{0}}$}};
    \node at (2,0.4) {$1$};
    \node at (3,0.4) {{$\red{\circled{1}}$}};
    \node at (4,0.4) {$1$};
    \node (A1) at (0,0) {$\circ$};
    \node (A2) at (1,0) {$\circ$};
    \node (A3) at (2,0) {$\circ$};
    \node (A4) at (3,0) {$\circ$};
    \node (A5) at (4,0) {$\circ$};
    \path (A1) edge [-] node [auto] {$\scriptstyle{}$} (A2);
    \path (A2) edge [-] node [auto] {$\scriptstyle{}$} (A3);
    \path (A3) edge [-] node [auto] {$\scriptstyle{}$} (A4);
    \path (A4) edge [-] node [auto] {$\scriptstyle{}$} (A5);
  \end{tikzpicture}$
\item\label{it:10112} $\begin{tikzpicture}[xscale=1.0,yscale=1,baseline={(0,0)}]
    \node at (0,0.4) {$1$};
    \node at (1,0.4) {{$\red{\circled{0}}$}};
    \node at (2,0.4) {$1$};
    \node at (3,0.4) {{$\red{\circled{1}}$}};
    \node at (4,0.4) {$2$};
    \node (A1) at (0,0) {$\circ$};
    \node (A2) at (1,0) {$\circ$};
    \node (A3) at (2,0) {$\circ$};
    \node (A4) at (3,0) {$\circ$};
    \node (A5) at (4,0) {$\circ$};
    \path (A1) edge [-] node [auto] {$\scriptstyle{}$} (A2);
    \path (A2) edge [-] node [auto] {$\scriptstyle{}$} (A3);
    \path (A3) edge [-] node [auto] {$\scriptstyle{}$} (A4);
    \path (A4) edge [-] node [auto] {$\scriptstyle{}$} (A5);
  \end{tikzpicture}$
\item\label{it:1111} $\begin{tikzpicture}[xscale=1.0,yscale=1,baseline={(0,0)}]
    \node at (0,0.4) {$1$};
    \node at (1,0.4) {{$\red{\circled{1}}$}};
    \node at (2,0.4) {{$\red{\circled{1}}$}};
    \node at (3,0.4) {$1$};
    \node (A1) at (0,0) {$\circ$};
    \node (A2) at (1,0) {$\circ$};
    \node (A3) at (2,0) {$\circ$};
    \node (A4) at (3,0) {$\circ$};
    \path (A1) edge [-] node [auto] {$\scriptstyle{}$} (A2);
    \path (A2) edge [-] node [auto] {$\scriptstyle{}$} (A3);
    \path (A3) edge [-] node [auto] {$\scriptstyle{}$} (A4);
  \end{tikzpicture}$
\item\label{it:11111} $\begin{tikzpicture}[xscale=1.0,yscale=1,baseline={(0,0)}]
    \node at (0,0.4) {$1$};
    \node at (1,0.4) {{$\red{\circled{1}}$}};
    \node at (2,0.4) {{$\red{\circled{1}}$}};
    \node at (3,0.4) {{$\red{\circled{1}}$}};
    \node at (4,0.4) {$1$};
    \node (A1) at (0,0) {$\circ$};
    \node (A2) at (1,0) {$\circ$};
    \node (A3) at (2,0) {$\circ$};
    \node (A4) at (3,0) {$\circ$};
    \node (A5) at (4,0) {$\circ$};
    \path (A1) edge [-] node [auto] {$\scriptstyle{}$} (A2);
    \path (A2) edge [-] node [auto] {$\scriptstyle{}$} (A3);
    \path (A3) edge [-] node [auto] {$\scriptstyle{}$} (A4);
    \path (A4) edge [-] node [auto] {$\scriptstyle{}$} (A5);
  \end{tikzpicture}$
\item\label{it:1121} $\begin{tikzpicture}[xscale=1.0,yscale=1,baseline={(0,0)}]
    \node at (0,0.4) {$1$};
    \node at (1,0.4) {{$\red{\circled{1}}$}};
    \node at (2,0.4) {{$\red{\circled{2}}$}};
    \node at (3,0.4) {$1$};
    \node (A1) at (0,0) {$\circ$};
    \node (A2) at (1,0) {$\circ$};
    \node (A3) at (2,0) {$\circ$};
    \node (A4) at (3,0) {$\circ$};
    \path (A1) edge [-] node [auto] {$\scriptstyle{}$} (A2);
    \path (A2) edge [-] node [auto] {$\scriptstyle{}$} (A3);
    \path (A3) edge [-] node [auto] {$\scriptstyle{}$} (A4);
  \end{tikzpicture}$.
\end{enumerate}
\end{multicols}
\end{enumerate}
\end{conditions}

\begin{rem}
Throughout this paper, unless explicitly stated otherwise, whenever we draw a subgraph of $\Gamma$ using adjusted weights, the labels are understood to be the adjusted weights computed in $\Gamma$.

For example, if $\Gamma$ is
\[
\begin{tikzpicture}[xscale=1.0,yscale=1,baseline={(0,0)}]
    \node at (0-0., .4) {$4$};
    \node at (1-0., .4) {$2$};
    \node at (2-0., .4) {$2$};
    \node at (3-0., .4) {$2$};
    \node at (4-0., .4) {$2$};
    \node at (5-0., .4) {$2$};
    \node (A1) at (0,0) {$\bullet$};
    \node (A2) at (1,0) {$\bullet$};
    \node (A3) at (2,0) {$\bullet$};
    \node (A4) at (3,0) {$\bullet$};
    \node (A5) at (4,0) {$\bullet$};
    \node (A6) at (5,0) {$\bullet$};
    \path (A1) edge [-] node [auto] {$\scriptstyle{}$} (A2);
    \path (A2) edge [-] node [auto] {$\scriptstyle{}$} (A3);
    \path (A3) edge [-] node [auto] {$\scriptstyle{}$} (A4);
    \path (A4) edge [-] node [auto] {$\scriptstyle{}$} (A5);
    \path (A5) edge [-] node [auto] {$\scriptstyle{}$} (A6);
\end{tikzpicture},
\]
then $\Gamma$ satisfies the Working Conditions. Consider the subgraph
\[
\begin{tikzpicture}[xscale=1.0,yscale=1,baseline={(0,0)}]
    \node at (0-0., .4) {$4$};
    \node at (1-0., .4) {$2$};
    \node at (2-0., .4) {$2$};
    \node at (3-0., .4) {$2$};
    \node at (4-0, .4) {$2$};
    \node (A1) at (0,0) {$\bullet$};
    \node (A2) at (1,0) {$\bullet$};
    \node (A3) at (2,0) {$\bullet$};
    \node (A4) at (3,0) {$\bullet$};
    \node (A5) at (4,0) {$\bullet$};
    \path (A1) edge [-] node [auto] {$\scriptstyle{}$} (A2);
    \path (A2) edge [-] node [auto] {$\scriptstyle{}$} (A3);
    \path (A3) edge [-] node [auto] {$\scriptstyle{}$} (A4);
    \path (A4) edge [-] node [auto] {$\scriptstyle{}$} (A5);
\end{tikzpicture}
\]
considered as an individual graph. Its adjusted weights are
\[
\begin{tikzpicture}[xscale=1.0,yscale=1,baseline={(0,0)}]
    \node at (0,0.4) {$3$};
    \node at (1,0.4) {$0$};
    \node at (2,0.4) {$0$};
    \node at (3,0.4) {$0$};
    \node at (4,0.4) {$1$};
    \node (A1) at (0,0) {$\circ$};
    \node (A2) at (1,0) {$\circ$};
    \node (A3) at (2,0) {$\circ$};
    \node (A4) at (3,0) {$\circ$};
    \node (A5) at (4,0) {$\circ$};
    \path (A1) edge [-] node [auto] {$\scriptstyle{}$} (A2);
    \path (A2) edge [-] node [auto] {$\scriptstyle{}$} (A3);
    \path (A3) edge [-] node [auto] {$\scriptstyle{}$} (A4);
    \path (A4) edge [-] node [auto] {$\scriptstyle{}$} (A5);
\end{tikzpicture}
\]
which is forbidden by Working Condition~\ref{it:forbidden_configs}. However, regarded as a subgraph of $\Gamma$, it has adjusted weights
\[
\begin{tikzpicture}[xscale=1.0,yscale=1,baseline={(0,0)}]
    \node at (0,0.4) {$3$};
    \node at (1,0.4) {$0$};
    \node at (2,0.4) {$0$};
    \node at (3,0.4) {$0$};
    \node at (4,0.4) {$0$};
    \node (A1) at (0,0) {$\circ$};
    \node (A2) at (1,0) {$\circ$};
    \node (A3) at (2,0) {$\circ$};
    \node (A4) at (3,0) {$\circ$};
    \node (A5) at (4,0) {$\circ$};
    \path (A1) edge [-] node [auto] {$\scriptstyle{}$} (A2);
    \path (A2) edge [-] node [auto] {$\scriptstyle{}$} (A3);
    \path (A3) edge [-] node [auto] {$\scriptstyle{}$} (A4);
    \path (A4) edge [-] node [auto] {$\scriptstyle{}$} (A5);
\end{tikzpicture},
\]
which does not violate Working Condition~\ref{it:forbidden_configs}.
\end{rem}

\begin{rem}\label{rem:comparintworkingcondition}
We can compare the Working Conditions in this paper with the Working Conditions in \cite{definite}. The first two conditions are the same. Condition~III in \cite{definite} says that there are no bad vertices. Condition~IV is the same. Combined with Condition~II, Condition~V in \cite{definite} disallows six configurations: three of them involve a bad vertex with adjusted weight $0$, and the other three are Conditions \ref{it:forbidden_configs}\ref{it:1002}, \ref{it:forbidden_configs}\ref{it:1003}, and \ref{it:forbidden_configs}\ref{it:10003} in this paper. The two configurations in Condition~VI of \cite{definite} are Conditions \ref{it:forbidden_configs}\ref{it:110012} and \ref{it:forbidden_configs}\ref{it:31001} in this paper.

Hence, if $\Gamma$ has no bad vertices and satisfies the Working Conditions in this paper, then it satisfies the Working Conditions in \cite{definite}. Therefore, once we translate the conditions in Theorem~\ref{thm:main}\ref{it:combinatorial} into our Working Conditions, we may restrict attention to linear plumbing graphs satisfying the Working Conditions and having at least one bad vertex.\end{rem}

\begin{lem}\label{lem:working conditions proof}
Let $P$ be a disjoint union of linear plumbing graphs such that $P$ does not contain any of the configurations in Theorem~\ref{thm:main}\ref{it:combinatorial} and $w(v)\leq -2$ for all vertices $v$. Then the dual graph $P^*$ satisfies the Working Conditions.
\end{lem}

\begin{proof}
As in \cite[Lemma~3.6]{definite}, Working Condition~\ref{it:positivity} arises from the definition of dual plumbing. Working Condition~\ref{it:largeweight} comes from the requirement to avoid Theorem~\ref{thm:main}\ref{it:combinatorial}\ref{it:2-2}. Working Condition~\ref{it:weight3condition} comes from the requirement to avoid Theorem~\ref{thm:main}\ref{it:combinatorial}\ref{it:3-22}. Working Conditions~\ref{it:forbidden_configs}\ref{it:1002}, \ref{it:forbidden_configs}\ref{it:1003}, \ref{it:forbidden_configs}\ref{it:10003}, \ref{it:forbidden_configs}\ref{it:110012}, and \ref{it:forbidden_configs}\ref{it:31001} come from the requirement to avoid Theorem~\ref{thm:main}\ref{it:combinatorial}\ref{it:52}, \ref{it:combinatorial}\ref{it:622}, \ref{it:combinatorial}\ref{it:3532}, and \ref{it:combinatorial}\ref{it:2235}.

The remaining conditions come from configurations that appear in this paper but not in \cite{definite}. We go through each of these conditions in turn and use Lemma~\ref{lem:duality} to show that, if any of them fails, then $P$ would contain one of the subgraphs listed in Theorem~\ref{thm:main}\ref{it:combinatorial}.

If we have a bad vertex $x$ with $w'(x)=1$ in $P^*$, this corresponds to a subgraph of $P$ of the form
\[
\begin{tikzpicture}[xscale=1.0,yscale=1,baseline={(0,0)}]
    \node at (0.9,0.4) {$-3$};
    \node at (1.9,0.4) {$-3$};
    \node (A1) at (1,0) {$\bullet$};
    \node (A2) at (2,0) {$\bullet$};
    \path (A1) edge [-] node [auto] {$\scriptstyle{}$} (A2);
\end{tikzpicture}.
\]
A vertex $v$ in $P^*$ with $w'(v)>1$ gives rise to a vertex of weight $-2$ in $P$. If $x$ and $v$ are not adjacent to each other, then $P$ contains the configuration in Theorem~\ref{thm:main}\ref{it:combinatorial}\ref{it:33-2}. This gives Working Condition~\ref{it:bad1}.

If we have a bad vertex $x$ with $w'(x)=2$ in $P^*$, this corresponds to a subgraph of $P$ of the form
\[
\begin{tikzpicture}[xscale=1.0,yscale=1,baseline={(0,0)}]
    \node at (0.9,0.4) {$-3$};
    \node at (1.9,0.4) {$-2$};
    \node at (2.9,0.4) {$-3$};
    \node (A1) at (1,0) {$\bullet$};
    \node (A2) at (2,0) {$\bullet$};
    \node (A3) at (3,0) {$\bullet$};
    \path (A1) edge [-] node [auto] {$\scriptstyle{}$} (A2);
    \path (A2) edge [-] node [auto] {$\scriptstyle{}$} (A3);
\end{tikzpicture}.
\]
The configuration
\[
\begin{tikzpicture}[xscale=1.0,yscale=1,baseline={(0,0)}]
    \node at (1,0.4) {$1$};
    \node at (2,0.4) {$1$};
    \node (A1) at (1,0) {$\circ$};
    \node (A2) at (2,0) {$\circ$};
    \path (A1) edge [-] node [auto] {$\scriptstyle{}$} (A2);
\end{tikzpicture}
\]
in $P^*$ gives rise to a vertex of weight $-3$ in $P$. If $x$ and a copy of
\[
\begin{tikzpicture}[xscale=1.0,yscale=1,baseline={(0,0)}]
    \node at (1,0.4) {$1$};
    \node at (2,0.4) {$1$};
    \node (A1) at (1,0) {$\circ$};
    \node (A2) at (2,0) {$\circ$};
    \path (A1) edge [-] node [auto] {$\scriptstyle{}$} (A2);
\end{tikzpicture}
\]
are not adjacent to each other, then $P$ contains the configuration in Theorem~\ref{thm:main}\ref{it:combinatorial}\ref{it:323-3}. This gives Working Condition~\ref{it:bad2}.

Working Conditions~\ref{it:forbidden_configs}\ref{it:1012} and \ref{it:forbidden_configs}\ref{it:1013} correspond to Theorem~\ref{thm:main}\ref{it:combinatorial}\ref{it:432}. Working Condition~\ref{it:forbidden_configs}\ref{it:1102} corresponds to Theorem~\ref{thm:main}\ref{it:combinatorial}\ref{it:342}. Working Condition~\ref{it:forbidden_configs}\ref{it:110112} corresponds to Theorem~\ref{thm:main}\ref{it:combinatorial}\ref{it:34332}.

If there is a bad vertex $x$ with $w'(x)=3$ in $P^*$, this gives rise to the configuration in Theorem~\ref{thm:main}\ref{it:combinatorial}\ref{it:3223} in $P$. Therefore, by Working Condition~\ref{it:largeweight}, for every bad vertex $x$ in $P^*$ we must have $w'(x)\leq 2$. If there is exactly one bad vertex, then, in order to comply with Working Conditions~\ref{it:largeweight} and \ref{it:weight3condition}, the bad vertex must arise from one of Working Conditions~\ref{it:bad_part}\ref{it:101}, \ref{it:bad_part}\ref{it:102}, \ref{it:bad_part}\ref{it:103}, \ref{it:bad_part}\ref{it:111}, \ref{it:bad_part}\ref{it:112}, or \ref{it:bad_part}\ref{it:121}.

If there is more than one bad vertex, first observe that when $x$ is a bad vertex in $P^*$, the case $w'(x)=0$ gives rise to
\[
\begin{tikzpicture}[xscale=1,yscale=1,baseline={(0,0)}]
    \node at (0.9,0.4) {$-4$};
    \node (A1) at (1,0) {$\bullet$};
\end{tikzpicture}
\]
in $P$, the case $w'(x)=1$ gives rise to
\[
\begin{tikzpicture}[xscale=1,yscale=1,baseline={(0,0)}]
    \node at (0.9,0.4) {$-3$};
    \node at (1.9,0.4) {$-3$};
    \node (A1) at (1,0) {$\bullet$};
    \node (A2) at (2,0) {$\bullet$};
    \path (A1) edge [-] node [auto] {$\scriptstyle{}$} (A2);
\end{tikzpicture}
\]
in $P$, and the case $w'(x)=2$ gives rise to
\[
\begin{tikzpicture}[xscale=1,yscale=1,baseline={(0,0)}]
    \node at (0.9,0.4) {$-3$};
    \node at (1.9,0.4) {$-2$};
    \node at (2.9,0.4) {$-3$};
    \node (A1) at (1,0) {$\bullet$};
    \node (A2) at (2,0) {$\bullet$};
    \node (A3) at (3,0) {$\bullet$};
    \path (A1) edge [-] node [auto] {$\scriptstyle{}$} (A2);
    \path (A2) edge [-] node [auto] {$\scriptstyle{}$} (A3);
\end{tikzpicture}
\]
in $P$.
To avoid the configurations in Theorem~\ref{thm:main}\ref{it:combinatorial}\ref{it:2-2}\ref{it:33-2}\ref{it:4-4}\ref{it:4-33}\ref{it:4-323}\ref{it:33-33}, the graph $P$ cannot have an induced subgraph with two components, each of which is one of the three subgraphs displayed above. Therefore, when there is more than one bad vertex in $P^*$, the corresponding configuration in $P$ must be obtained by connecting copies of these three subgraphs together. By also taking into account the need to avoid the configurations in Theorem~\ref{thm:main}\ref{it:combinatorial}\ref{it:2-2}\ref{it:33-2}\ref{it:432}, we see that there are only five possibilities:
\[
\begin{tikzpicture}[xscale=1,yscale=1,baseline={(0,0)}]
    \node at (0.9,0.4) {$-4$};
    \node at (1.9,0.4) {$-4$};
    \node (A1) at (1,0) {$\bullet$};
    \node (A2) at (2,0) {$\bullet$};
    \path (A1) edge [-] node [auto] {$\scriptstyle{}$} (A2);
\end{tikzpicture},\quad
\begin{tikzpicture}[xscale=1,yscale=1,baseline={(0,0)}]
    \node at (0.9,0.4) {$-4$};
    \node at (1.9,0.4) {$-3$};
    \node at (2.9,0.4) {$-3$};
    \node (A1) at (1,0) {$\bullet$};
    \node (A2) at (2,0) {$\bullet$};
    \node (A3) at (3,0) {$\bullet$};
    \path (A1) edge [-] node [auto] {$\scriptstyle{}$} (A2);
    \path (A2) edge [-] node [auto] {$\scriptstyle{}$} (A3);
\end{tikzpicture},\quad
\begin{tikzpicture}[xscale=1,yscale=1,baseline={(0,0)}]
    \node at (0.9,0.4) {$-3$};
    \node at (1.9,0.4) {$-3$};
    \node at (2.9,0.4) {$-3$};
    \node (A1) at (1,0) {$\bullet$};
    \node (A2) at (2,0) {$\bullet$};
    \node (A3) at (3,0) {$\bullet$};
    \path (A1) edge [-] node [auto] {$\scriptstyle{}$} (A2);
    \path (A2) edge [-] node [auto] {$\scriptstyle{}$} (A3);
\end{tikzpicture},\]

\[
\begin{tikzpicture}[xscale=1,yscale=1,baseline={(0,0)}]
    \node at (0.9,0.4) {$-3$};
    \node at (1.9,0.4) {$-3$};
    \node at (2.9,0.4) {$-3$};
    \node at (3.9,0.4) {$-3$};
    \node (A1) at (1,0) {$\bullet$};
    \node (A2) at (2,0) {$\bullet$};
    \node (A3) at (3,0) {$\bullet$};
    \node (A4) at (4,0) {$\bullet$};
    \path (A1) edge [-] node [auto] {$\scriptstyle{}$} (A2);
    \path (A2) edge [-] node [auto] {$\scriptstyle{}$} (A3);
    \path (A3) edge [-] node [auto] {$\scriptstyle{}$} (A4);
\end{tikzpicture},\quad
\begin{tikzpicture}[xscale=1,yscale=1,baseline={(0,0)}]
    \node at (0.9,0.4) {$-3$};
    \node at (1.9,0.4) {$-3$};
    \node at (2.9,0.4) {$-2$};
    \node at (3.9,0.4) {$-3$};
    \node (A1) at (1,0) {$\bullet$};
    \node (A2) at (2,0) {$\bullet$};
    \node (A3) at (3,0) {$\bullet$};
    \node (A4) at (4,0) {$\bullet$};
    \path (A1) edge [-] node [auto] {$\scriptstyle{}$} (A2);
    \path (A2) edge [-] node [auto] {$\scriptstyle{}$} (A3);
    \path (A3) edge [-] node [auto] {$\scriptstyle{}$} (A4);
\end{tikzpicture}.
\]

To avoid the configuration in Theorem~\ref{thm:main}\ref{it:combinatorial}\ref{it:4422}, we cannot have
\[
\begin{tikzpicture}[xscale=1.0,yscale=1,baseline={(0,0)}]
    \node at (0,0.4) {$1$};
    \node at (1,0.4) {{$\red{\circled{0}}$}};
    \node at (2,0.4) {$1$};
    \node at (3,0.4) {{$\red{\circled{0}}$}};
    \node at (4,0.4) {$3$};
    \node (A1) at (0,0) {$\circ$};
    \node (A2) at (1,0) {$\circ$};
    \node (A3) at (2,0) {$\circ$};
    \node (A4) at (3,0) {$\circ$};
    \node (A5) at (4,0) {$\circ$};
    \path (A1) edge [-] node [auto] {$\scriptstyle{}$} (A2);
    \path (A2) edge [-] node [auto] {$\scriptstyle{}$} (A3);
    \path (A3) edge [-] node [auto] {$\scriptstyle{}$} (A4);
    \path (A4) edge [-] node [auto] {$\scriptstyle{}$} (A5);
\end{tikzpicture}
\]
in $P^*$. Therefore, by also considering Working Condition~\ref{it:largeweight}, the configuration
\[
\begin{tikzpicture}[xscale=1,yscale=1,baseline={(0,0)}]
    \node at (0.9,0.4) {$-4$};
    \node at (1.9,0.4) {$-4$};
    \node (A1) at (1,0) {$\bullet$};
    \node (A2) at (2,0) {$\bullet$};
    \path (A1) edge [-] node [auto] {$\scriptstyle{}$} (A2);
\end{tikzpicture}
\]
in $P$ corresponds to Working Conditions~\ref{it:bad_part}\ref{it:10101} and \ref{it:bad_part}\ref{it:10102} in $P^*$. By considering Working Conditions~\ref{it:largeweight}, \ref{it:weight3condition}, and \ref{it:forbidden_configs}\ref{it:1102}, the configuration
\[
\begin{tikzpicture}[xscale=1,yscale=1,baseline={(0,0)}]
    \node at (0.9,0.4) {$-4$};
    \node at (1.9,0.4) {$-3$};
    \node at (2.9,0.4) {$-3$};
    \node (A1) at (1,0) {$\bullet$};
    \node (A2) at (2,0) {$\bullet$};
    \node (A3) at (3,0) {$\bullet$};
    \path (A1) edge [-] node [auto] {$\scriptstyle{}$} (A2);
    \path (A2) edge [-] node [auto] {$\scriptstyle{}$} (A3);
\end{tikzpicture}
\]
in $P$ corresponds to Working Conditions~\ref{it:bad_part}\ref{it:10111} and \ref{it:bad_part}\ref{it:10112} in $P^*$. By considering Working Conditions~\ref{it:largeweight} and \ref{it:bad1}, the configuration
\[
\begin{tikzpicture}[xscale=1,yscale=1,baseline={(0,0)}]
    \node at (0.9,0.4) {$-3$};
    \node at (1.9,0.4) {$-3$};
    \node at (2.9,0.4) {$-3$};
    \node (A1) at (1,0) {$\bullet$};
    \node (A2) at (2,0) {$\bullet$};
    \node (A3) at (3,0) {$\bullet$};
    \path (A1) edge [-] node [auto] {$\scriptstyle{}$} (A2);
    \path (A2) edge [-] node [auto] {$\scriptstyle{}$} (A3);
\end{tikzpicture}
\]
in $P$ corresponds to Working Condition~\ref{it:bad_part}\ref{it:1111} in $P^*$. Again by considering Working Conditions~\ref{it:largeweight} and \ref{it:bad1}, the configuration
\[
\begin{tikzpicture}[xscale=1,yscale=1,baseline={(0,0)}]
    \node at (0.9,0.4) {$-3$};
    \node at (1.9,0.4) {$-3$};
    \node at (2.9,0.4) {$-3$};
    \node at (3.9,0.4) {$-3$};
    \node (A1) at (1,0) {$\bullet$};
    \node (A2) at (2,0) {$\bullet$};
    \node (A3) at (3,0) {$\bullet$};
    \node (A4) at (4,0) {$\bullet$};
    \path (A1) edge [-] node [auto] {$\scriptstyle{}$} (A2);
    \path (A2) edge [-] node [auto] {$\scriptstyle{}$} (A3);
    \path (A3) edge [-] node [auto] {$\scriptstyle{}$} (A4);
\end{tikzpicture}
\]
in $P$ corresponds to Working Condition~\ref{it:bad_part}\ref{it:11111} in $P^*$. Finally, by considering Working Condition~\ref{it:largeweight}, the configuration
\[
\begin{tikzpicture}[xscale=1,yscale=1,baseline={(0,0)}]
    \node at (0.9,0.4) {$-3$};
    \node at (1.9,0.4) {$-3$};
    \node at (2.9,0.4) {$-2$};
    \node at (3.9,0.4) {$-3$};
    \node (A1) at (1,0) {$\bullet$};
    \node (A2) at (2,0) {$\bullet$};
    \node (A3) at (3,0) {$\bullet$};
    \node (A4) at (4,0) {$\bullet$};
    \path (A1) edge [-] node [auto] {$\scriptstyle{}$} (A2);
    \path (A2) edge [-] node [auto] {$\scriptstyle{}$} (A3);
    \path (A3) edge [-] node [auto] {$\scriptstyle{}$} (A4);
\end{tikzpicture}
\]
in $P$ corresponds to Working Condition~\ref{it:bad_part}\ref{it:1121} in $P^*$.
\end{proof}

\section{Standard and semi-standard embeddings}\label{sec:definitions}

In this section, we make some definitions to help with the lattice embedding analysis in Section~\ref{sec:lattice_analysis}.

Throughout this section, we let $\Gamma$ be a disjoint union of linear plumbing graphs and let $V(\Gamma)$ denote its set of vertices. We first introduce some terminology specific to the problem of mapping linear lattices into diagonal lattices. Throughout this section, for simplicity we refer to the standard positive-definite lattice $(\Z^n,\langle 1 \rangle^n)$ simply as $\Z^n$.

\begin{defn}
An \emph{embedding} of $\Gamma$ into $\Z^n$ is a function
\[
\phi \colon V(\Gamma) \longrightarrow \Z^n
\]
such that for all $u,v \in V(\Gamma)$ we have
\[
\phi(u) \cdot \phi(v)= 
\begin{cases}
w(u) & \text{if $u=v$},\\
-1 & \text{if $u$ and $v$ are adjacent},\\
0 & \text{otherwise.}
\end{cases}
\]
\end{defn}

% Since lattice maps from $Q_\Gamma$ to $\Z^n$ are determined by the images of the vertices, embeddings of $\Gamma$ into $\Z^n$ are in bijection with lattice maps from $Q_\Gamma$ to $\Z^n$. Similarly, we define what it means for a plumbing graph to be rigid.

Given an orthonormal basis $E$ of $\Z^n$, we define the \textit{support} of a vector $v\in\Z^n$ with respect to $E$ to be
\[
\supp(v):=\{e\in E \mid \langle v,e\rangle\neq 0\}.
\]

When we mention the support of a vector without mentioning an orthonormal basis, it is implicitly understood that we have chosen an orthonormal basis.

\begin{defn}\cite[Def.~4.4]{definite}\label{def:standard}
An embedding $\phi \colon V(\Gamma) \rightarrow \Z^n$ is \emph{standard} if 
\[
|\supp(\phi(u))\cap \supp(\phi(v))| = \begin{cases}
w(v) &\text{if $u=v$},\\
1 & \text{if $u$ and $v$ are adjacent},\\
0 &\text{otherwise}
\end{cases}
\] 
for all vertices $u$ and $v$ in $V(\Gamma)$.
\end{defn}

For our purposes, we need the following new definitions.

\begin{defn}\label{def:semi-standard1}
An embedding $\phi \colon V(\Gamma) \rightarrow \Z^n$ is \emph{semi-standard} at a bad vertex $x$ if 
\[
|\supp(\phi(u))\cap \supp(\phi(v))| = \begin{cases}
w(v) &\text{if $u=v$},\\
1 & \text{if $u$ and $v$ are adjacent},\\
2 & \text{if $u$ and $v$ are both adjacent to $x$ and $u\neq v$},\\
0 &\text{otherwise}
\end{cases}
\] 
for all vertices $u$ and $v$ in $V(\Gamma)$. In addition, if $u$ and $v$ are both adjacent to $x$ and $u\neq v$, then
\[
|\supp(\phi(u))\cap\supp(\phi(x))\cap\supp(\phi(v))|=1.
\]
\end{defn}

\begin{defn}\label{def:semi-standard2}
An embedding $\phi \colon V(\Gamma) \rightarrow \Z^n$ is \emph{semi-standard} if it is semi-standard at one bad vertex.

Whenever we say that an induced subgraph of $\Gamma$ is embedded standardly or semi-standardly, we mean that $\phi$ restricted to the vertices of that subgraph is standard or semi-standard, respectively.
\end{defn}

\begin{defn}\label{def:bad-part}
If $\Gamma$ satisfies the Working Conditions and has at least one bad vertex, the \emph{bad part} of $\Gamma$ is the subgraph appearing in Working Condition~\ref{it:bad_part}.
\end{defn}

\begin{defn}\label{def:extended-bad-part}
Suppose $\Gamma$ satisfies the Working Conditions and has at least one bad vertex. Consider the set of subgraphs of $\Gamma$ that
\begin{itemize}
\item are of the form of a chain of adjacent vertices $v_0,\ldots,v_k$,
\item contain the bad part, and
\item satisfy both of the following:
\begin{enumerate}
\item either the vertex $v_0$ is a leaf of $\Gamma$, or $w'(v_1)=w'(v_2)=0$;
\item either the vertex $v_k$ is a leaf of $\Gamma$, or $w'(v_{k-1})=w'(v_{k-2})=0$.
\end{enumerate}
\end{itemize}
Among all these subgraphs, the one with the smallest number of vertices is called the \emph{extended bad part} of $\Gamma$.
\end{defn}

The extended bad part can be obtained by starting with the bad part and then traveling in both directions along the chain. We stop when we either reach a leaf, or have just passed two consecutive vertices with vanishing adjusted weights.

The intuitive idea for why the extended bad part is defined this way is that whenever we have two consecutive vertices with vanishing adjusted weights, we use them to ``glue'' the embeddings together. We will eventually show that the embedding is as desired within the extended bad part, and also outside of it, and then ``glue'' them along these consecutive vertices with vanishing adjusted weights.

\begin{defn}
Suppose $\Gamma$ satisfies the Working Conditions and has at least one bad vertex. Let $v_0,\ldots,v_k$ be the extended bad part of $\Gamma$. Let $V_c$ be the set of vertices of $\Gamma$ that are not in the extended bad part. Define
\[
V_- :=
\begin{cases}
\{v_0,v_1,v_2\} & \text{if $v_0$ is not a leaf of $\Gamma$},\\
\varnothing & \text{otherwise,}
\end{cases}
\qquad
V_+ :=
\begin{cases}
\{v_k,v_{k-1},v_{k-2}\} & \text{if $v_k$ is a not leaf of $\Gamma$},\\
\varnothing & \text{otherwise.}
\end{cases}
\]
The \emph{outside part} of $\Gamma$ is the subgraph of $\Gamma$ induced by the set of vertices $V_c \cup V_- \cup V_+$.
\end{defn}

Intuitively, the outside part is whatever outside of the extended bad part, plus the ``gluing part'' mentioned above.

\begin{defn}
Suppose $\Gamma$ satisfies the Working Conditions and has at least one bad vertex. Let $v_0,\ldots,v_k$ be the extended bad part of $\Gamma$, and let $\phi \colon V(\Gamma) \rightarrow \Z^n$ be an embedding of $\Gamma$.

If $v_0$ is not a leaf of $\Gamma$, then by the definition of the extended bad part we have $w'(v_2)=0$, and hence $w(v_2)=2$. Since $v_2 \cdot v_3 = -1$, we have
\[
\supp(\phi(v_2)) \cap \supp(\phi(v_3)) = \{e\}
\]
for some basis vector $e$. We call $e$ a \emph{screw} of $\Gamma$. Furthermore, if for all vertices $v \neq v_2,v_3$ in $\Gamma$ we have $\phi(v)\cdot e = 0$, we say that $e$ is a \emph{clean screw}.

Similarly, if $v_k$ is not a leaf of $\Gamma$, we call the unique basis vector in
\[
\supp(\phi(v_{k-2})) \cap \supp(\phi(v_{k-3}))= \{e\}
\]
a \emph{screw}, and we say that $e$ is \emph{clean} if $\phi(v)\cdot e = 0$ for all vertices $v \neq v_{k-2},v_{k-3}$ in $\Gamma$.
\end{defn}

By definition, $\Gamma$ has at most two screws. Intuitively, screws connect the outside part to the remaining portion of the extended bad part. If all screws are clean, we can “unscrew’’ them to detach the outside part and then argue that the outside part must be embedded standardly. We need the screws to be clean to locally modify the embedding there without messing up the intersection pairing and weights of vertices elsewhere.

\begin{lem}\label{lem:unscrew}
Suppose $\Gamma$ satisfies the Working Conditions and has at least one bad vertex. Let $\phi \colon V(\Gamma) \rightarrow \Z^n$ be an embedding of $\Gamma$. If all screws of $\Gamma$ are clean, then the outside part of $\Gamma$ is embedded standardly.
\end{lem}

It is tempting to argue that, since the outside part does not contain any bad vertices, it satisfies the Working Conditions of \cite{definite} and therefore must be embedded standardly. However, a problem with this argument is that the adjusted weights depend on the number of neighbours, not just the weights themselves. For a vertex in the outside part that is attached to the rest of $\Gamma$, the adjusted weight computed in $\Gamma$ differs from the adjusted weight of the same vertex when we view the outside part as an individual graph. To resolve this, we use the assumption that the screws are clean to manipulate the weights.

\begin{proof}[Proof of Lemma \ref{lem:unscrew}]
Let $\{e_1,\ldots,e_n\}$ be an orthonormal basis of $\Z^n$.

If there are no screws, then the extended bad part is a separate component of $\Gamma$. Hence, all vertices in the outside part have the same number of neighbours regardless of whether we view them as vertices in $\Gamma$ or as vertices in the outside part considered as an individual graph. Thus, as an individual graph, the outside part satisfies the Working Conditions in \cite{definite}, and therefore it must be embedded standardly.

If there is exactly one screw, let the extended bad part be $v_0,\ldots,v_k$, where $v_0$ is not a leaf and $v_k$ is a leaf. Let $\Gamma'$ be the graph obtained by adding an extra vertex $v_+$ to the outside part of $\Gamma$ such that $v_+$ is adjacent to $v_2$, is not adjacent to any other vertices, and $w(v_+)=w(v_3)-1$.

By construction, the adjusted weight of $v_+$ in $\Gamma'$ is the same as the adjusted weight of $v_3$ in $\Gamma$. Hence $\Gamma'$ satisfies the Working Conditions. Since $\Gamma'$ has no bad vertices, $\Gamma'$ also satisfies the Working Conditions in \cite{definite}, and therefore any embedding of $\Gamma'$ must be standard. Consider the embedding $$\phi' \colon V(\Gamma')\rightarrow\Z^{n+w(v_+)-1}\cong\Z^n\oplus\Z^{w(v_+)-1}$$ defined as follows.  For $v\neq v_+$, we embed $v$ into $\Z^n\subset\Z^{n+w(v_+)-1}$ in the same way as $\phi$. Let $e_{n+1},\ldots,e_{n+w(v_+)-1}$ be new basis elements. Up to relabelling the basis vectors in $\Z^n$, we may assume that $e_1$ is the screw, so that $e_1\in \supp(\phi(v_2))\cap\supp(\phi(v_3))$ and $\phi(v_2)\cdot e_1=-1$. We then set
\[
\phi'(v_+) = e_1 + e_{n+1} + \cdots + e_{n+w(v_+)-1}.
\]
Since the screw is clean, $\phi'$ gives the correct intersection pairing between $v_+$ and all other vertices, and hence $\phi'$ is indeed an embedding of $\Gamma'$. As this embedding must be standard, the embedding $\phi$ restricted to the outside part must also be standard.

If there are two screws, then when forming $\Gamma'$, instead of only attaching a new vertex $v_+$ to $v_2$, we also attach a new vertex $v_{++}$ to $v_{k-2}$, and construct an embedding
\[
\phi' \colon V(\Gamma')\rightarrow\Z^{n+w(v_+)+w(v_{++})-2}
\]
in the same way as above. The same argument then shows that the outside part is embedded standardly.
\end{proof}

\section{Theorem \ref{thm:main} from the Working Conditions}\label{sec:lattice_analysis}
Throughout Section~\ref{sec:lattice_analysis}, we assume that $\Gamma$ satisfies the Working Conditions and has at least one bad vertex, and we fix an embedding $\phi \colon V(\Gamma) \to \Z^n$. The aim of this section is to show that $\Gamma$ is embedded either standardly or semi-standardly. A key step will be to show that the supports of vertices in the bad part are disjoint from the supports of vertices outside the extended bad part.

\begin{lem}\label{lem:GAP check}
The extended bad part is embedded either standardly or semi-standardly.
\end{lem}

\begin{proof}
While it is possible to analyze each possibility by hand, this paper would be more than 100 pages long if we do that. Instead, we use some GAP~\cite{GAP} code to quickly analyze all possible embeddings. The details are in Appendix \ref{computer data}.
\end{proof}

\begin{lem}\label{lem:unclean criteria}
Let $v_0,\ldots,v_k$ be the extended bad part, and suppose that $v_0$ is not a leaf in $\Gamma$. If the screw connecting $v_2$ and $v_3$ is not clean, then there exists a vertex $v$ such that $w(v)\geq 4$ and
\[
|\supp(\phi(v))\cap\supp(\phi(v_1))|
=|\supp(\phi(v))\cap\supp(\phi(v_2))|
=|\supp(\phi(v))\cap\supp(\phi(v_3))|=2.
\]
\end{lem}

\begin{proof}
By the definition of the extended bad part, we have $w(v_1)=w(v_2)=2$. By Lemma~\ref{lem:GAP check}, we may write
\[
\phi(v_1)=e_1-e_2 \qquad\text{ and }\qquad \phi(v_2)=e_2-e_3.
\]
Then $e_3$ is the screw, so $e_3\in\supp(\phi(v_3))$ and $e_1,e_2\notin\supp(\phi(v_3))$.

By the definition of a clean screw, there exists a vertex $v\neq v_2,v_3$ such that $e_3\in\supp(\phi(v))$. By Lemma~\ref{lem:GAP check}, $v$ cannot lie in the extended bad part. Hence $\phi(v)\cdot\phi(v_i)=0$ for $i=1,2,3$, and therefore $|\supp(\phi(v))\cap\supp(\phi(v_i))|$ must be even for each $i=1,2,3$.

From $\phi(v)\cdot\phi(v_3)=0$ and the fact that $e_3\in\supp(\phi(v))\cap\supp(\phi(v_3))$, we deduce that there exists some $e_4\neq e_1,e_2,e_3$ such that
\[
e_4\in\supp(\phi(v))\cap\supp(\phi(v_3)).
\]
From $\phi(v)\cdot\phi(v_2)=0$ and the fact that $e_3\in\supp(\phi(v))\cap\supp(\phi(v_2))$, we conclude that $e_2\in\supp(\phi(v))$. Finally, from $\phi(v)\cdot\phi(v_1)=0$ and $e_2\in\supp(\phi(v))\cap\supp(\phi(v_1))$, we conclude that $e_1\in\supp(\phi(v))$.

Thus $e_1,e_2,e_3,e_4\in\supp(\phi(v))$, so $w(v)\geq 4$, and we have
\[
|\supp(\phi(v))\cap\supp(\phi(v_i))|=2 \quad\text{for } i=1,2,3,
\]
as claimed.
\end{proof}

\begin{cor}\label{cor:large bad}
If there exists a bad vertex $x$ with $w'(x)\geq 1$, then the outside part of $\Gamma$ is embedded standardly.
\end{cor}

\begin{proof}
By Working Conditions~\ref{it:largeweight} and~\ref{it:bad1}, together with Lemma~\ref{lem:GAP check}, the vertex $v$ appearing in the statement of Lemma~\ref{lem:unclean criteria} cannot exist. Hence all screws are clean, and therefore, by Lemma~\ref{lem:unscrew}, the outside part of $\Gamma$ is embedded standardly.
\end{proof}

\begin{lem}\label{lem:semi}
Let $x$ be a bad vertex and suppose the extended bad part is embedded semi-standardly at $x$. If $v\neq x$ and $|\supp(\phi(v))\cap\supp(\phi(x))|\neq 0$, then $v$ is adjacent to $x$.
\end{lem}

\begin{proof}
Suppose $v\neq x$, $|\supp(\phi(v))\cap\supp(\phi(x))|\neq 0$, and $v$ is not adjacent to $x$. We will arrive at a contradiction.

Let $u_1$ and $u_2$ be vertices adjacent to $x$, and choose basis vectors $e_1,e_2$ such that
\[
e_1\cdot\phi(x)=-1,\quad e_1\cdot\phi(u_1)=1,\quad e_1\cdot\phi(u_2)=1,\quad
e_2\cdot\phi(u_1)=1,\quad e_2\cdot\phi(u_2)=-1.
\] 
By Lemma~\ref{lem:GAP check}, $v$ is not in the extended bad part.

\medskip\noindent\textbf{Case 1:}{ $e_1\in\supp(\phi(v))$.}
Since $\phi(v)\cdot\phi(x)=0$, there exists another basis vector
$e_0\in\supp(\phi(x))\cap\supp(\phi(v))$.  Write
\[
  \phi(v) = a e_0 + b e_1 + c e_2 + z,
\]
where $z\cdot e_i = 0$ for $i=0,1,2$.

If $c\neq 0$, by Working Condition \ref{it:largeweight}, $b,c\in\{1,-1\}$, and hence one of $|b+c|$ and $|b-c|$ must be 2. Therefore, to achieve $\phi(v)\cdot\phi(u_1)=\phi(v)\cdot\phi(u_2)$, we need two more basis vectors that are in $\supp(\phi(v))\cap\supp(\phi(u_i))$ for some $i=1,2$. Therefore, we have $w(v)\geq 5$ and $w(u_i)\geq 4$ for some $i=1,2$, contradicting with Working Condition \ref{it:largeweight}.

Hence, $c$ must be 0. To achieve $\phi(x)\cdot\phi(u_1)=\phi(x)\cdot\phi(u_2)=0$, there are some new basis vectors $e_3\in\supp(\phi(u_1))\cap\supp(\phi(v))$ and $e_4\in\supp(\phi(u_2))\cap\supp(\phi(v))$. Now we have $w(v)\geq 4$, and therefore by Working Condition \ref{it:largeweight}, $w'(u_1),w'(u_2)\leq 1$. Hence, $\supp(u_1)=\{e_1,e_2,e_3\}$, $\supp(u_2)=\{e_1,e_2,e_4\}$, and $u_1,u_2$ are not leaves. Let $u_0$ be adjacent to $u_1$, $u_3$ be adjacent to $u_2$, and $x\neq u_0,u_3$. In that case, $e_3\in\supp(u_0)$, and $e_4\in\supp(u_3)$.

By the definition of extended bad part, $u_0,u_3$ and all their neighbors are in the extended bad part. Since $v$ is not in the extended bad part, we have $\phi(v)\cdot\phi(u_0)=\phi(v)\cdot\phi(u_3)=0$. Hence, there is some basis vectors $e_5\in\supp(\phi(u_0))\cap\supp(\phi(v))$ and $e_6\in\supp(\phi(u_3))\cap\supp(\phi(v))$, such that $e_5\neq e_3$ and $e_6\neq e_4$. From the assumption that the extended bad part is embedded semi-standardly at $x$, we know that $e_0,\ldots,e_6$ are all different. Therefore, $w(v)\geq 6$, contradicting with Working Condition \ref{it:largeweight}.

\medskip\noindent\textbf{Case 2:}{ $e_1\notin\supp(\phi(v))$.}
Since $\phi(x)\cdot\phi(v)=0$, there are some new basis vectors $e_{-1},e_0\in\supp(\phi(x))\cap\supp(\phi(v))$. Hence, $w(x)\geq 3$.

By Working Conditions \ref{it:largeweight} and \ref{it:bad1}, apart from potentially $x$, all vertices in $\Gamma$ have adjusted weight 0 or 1. Hence, $w'(v)\leq 1$. Also, by Corollary \ref{cor:large bad}, the subgraph induced by $v$ and its neighbor(s) is embedded standardly.

Since $w'(v)\leq 1$, either $w(v)=2$ and has at least one neighbor, or $w(v)=3$ and has two neighbors. Since the subgraph induced by $v$ and its neighbor(s) is embedded standardly, there is at most one basis vector in the support of $\phi(v)$ that is not in the support of any of $v$'s neighbor(s). Hence, we can let $u$ be adjacent to $v$ and $e_0\in\supp(u)$. If $e_1\in\supp(\phi(u))$, we can apply Case 1 on $u$ instead of $v$ to arrive at a contradiction. Hence, $e_1\notin\supp(\phi(u))$. Also, because the subgraph induced by $v$ and its neighbor(s) is embedded standardly, given that we already have $e_0$ being in $\supp(\phi(u))\cap\supp(\phi(v))$, $e_{-1}$ cannot be in $\supp(\phi(u))$.

Since $v$ is not in the extended bad part, $u$ cannot be adjacent to $x$. To achieve $\phi(u)\cdot\phi(x)=0$, since $e_0$ is in $\supp(\phi(u))$ but not $e_{-1},e_1$, we must have some new basis vector $e_{-2}\in\supp(\phi(u))\cap\supp(\phi(x))$. By Working Condition \ref{it:bad_part}, $w(x)\leq 4$. Hence, $\supp(\phi(x))=\{e_{-2},e_{-1},e_0,e_1\}$.

By Working Condition \ref{it:bad2}, $w'(v)$ and $w'(u)$ cannot be both 1. Therefore, at least one of $u,v$ has weight 2 and has another neighbor.

If $v$ has weight 2 and has another neighbor, let the neighbor be $y$. Since $\supp(\phi(v))=\{e_{-1},e_0\}$, exactly one of $e_{-1},e_0$ is in $\supp(\phi(y))$. Applying the argument of Case 1 on $y$ instead of $v$, we can conclude that $e_1\notin\supp(\phi(y))$. Hence, to achieve $\phi(y)\cdot\phi(x)=0$, we must have $e_{-2}\in\supp(\phi(y))$. However, since $e_{-2}\in\supp(\phi(u))$, that contradicts with the subgraph induced by $v$ and its neighbors being embedded standardly.

If $v$ has weight 2 and has another neighbor, we apply all the arguments we did with $v$ but on $u$ instead, and arrive at a contradiction in the same way.
\end{proof}

\begin{lem}\label{lem:weight0}
Suppose $x$ is a bad vertex with adjusted weight 0, and the extended bad part is embedded either standardly or semi-standardly at a bad vertex that is not $x$. If $v\neq x$ and $|\supp(\phi(v))\cap\supp(\phi(x))|\neq 0$, then $v$ is adjacent to $x$.
\end{lem}

\begin{proof}
Suppose $v\neq x$, $|\supp(\phi(v))\cap\supp(\phi(x))|\neq 0$, and $v$ is not adjacent to $x$. We want to arrive at a contradiction.

Since the extended bad part is embedded either standardly or semi-standardly at a bad vertex that is not $x$, we know that $v$ cannot come from the extended bad part.

Let $u_1,u_2$ be adjacent to $x$. Let $\phi(x)=e_1-e_2$, $e_1\in\supp(\phi(u_1))$, and $e_2\in\supp(\phi(u_2))$. Since $\phi(v)\cdot\phi(x)=0$, we have both $e_1,e_2$ being in $\supp(\phi(v))$. From $\phi(v)\cdot\phi(u_1)=\phi(v)\cdot\phi(u_2)=0$, and also the way the extended bad part is embedded, we conclude that there are distinct basis vectors $e_0,e_1,e_2,e_3\in\supp(\phi(v))$ such that $e_0,e_1\in\supp(\phi(u_1))$ and $e_2,e_3\in\supp(\phi(u_2))$.

Therefore, $w'(v)\geq 2$. By Working Conditions \ref{it:largeweight}, \ref{it:bad1}, and \ref{it:bad_part},the bad part must be either

$$\begin{tikzpicture}[xscale=1.0,yscale=1,baseline={(0,0)}]
    \node at (0,0.4) {$1$};
    \node at (1,0.4) {{$\red{\circled{0}}$}};
    \node at (2,0.4) {$1$};
    \node (A1) at (0,0) {$\circ$};
    \node (A2) at (1,0) {$\circ$};
    \node (A3) at (2,0) {$\circ$};
    \path (A1) edge [-] node [auto] {$\scriptstyle{}$} (A2);
    \path (A2) edge [-] node [auto] {$\scriptstyle{}$} (A3);
  \end{tikzpicture}
  \qquad\text{ or }\qquad
  \begin{tikzpicture}[xscale=1.0,yscale=1,baseline={(0,0)}]
    \node at (0,0.4) {$1$};
    \node at (1,0.4) {{$\red{\circled{0}}$}};
    \node at (2,0.4) {$1$};
    \node at (3,0.4) {{$\red{\circled{0}}$}};
    \node at (4,0.4) {$1$};
    \node (A1) at (0,0) {$\circ$};
    \node (A2) at (1,0) {$\circ$};
    \node (A3) at (2,0) {$\circ$};
    \node (A4) at (3,0) {$\circ$};
    \node (A5) at (4,0) {$\circ$};
    \path (A1) edge [-] node [auto] {$\scriptstyle{}$} (A2);
    \path (A2) edge [-] node [auto] {$\scriptstyle{}$} (A3);
    \path (A3) edge [-] node [auto] {$\scriptstyle{}$} (A4);
    \path (A4) edge [-] node [auto] {$\scriptstyle{}$} (A5);
  \end{tikzpicture}.$$
Moreover, since $w'(v)\ge 2$, Working Condition~\ref{it:largeweight} together with
Lemma~\ref{lem:unscrew} and Lemma~\ref{lem:unclean criteria} implies that the outside part is embedded
standardly. In particular, the
embedding is standard in a neighborhood of $v$.

By Working Condition \ref{it:largeweight}, either $w(v)=4$ and $v$ has at least one neighbor, or $w(v)=5$ and $v$ has two neighbors. In either case, since the embedding is standard locally around $v$, at least one of $e_0,e_1,e_2,e_3$ is in the support of a neighbor of $v$. Let $z_1$ be a neighbor of $v$. Since $\phi(z_1)\cdot\phi(x)=0$, either $e_1,e_2$ are both in $\supp(\phi(z_1))$ or none of $e_1,e_2$ are in $\supp(\phi(z_1))$. Since the embedding is standard locally around $v$, none of $e_1,e_2$ are in $\supp(\phi(z_1))$. Without loss of generality, let $e_3\in\supp(\phi(z_1))$. Since $\phi(z_1)\cdot\phi(u_2)=0$, there is some new basis vector $e_4\in\supp(\phi(z_1))\cap\supp(\phi(u_2))$.

Now, we have $e_2,e_3,e_4\in\supp(\phi(u_2))$. By Working Condition \ref{it:largeweight}, $w(u_2)=3$ and $u_2$ is not a leaf. Let $u_3$ be adjacent to $u_2$ and $u_3\neq x$.

Given our assumption on how the extended bad part is embedded, we know that exactly one of $e_3,e_4$ is in $\supp(\phi(u_3))$. Since $\phi(z_1)\cdot\phi(u_3)=0$ and both $e_3,e_4$ are in $\supp(\phi(z_1))$, there is some new basis vector $e_5\in\supp(\phi(z_1))\cap\supp(\phi(u_3))$. By Working Condition \ref{it:largeweight}, we have $\supp(\phi(z_1))=\{e_3,e_4,e_5\}$ and $z_1$ is not a leaf. Let $z_2$ be a neighbor of $z_1$ and $z_2\neq v$. Then, exactly one of $e_3,e_4,e_5$ is in $\supp(\phi(z_2))$. Since $\phi(u_2)\cdot\phi(z_2)=0$ and the outside part is embedded standardly, $e_3,e_4$ cannot be in $\supp(\phi(z_2))$. Hence, $e_5\in\supp(\phi(z_2))$.

Since $z_1$ is not in the extended bad part, $z_2$ and $u_3$ are not adjacent to each other. From $\phi(z_2)\cdot\phi(u_3)=0$, since $e_5\in\supp(\phi(u_3))\cap\supp(\phi(z_2))$, we can conclude that there is some new basis vector $e_6$ such that $e_6\in\supp(\phi(u_3))\cap\supp(\phi(z_2))$

Since $z_1$ is not in the bad part, it is not bad. Hence, $w'(z_2)=0$, and therefore $\supp(\phi(z_2))=\{e_5,e_6\}$.

Now, $e_4,e_5,e_6\in\supp(\phi(u_3))$. By Working Condition \ref{it:largeweight}, we know that $\supp(\phi(u_3))=\{e_4,e_5,e_6\}$ and $u_3$ is not a leaf. Let $u_4$ be a neighbor of $u_3$ and $u_4\neq u_2$. Given the assumption on how the extended bad part is embedded, we know that exactly one of $e_4,e_5,e_6$ is in $\supp(\phi(u_4))$.

Since $z_1$ is not in the bad part, $u_4$ and $z_2$ are not adjacent to each other. Hence, $e_5,e_6$ are not in $\supp(\phi(u_4))$. Therefore, $e_4\in\supp(\phi(u_4))$. Since $e_4\in\in\supp(\phi(u_2))$, given our assumption on how the extended bad part is embedded, $u_3$ has to be a bad vertex and the extended bad part is embedded semi-standardly at $u_3$. However, that violates Working Condition \ref{it:bad1}. Therefore, we arrived at a contradiction.

\begin{center}
\begin{tabular}{ c|c c c c c c c} 
\ & $e_0$ & $e_1$ & $e_2$ & $e_3$ & $e_4$ & $e_5$ & $e_6$ \\
\hline
$u_1$ & \checkmark & \checkmark & \ & \ & \ & \ & \ \\ 
$x$ & \ & \checkmark & \checkmark & \ & \ & \ & \ \\ 
$u_2$ & \ & \ & \checkmark & \checkmark & \checkmark & \ & \ \\ 
$u_3$ & \ & \ & \ & \ & \checkmark & \checkmark & \checkmark \\ 
$v$ & \checkmark & \checkmark & \checkmark & \checkmark & \ & \ & \ \\ 
$z_1$ & \ & \ & \ & \checkmark & \checkmark & \checkmark & \ \\ 
$z_2$ & \ & \ & \ & \ & \ & \checkmark & \checkmark \\ 
\end{tabular}
\end{center}

\end{proof}

\begin{lem}\label{lem:weight1}
Suppose $x$ is a bad vertex with adjusted weight 1, and the extended bad part is embedded either standardly or semi-standardly at a bad vertex that is not $x$. If $v\neq x$ and $|\supp(\phi(v))\cap\supp(\phi(x))|\neq 0$, then $v$ is adjacent to $x$.
\end{lem}

\begin{proof}
Suppose $v\neq x$, $|\supp(\phi(v))\cap\supp(\phi(x))|\neq 0$, and $v$ is not adjacent to $x$. We want to arrive at a contradiction.

Since the extended bad part is embedded either standardly or semi-standardly at a bad vertex that is not $x$, we know that $v$ cannot come from the extended bad part.

By Working Condition \ref{it:bad1}, we know that $w'(v)\leq 1$.

Let $u_1,u_2$ be adjacent to $x$. Let $\phi(x)=e_1+e_2-e_3$, $e_1\in\supp(\phi(u_1))$, and $e_3\in\supp(\phi(u_2))$.

Since $\phi(v)\cdot\phi(x)=0$, we know that exactly two of $e_1,e_2,e_3$ are in $\supp(\phi(v))$. If $e_1\in\supp(\phi(v))$, then since $\phi(v)\cdot\phi(u_1)=0$, there is some other basis vector in $\supp(\phi(v))\cap\supp(\phi(u_1))$. Similarly, if $e_3\in\supp(\phi(v))$, then there is some other basis vector in $\supp(\phi(v))\cap\supp(\phi(u_2))$. If both $e_1,e_3$ are in $\supp(\phi(v))$, given our assumption on how the extended bad part is embedded, those two other basis vectors we just mentioned are distinct, and hence $w(v)\geq 4$, violating Working Condition \ref{it:bad1}.

Hence, $e_2$ and exactly one of $e_1,e_3$ are in $\supp(\phi(v))$. Without loss of generality, let $e_2,e_3\in\supp(\phi(v))$. Since $\phi(v)\cdot\phi(u_2)=0$, there is some new basis vector $e_4$ in $\supp(\phi(v))\cap\supp(\phi(u_2))$. Therefore, $w(v)\geq 3$. By Working Condition \ref{it:bad1}, we know that $w(v)=3$ and $v$ has two neighbors.

Let $z_1,z_2$ be the neighbors of $v$. By Corollary \ref{cor:large bad}, exactly one of $e_2,e_3,e_4$ is in $\supp(\phi(z_1))$, exactly one of $e_2,e_3,e_4$ is in $\supp(\phi(z_2))$, and $|\supp(\phi(z_1))\cap\supp(\phi(z_2))|=0$.

First, we show that $e_3$ cannot be in $\supp(\phi(z_1))$. If $e_3\in\supp(\phi(z_1))$, then $e_2\notin\supp(\phi(z_1))$, and hence $e_1\in\supp(\phi(z_1))$. From $\phi(z_1)\cdot\phi(u_1)=0$ and $\phi(z_1)\cdot\phi(u_2)=0$, we know that there are at least two other basis vectors in $\supp(\phi(z_1))$, and hence $w(z_1)\geq 4$, violating Working Condition \ref{it:bad1}.

Hence, $e_3$ cannot be in $\supp(\phi(z_1))$. Similarly, $e_3$ cannot be in $\supp(\phi(z_2))$. Up to relabeling, let $e_2\in\supp(\phi(z_1))$ and $e_4\in\supp(\phi(z_2))$. Since $\phi(z_1)\cdot\phi(x)=0$, we must have $e_1\in\supp(\phi(z_1))$. Then, from $\phi(z_1)\cdot\phi(u_1)=0$, we know that there is some new basis vector $e_0$ in $\supp(\phi(z_1))\cap\supp(\phi(u_1))$. Therefore, $w(z_1)\geq 3$, and hence $w'(z_1)>0$.

Since $v$ is not in the bad part, it is not bad. Hence, $w'(z_2)=0$. Therefore, $w(z_2)=2$, and $z_2$ is not a leaf. Let $z_3$ be adjacent to $z_2$ and $z_3\neq v$. Since $z_2$ is also not in the bad part, $w'(z_3)=0$ as well, and therefore $w(z_3)=2$ and $z_3$ is not a leaf. Let $z_4$ be adjacent to $z_2$ and $z_4\neq z_2$.

Since $\phi(z_2)\cdot\phi(u_2)=0$, apart from $e_4$, there is some other basis vector $e_5\in\supp(\phi(z_2))\cap\supp(\phi(u_2))$. Since $w(z_2)=2$, we have $\supp(\phi(z_2))=\{e_4,e_5\}$. By Corollary \ref{cor:large bad}, $e_5$ is in $\supp(\phi(z_3))$, while $e_3,e_4$ are not in $\supp(\phi(z_3))$.

Given our assumption on how the extended bad part is embedded, we know that $z_3$ is not adjacent to $u_2$. Hence, apart from $e_5$, there is some other basis vector $e_6\in\supp(\phi(z_2))\cap\supp(\phi(u_2))$.

By Working Condition \ref{it:bad_part} (or \ref{it:weight3condition}), we know that $w(u_2)\leq 4$. Hence, $\supp(\phi(u_2))=\{e_3,e_4,e_5,e_6\}$. By Corollary \ref{cor:large bad}, $e_6\in\supp(\phi(z_4))$, while $e_3,e_4,e_5\notin\supp(\phi(z_4))$. Hence, $\phi(u_2)\cdot\phi(z_4)\neq 0$. Therefore, $z_4$ is adjacent to $u_2$, contradicting with our assumption on how the extended bad part is embedded.

\begin{center}
\begin{tabular}{ c|c c c c c c c} 
\ & $e_0$ & $e_1$ & $e_2$ & $e_3$ & $e_4$ & $e_5$ & $e_6$ \\
\hline
$u_1$ & \checkmark & \checkmark & \ & \ & \ & \ & \ \\ 
$x$ & \ & \checkmark & \checkmark & \checkmark & \ & \ & \ \\ 
$u_2$ & \ & \ & \ & \checkmark & \checkmark & \checkmark & \checkmark \\ 
$z_1$ & \checkmark & \checkmark & \checkmark & \ & \ & \ & \ \\ 
$v$ & \ & \ & \checkmark & \checkmark & \checkmark & \ & \ \\ 
$z_2$ & \ & \ & \ & \ & \checkmark & \checkmark & \ \\ 
$z_3$ & \ & \ & \ & \ & \ & \checkmark & \checkmark \\ 
$z_4$ & \ & \ & \ & \ & \ & \ & \checkmark \\ 
\end{tabular}
\end{center}

\end{proof}

\begin{lem}\label{lem:weight2}
Suppose $x$ is a bad vertex with adjusted weight 2, and the extended bad part is embedded either standardly or semi-standardly at a bad vertex that is not $x$. If $v\neq x$ and $|\supp(\phi(v))\cap\supp(\phi(x))|\neq 0$, then $v$ is adjacent to $x$.
\end{lem}

\begin{proof}
Suppose $v\neq x$, $|\supp(\phi(v))\cap\supp(\phi(x))|\neq 0$, and $v$ is not adjacent to $x$. We want to arrive at a contradiction.

Since the extended bad part is embedded either standardly or semi-standardly at a bad vertex that is not $x$, we know that $v$ cannot come from the extended bad part.

Let $u_1,u_2$ be adjacent to $x$. Let $\phi(x)=e_1+e_2+e_3-e_4$, $e_1\in\supp(\phi(u_1))$, and $e_4\in\supp(\phi(u_2))$. By Working Condition \ref{it:bad1}, we know that $w'(v)\leq 1$, and hence $w(v)\leq 3$. Since $\phi(v)\cdot\phi(x)=0$, we know that exactly two of $e_1,e_2,e_3,e_4$ are in $\supp(\phi(v))$.

\medskip\noindent\textbf{Case 1: }{$e_1\in\supp(\phi(v))$ or $e_4\in\supp(\phi(v))$.}
Without loss of generality, let $e_1\in\supp(\phi(v))$. Since $\phi(v)\cdot\phi(u_1)=0$, there is some new basis vector $e_0$ in $\supp(\phi(v))\cap\supp(\phi(u_1))$. If $e_4$ is also in $\supp(\phi(v))$, since $\phi(v)\cdot\phi(u_2)=0$, there is another basis vector in $\supp(\phi(v))\cap\supp(\phi(u_2))$, and hence $w(v)\geq 4$, contradicting Working Condition \ref{it:largeweight}. Hence, $e_4$ cannot be in $\supp(\phi(v))$, and exactly one of $e_2,e_3$ is in $\supp(\phi(v))$. Without loss of generality, let $e_2\in\supp(\phi(v))$.

By Working Condition \ref{it:bad1}, we must have $\supp(\phi(v))=\{e_0,e_1,e_2\}$ and $v$ has two neighbors. By Corollary \ref{cor:large bad}, at least one of $v$'s neighbor has a support that contains exactly one of $e_1,e_2$. Let $z_1$ be adjacent to $v$ and exactly one of $e_1,e_2$ is in $\supp(\phi(z_1))$. Since $\phi(z_1)\cdot\phi(x)=0$, exactly one of $e_3,e_4$ is in $\supp(\phi(z_1))$.

By Working Condition \ref{it:bad2}, we know that $w'(z_1)=0$, and hence $w(z_1)=2$ and is not a leaf. Let $z_2$ be adjacent to $z_1$ and $z_2\neq v$. Since $z_1$ is not in the bad part, it is not bad, and hence $w'(z_2)=0$, and therefore $w(z_2)=2$.

Since $w(z_1)=2$, $\supp(\phi(z_1))$ does not contain any elements outside of $e_1,e_2,e_3,e_4$. Since $\phi(z_1)\cdot\phi(u_2)=0$, $e_4$ cannot be in $\supp(\phi(z_1))$. Therefore, $e_3\in\supp(\phi(z_1))$.

By Corollary \ref{cor:large bad}, $e_3\in\supp(\phi(z_2))$, while $e_1,e_2\notin\supp(\phi(z_2))$. Since $\phi(z_2)\cdot\phi(x)=0$, we have $e_4\in\supp(\phi(z_2))$. However, since $w(z_2)=2$, $e_3,e_4$ are the only elements of $\supp(\phi(z_2))$, and hence $\phi(z_2)\cdot\phi(u_2)\neq 0$. So, $z_2$ is adjacent to $u_2$, contradicting with our assumption on how the extended bad part is embedded.

\begin{center}
\begin{tabular}{ c|c c c c c} 
\ & $e_0$ & $e_1$ & $e_2$ & $e_3$ & $e_4$ \\
\hline
$u_1$ & \checkmark & \checkmark & \ & \ & \ \\ 
$x$ & \ & \checkmark & \checkmark & \checkmark & \checkmark \\ 
$u_2$ & \ & \ & \ & \ & \checkmark \\  
$v$ & \checkmark & \checkmark & \checkmark & \ & \ \\ 
$z_1$ & \ & \ & \checkmark & \checkmark & \ \\
$z_2$ & \ & \ & \ & \checkmark & \checkmark \\ 
\end{tabular}
\end{center}

\medskip\noindent\textbf{Case 2: }{$e_2,e_3\in\supp(\phi(v))$.}
By Working Condition \ref{it:largeweight}, either $w(v)=2$ and $v$ has at least one neighbor, or $w(v)=3$ and $v$ has two neighbors. Either way, by Corollary \ref{cor:large bad}, $\supp(\phi(v))$ has at most one element that is not in the support of any of its neighbor(s). Let $z$ be a neighbor of $x$ such that exactly one of $e_2,e_3$ is in $\supp(\phi(z))$. Since $\phi(z)\cdot\phi(x)=0$, $e_1$ or $e_4$ is also in $\supp(\phi(z))$. Hence, we arrive at Case 1 with $z$ replacing $v$. Applying arguments we did with $v$ but on $z$ instead, we arrive at a contradiction.
\end{proof}

\begin{cor}\label{cor:bad clean}
Suppose $x$ is a bad vertex. If $v\neq x$ and $|\supp(\phi(v))\cap\supp(\phi(x))|\neq 0$, then $v$ is adjacent to $x$.
\end{cor}

\begin{proof}
Lemmas \ref{lem:semi}, \ref{lem:weight0}, \ref{lem:weight1}, and \ref{lem:weight2} covered all cases that can arise from Lemma \ref{lem:GAP check}.
\end{proof}

\begin{lem}\label{lem:sandwiched}
Suppose $x$ is a vertex that is not bad and is adjacent to two bad vertices. If $v\neq x$ and $|\supp(\phi(v))\cap\supp(\phi(x))|\neq 0$, then $v$ is in the bad part of $\Gamma$.
\end{lem}

\begin{proof}
Suppose $v\neq x$, $|\supp(\phi(v))\cap\supp(\phi(x))|\neq 0$, and $v$ is not in the bad part. We want to arrive at a contradiction.

By Working Condition \ref{it:bad_part}, we know that $w(x)$ must be 3. Hence, $|\supp(\phi(x))|=3$. By Lemma \ref{lem:GAP check}, two elements of $\supp(\phi(x))$ are also present in the supports of $x$'s neighbors. Since $x$'s neighbors are bad, by Corollary \ref{cor:bad clean}, elements of $\supp(\phi(x))$ that are also present in the supports of $x$'s neighbors cannot be in $\supp(\phi(v))$. Therefore, $|\supp(\phi(v))\cap\supp(\phi(x))|=1$, and hence $v$ and $x$ are adjacent to each other, which is impossible.
\end{proof}

\begin{lem}\label{lem:folded basis}
Let $x$ be a bad vertex, and $u_1,u_2$ be its neighbors. Suppose the extended bad part is embedded semi-standardly at $x$. Let $e_0$ be the basis vector that is in $\supp(\phi(u_1))\cap\supp(\phi(u_2))$ but not in $\supp(\phi(x))$. If $v$ is a vertex such that $v\neq u_1,u_2$, then $e_0\notin\supp(\phi(v))$.
\end{lem}

\begin{proof}
Suppose $v\neq u_1,u_2$ and $e_0\in\supp(\phi(v))$. We want to arrive at a contradiction.

Let $e_{-1}$ be the basis vector in $\supp(\phi(u_1))\cap\supp(\phi(x))\cap\supp(\phi(u_2))$.

From our assumption on how the extended bad part is embedded, $v$ cannot be in the extended bad part. Hence, $v$ is not adjacent to $u_1,u_2$. So, there exists some other basis vectors $e_1\in\supp(\phi(u_1))\cap\supp(\phi(v))$ and $e_2\in\supp(\phi(u_2))\cap\supp(\phi(v))$.

Now we have $w(u_1),w(u_2)\geq 3$. By Working Condition \ref{it:largeweight}, at least one of $u_1,u_2$ has weight 3 and not being a leaf. Without loss of generality, let $w(u_2)=3$, $u_3$ be a neighbor of $u_2$, and $u_3\neq x$. From our assumption on how the extended bad part is embedded, $\supp(\phi(u_3))\cap\supp(\phi(u_2))=\{e_2\}$. Since $\phi(v)\cdot\phi(u_3)=0$, there is some new $e_3\in\supp(\phi(v))\cap\supp(\phi(u_3))$. Therefore, $w(v)\geq 4$. By Working Condition \ref{it:largeweight} again, there is some $u_0$ such that $\supp(\phi(u_0))\cap\supp(\phi(u_1))=\{e_1\}$. Since $\phi(v)\cdot\phi(u_0)=0$, there is some new $e_4\in\supp(\phi(v))\cap\supp(\phi(u_0))$. Hence, $w(v)\geq 5$. By Working Condition \ref{it:largeweight}, $w'(v)=3$ and $w(v)=5$. By Working Condition \ref{it:weight3condition}, $w'(u_0)=0$. Hence, $u_0$ is not a leaf. Let $u_{-1}$ be adjacent to $u_1$ and $u_{-1}\neq u_1$. Given our assumption on how the extended bad part is embedded, we have $\supp(\phi(u_{-1}))\cap\supp(\phi(u_0))=\{e_4\}$, and $e_0,e_1,e_2,e_3\notin\supp(\phi(u_{-1}))$. This implies that $u_{-1}$ is adjacent to $v$, and therefore $v$ lies in the extended bad part, contradicting with how the extended bad part is embedded.

\begin{center}
\begin{tabular}{ c|c c c c c c} 
\ & $e_{-1}$ & $e_0$ & $e_1$ & $e_2$ & $e_3$ & $e_4$\\
\hline
$u_{-1}$ & \ & \ & \ & \ & \ & \checkmark \\ 
$u_0$ & \ & \ & \checkmark & \ & \ & \checkmark \\ 
$u_1$ & \checkmark & \checkmark & \checkmark & \ & \ & \ \\ 
$x$ & \checkmark & \ & \ & \ & \ \\ 
$u_2$ & \checkmark & \checkmark & \ & \checkmark & \ & \ \\  
$u_3$ & \ & \ & \ & \checkmark & \checkmark & \ \\  
$v$ & \ & \checkmark & \checkmark & \checkmark & \checkmark & \checkmark \\ 
\end{tabular}
\end{center}

\end{proof}

Recall that throughout this section we assume that $\Gamma$ satisfies the Working Conditions and has at least one bad vertex, and that we have fixed an embedding $\phi \colon V(\Gamma) \to \Z^n$. We are finally ready to prove that this embedding is either standard or semi-standard. For the sake of simplicity, we introduce two technical definitions that will be used in the proof.

\begin{defn}
The \emph{inner bad part} of $\Gamma$ is the subgraph induced by the set of all vertices of $\Gamma$ that are either bad or have two bad vertices as neighbors. Equivalently, the \emph{inner bad part} is the bad part with the two ends removed.
\end{defn}

\begin{defn}
A vertex of $\Gamma$ is a \emph{neck vertex} if it not in the inner bad part and is adjacent to a vertex in the inner bad part. Equivalently, \emph{neck vertices} are the two ends of the bad part.
\end{defn}

Recall that we are assuming that $\Gamma$ satisfies the Working Conditions and contains at least one bad vertex:

\begin{lem}\label{lem:final lemma}
$\Gamma$ is embedded either standardly or semi-standardly.
\end{lem}

\begin{proof}
Let $S=\{e_1,\ldots,e_n\}$ be the set of basis vectors. Let $S'$ be the set of basis vectors that are in the support of at least one vertex in the inner bad part. Let $\Gamma'$ be the graph obtained from $\Gamma$ by deleting the vertices in the inner bad part and reducing the weight of each neck vertex by 1.

By construction, the adjusted weights of vertices of $\Gamma'$ as an individual graph is the same as the adjusted weights of vertices in $\Gamma'$ as a subgraph of $\Gamma$. Hence, as an individual graph, $\Gamma'$ satisfies the Working Conditions. Since $\Gamma'$ has no bad vertices, it also satisfies the Working Conditions in \cite{definite}. Therefore, all embeddings of $\Gamma'$ are standard.

From Lemma \ref{lem:GAP check}, we know that the extended bad part is embedded either standardly or semi-standardly. Let $u,v$ be the two neck vertices in $\Gamma$. To avoid confusion, we write $u',v'$ to refer to $u,v$ being viewed as vertices in $\Gamma'$. So, $w(u')=w(u)-1$ and $w(v')=w(v)-1$.

\medskip\noindent\textbf{Case 1:}{ There is only one bad vertex and the extended bad part is embedded semi-standardly.}
Then $u$ and $v$ are both adjacent to the bad vertex, and $\supp(\phi(u))\cap\supp(\phi(v))=\{e_i,e_o\}$ for some $e_i\in S'$ and $e_o\in S\smallsetminus S'$. Up to a sign change automorphism on the basis vectors, we let $u=z_u+e_i+e_o$ and $v=z_v+e_i-e_o$ for some $z_u,z_v$ satisfying
$$z_u\cdot e_i=z_u\cdot e_o=z_v\cdot e_i=z_v\cdot e_o=0.$$
Consider the embedding $\phi'\colon V(\Gamma')\rightarrow\Z^n$ as follows. For any vertex $t\neq u',v'$, we set $\phi'(t)=\phi(t)$. We also set $\phi'(u')=z_u+e_o$ and $\phi'(v')=z_v+e_i$.

By Lemmas \ref{lem:semi} and \ref{lem:folded basis}, no vertices in $\Gamma'$ apart from $u',v'$ has a support that contains $e_i$ or $e_o$. Hence, changing how the images of $u$ and $v$ in $\Z^n$ interact with $e_i$ and $e_o$ does not affect their pairing with vertices not in the bad part. Therefore, $\phi'$ is indeed an embedding of $\Gamma'$. Since any embedding of $\Gamma'$ must be standard, $\phi'$ is standard.
Therefore, by Lemmas \ref{lem:semi} and \ref{lem:folded basis}, we see that $\phi$ embeds $\Gamma$ semi-standardly by checking the pairing requirements in the definition of semi-standard.

\medskip\noindent\textbf{Case 2:}{ There are at least two bad vertices, or the extended bad part is embedded standardly.}
Then $|\supp(\phi(u))\cap\supp(\phi(v))|=0$. Up to a sign change automorphism on the basis vectors, let $u=z_u+e_u$ and $v=z_v+e_v$ for some $e_u\in S'$ and $e_v\in S'$. Since $u$ and $v$ are neck vertices, $e_u$ and $e_v$ are in the supports of bad vertices. Consider the embedding $\phi'\colon V(\Gamma')\rightarrow\Z^n$ as follows. For any vertex $t\neq u',v'$, we set $\phi'(t)=\phi(t)$. We also set $\phi'(u')=z_u$ and $\phi'(v')=z_v$. By Corollary \ref{cor:bad clean}, changing how the images of $u$ and $v$ in $\Z^n$ interact with $e_u$ and $e_v$ does not affect their pairing with vertices not in the bad part. Therefore, $\phi'$ is indeed an embedding of $\Gamma'$. Since any embedding of $\Gamma'$ must be standard, $\phi'$ is standard.

Therefore, using Corollary \ref{cor:bad clean} and Lemmas \ref{lem:sandwiched} and \ref{lem:folded basis}, by checking the pairing requirements in the definition of standard and semi-standard, we see that $\phi$ embeds $\Gamma$ standardly if the extended bad part is embedded standardly, and $\phi$ embeds $\Gamma$ semi-standardly if the extended bad part is embedded semi-standardly.
\end{proof}

Finally, we prove that Theorem \ref{thm:main}:

% \ref{it:combinatorial} implies Theorem \ref{thm:main}\ref{it:intersectionform}.

\begin{proof}[Proof of Theorem~\ref{thm:main}]
The implication \ref{it:min_filling} $\Rightarrow$ \ref{it:submanifold} was established in Lemma~\ref{thm:1implies4}, and \ref{it:submanifold} $\Rightarrow$ \ref{it:combinatorial} in Lemma~\ref{lem:4implies3}, while \ref{it:intersectionform} $\Rightarrow$ \ref{it:min_filling} follows directly from the definition. It therefore remains to show the implication \ref{it:combinatorial} $\Rightarrow$ \ref{it:intersectionform}.

Assume that a connected sum of lens spaces $L$ satisfies \ref{it:combinatorial}. By Lemma~\ref{lem:working conditions proof}, the dual plumbing graph $\Gamma$ satisfies the Working Conditions. Let $$\phi \colon V(\Gamma) \rightarrow \mathbb{Z}^n$$ be an embedding of $\Gamma$. If $\Gamma$ has no bad vertex, then by Remark~\ref{rem:comparintworkingcondition} and \cite[Lemma~3.6]{definite}, the embedding $\phi$ is standard. On the other hand, if $\Gamma$ has at least one bad vertex, then by Lemma~\ref{lem:working conditions proof}, the embedding $\phi$ is either standard or semi-standard.

% \ref{it:intersectionform}

% the orthogonal complement of a standard embedding coincide 

% The standard embeddings of a linear graph and its dual are orthogonal complements of each other \textcolor{red}{\AF{is there a citation for this?}}.

As a standard embedding is unique up to automorphisms of $\mathbb{Z}^n$~\cite[Remark~4.5]{definite}, and the embedding obtained by gluing $\natural_i X(p_i,q_i)$ and $\natural_i X(p_i,p_i-q_i)$ is standard (see, for instance, \cite{Lisca:2008-1,SSZ:2008,Starkston:2015}), we conclude that, in this case, the orthogonal complement of the embedding coincides with the intersection form $Q_{X(p_i,q_i)} \oplus \langle -1 \rangle^m$ for some nonnegative integer $m$. Therefore, to complete the proof it remains to consider the case when $\Gamma$ has at least one bad vertex and the embedding $\phi$ is semi-standard.

Recall from Remark~\ref{rem:badvertex} that each bad vertex in the dual plumbing graph $\Gamma$ corresponds to a copy of
\[
\begin{tikzpicture}[xscale=1.0,yscale=1,baseline={(0,0)}]
    \node at (1-0.1,0.4) {$-4$};
    \node (A1) at (1,0) {$\bullet$};
\end{tikzpicture}
\qquad \qquad \qquad
\begin{tikzpicture}[xscale=1.0,yscale=1,baseline={(0,0)}]
    \node at (1-0.1, .4) {$-3$};
    \node at (2-0.1, .4) {$-3$};
    \node (A_1) at (1, 0) {$\bullet$};
    \node (A_2) at (2, 0) {$\bullet$};
    \path (A_1) edge [-] node [auto] {$\scriptstyle{}$} (A_2);
\end{tikzpicture}
\qquad  \qquad \qquad
\begin{tikzpicture}[xscale=1.0,yscale=1,baseline={(0,0)}]
    \node at (1-0.1, .4) {$-3$};
    \node at (2-0.1, .4) {$-2$};
    \node at (3-0.1, .4) {$-3$};
    \node (A_1) at (1, 0) {$\bullet$};
    \node (A_2) at (2, 0) {$\bullet$};
    \node (A_3) at (3, 0) {$\bullet$};
    \path (A_1) edge [-] node [auto] {$\scriptstyle{}$} (A_2);
    \path (A_2) edge [-] node [auto] {$\scriptstyle{}$} (A_3);
\end{tikzpicture}
\]
in the original graph. For each case, we examine the orthogonal complement of the embedding $\phi$ and compare it with the result of blowing down the $-4$-sphere in the canonical negative-definite plumbing $\natural_i X(p_i,q_i)$. In each of the three configurations in the graph above, this $-4$-sphere is represented by the union of the spheres corresponding to the vertices of that configuration. Moreover, for the remainder of the proof we assume that $\phi$ uses all $n$ basis vectors, that is,
\[
\bigcup_{v \in V(\Gamma)} \supp(\phi(v)) = \{e_1, \dots, e_n\},
\]
where $\{e_1, \dots, e_n\}$ is an orthonormal basis of $\mathbb{Z}^n$. If this is not the case, the orthogonal complement is obtained from the ones computed below by adding a summand $\langle -1 \rangle^m$ for some nonnegative integer $m$. Recall that a Kirby diagram for a rational blowdown along a $-4$-sphere is obtained by replacing the $-4$–framed unknot with a 1–handle and a 2–handle, as shown in Figure~\ref{fig:c1} (see, e.g., \cite{Gompf-Stipsicz:1999}).

\vspace{.2cm}

\begin{figure}[htbp]
    \centering
    \includegraphics[width=0.4\linewidth]{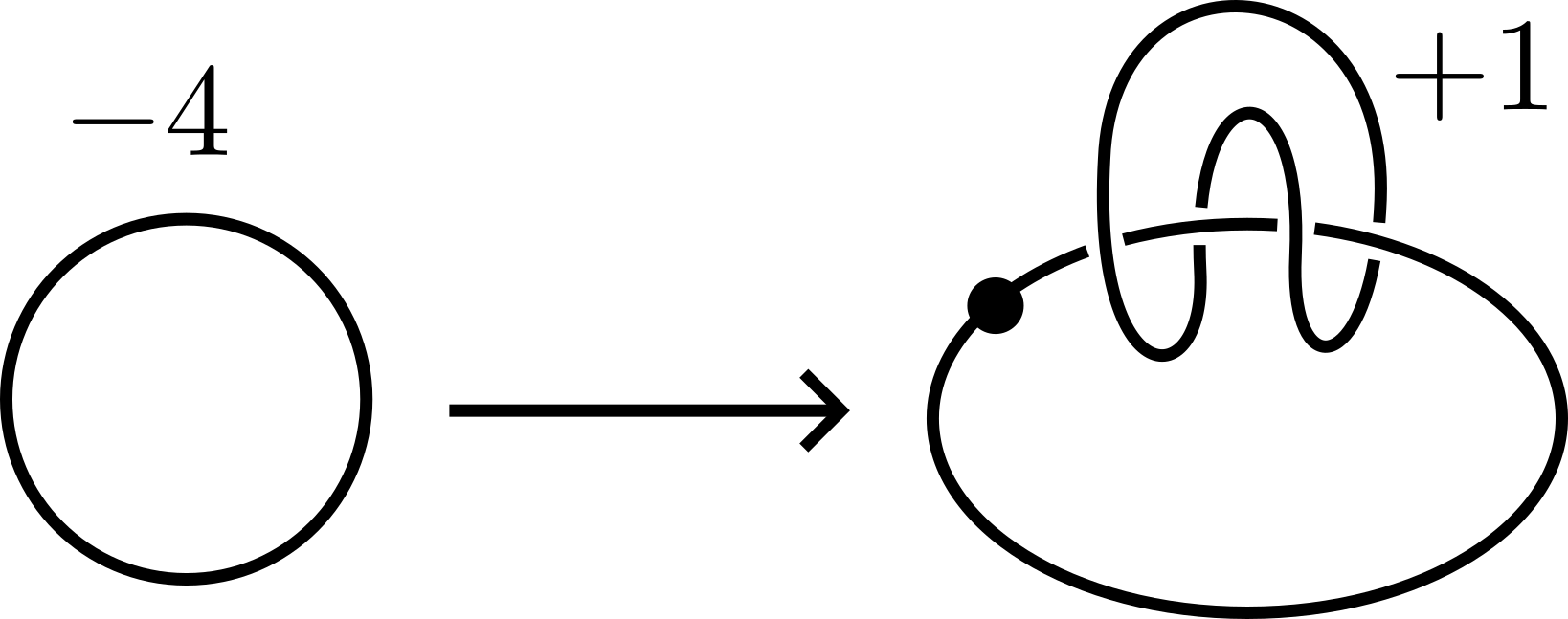}
    \caption{Rational blowdown.}
    \label{fig:c1}
\end{figure}

First, we consider what happens to the intersection form when we perform rational blowdown to the $-4$-sphere in 
$$\begin{tikzpicture}[xscale=1.0,yscale=1,baseline={(0,0)}]
    \node at (1,0.4) {$a_2$};
    \node at (2,0.4) {$a_1$};
    \node at (3-0.1,0.4) {$-4$};
    \node at (4,0.4) {$b_1$};
    \node at (5,0.4) {$b_2$};
    \node (A1) at (0,0) {$\cdots$};
    \node (A2) at (1,0) {$\bullet$};
    \node (A3) at (2,0) {$\bullet$};
    \node (A4) at (3,0) {$\bullet$};
    \node (A5) at (4,0) {$\bullet$};
    \node (A6) at (5,0) {$\bullet$};
    \node (A7) at (6,0) {$\cdots$};
    \path (A1) edge [-] node [auto] {$\scriptstyle{}$} (A2);
    \path (A2) edge [-] node [auto] {$\scriptstyle{}$} (A3);
    \path (A3) edge [-] node [auto] {$\scriptstyle{}$} (A4);
    \path (A4) edge [-] node [auto] {$\scriptstyle{}$} (A5);
    \path (A5) edge [-] node [auto] {$\scriptstyle{}$} (A6);
    \path (A6) edge [-] node [auto] {$\scriptstyle{}$} (A7);
  \end{tikzpicture}.$$
Following Figure~\ref{fig:c2}, we see that, locally in the matrix of the intersection form, we replace
\[
\begin{pmatrix}
a_2 & 1   & 0   & 0 \\
1   & a_1 & 1   & 0 \\
0   & 1   & -4  & 1 \\
0   & 0   & 1   & b_1
\end{pmatrix}
\qquad\text{ with }\qquad
\begin{pmatrix}
a_2      & 2        & 1 \\
2        & 1+4a_1   & 2a_1 \\
1        & 2a_1     & a_1 + b_1
\end{pmatrix}.
\]

\begin{figure}
    \centering
    \includegraphics[width=0.9\linewidth]{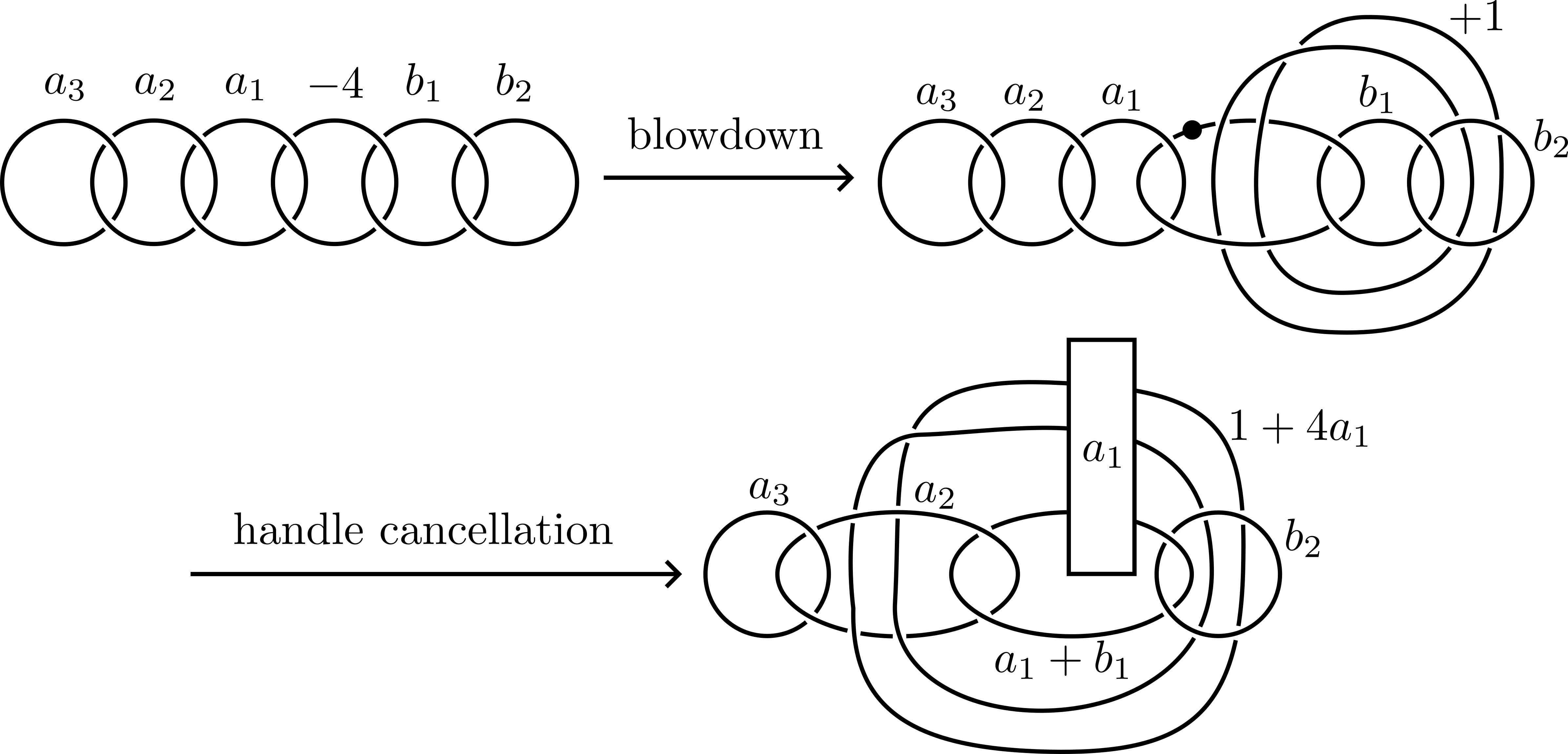}
    \caption{Blowdown for $[-4]$.}
    \label{fig:c2}
\end{figure}

Now, consider the following basis for the orthogonal complement of the semi-standard embedding associated to a bad vertex corresponding to the vertex with weight $-4$ in
$$\begin{tikzpicture}[xscale=1.0,yscale=1,baseline={(0,0)}]
    \node at (1,0.4) {$a_2$};
    \node at (2,0.4) {$a_1$};
    \node at (3-0.1,0.4) {$-4$};
    \node at (4,0.4) {$b_1$};
    \node at (5,0.4) {$b_2$};
    \node (A1) at (0,0) {$\cdots$};
    \node (A2) at (1,0) {$\bullet$};
    \node (A3) at (2,0) {$\bullet$};
    \node (A4) at (3,0) {$\bullet$};
    \node (A5) at (4,0) {$\bullet$};
    \node (A6) at (5,0) {$\bullet$};
    \node (A7) at (6,0) {$\cdots$};
    \path (A1) edge [-] node [auto] {$\scriptstyle{}$} (A2);
    \path (A2) edge [-] node [auto] {$\scriptstyle{}$} (A3);
    \path (A3) edge [-] node [auto] {$\scriptstyle{}$} (A4);
    \path (A4) edge [-] node [auto] {$\scriptstyle{}$} (A5);
    \path (A5) edge [-] node [auto] {$\scriptstyle{}$} (A6);
    \path (A6) edge [-] node [auto] {$\scriptstyle{}$} (A7);
  \end{tikzpicture}.$$
The embedding above the double horizontal line is the semi-standard embedding. Below the double horizontal line, it describes a basis of its orthogonal complement. If $a_1=-2$, the top row becomes the row above the bad vertex. Similar for the $b_1=-2$ case.
\begin{center}
\begin{tabular}{ccc|cccccc|ccc|cccccc|ccc} 
 & & & & & & & & & & & & & & & & & & & & \\ 
 \hline
 & $\cdots$ & 1 & 1 & -1 & & & & & & & & & & & & & & & & \tikzmark{ee} \\ 
 & & & & 1 & -1 & & & & & & & & & & & & & & & \\ 
 & & & & & & $\ddots$ & & & & & & & & & & & & & & \\ 
 & & & & & & & 1 & -1 & & & & & & & & & & & & \tikzmark{ff} \\ 
 \hline
 & & & & & & & & 1 & 1 & -1 & & & & & & & & & & \\ 
 & & & & & & & & & & 1 & -1 & & & & & & & & & \\
 & & & & & & & & & -1 & -1 & & 1 & & & & & & & & \\
 \hline
 & & & & & & & & & & & & -1 & 1 & & & & & & & \tikzmark{gg} \\
 & & & & & & & & & & & & & & $\ddots$ & & & & & & \\
 & & & & & & & & & & & & & & & -1 & 1 & & & & \\
 & & & & & & & & & & & & & & & & -1 & 1 & 1 & $\cdots$ & \tikzmark{hh} \\
 & & & \tikzmark{aa} & & & & & \tikzmark{bb} & & & & \tikzmark{cc} & & & & & \tikzmark{dd} & & & \\
 \hline
 \hline
 & $\cdots$ & 1 & -1 & & & & & & & & & & & & & & & & & \\
$v_1:$ & & & 2 & 2 & 2 & $\cdots$ & 2 & 2 & -1 & 1 & 1 & & & & & & & & & \\
$v_2:$ & & & 1 & 1 & 1 & $\cdots$ & 1 & 1 & & 1 & 1 & 1 & 1 & $\cdots$ & 1 & 1 & 1 & & & \\
 & & & & & & & & & & & & & & & & & -1 & 1 & $\cdots$ & \\
\begin{tikzpicture}[overlay,remember picture]
\draw[<->,black,thick] (pic cs:aa) -- (pic cs:bb) node[midway,above]{$-a_1-1$};
\draw[<->,black,thick] (pic cs:cc) -- (pic cs:dd) node[midway,above]{$-b_1-1$};
\draw[<->,black,thick] (pic cs:ee) -- (pic cs:ff) node[midway,right]{$-a_1-2$};
\draw[<->,black,thick] (pic cs:gg) -- (pic cs:hh) node[midway,right]{$-b_1-2$};
\end{tikzpicture}
\end{tabular}
\end{center}
We compute the intersection pairings of the elements labeled as $v_1,v_2$ in the table: 
\begin{align*}
\langle v_1, v_1 \rangle &= -4(-a_1 - 1) - 3 = 4a_1 + 1,\\
\langle v_1, v_2 \rangle &= -2(-a_1 - 1) - 2 = 2a_1,\\
\langle v_2, v_2 \rangle &= -(-a_1 - 1) - 2 - (-b_1 - 1) = a_1 + b_1.
\end{align*}
We see that the intersection pairing coincides with the one obtained from rational blowdown.

Next, we consider what happens to the intersection form when we perform rational blowdown with 
$$\begin{tikzpicture}[xscale=1.0,yscale=1,baseline={(0,0)}]
    \node at (1,0.4) {$a_2$};
    \node at (2,0.4) {$a_1$};
    \node at (3-0.1,0.4) {$-3$};
    \node at (4-0.1,0.4) {$-3$};
    \node at (5,0.4) {$b_1$};
    \node at (6,0.4) {$b_2$};
    \node (A1) at (0,0) {$\cdots$};
    \node (A2) at (1,0) {$\bullet$};
    \node (A3) at (2,0) {$\bullet$};
    \node (A4) at (3,0) {$\bullet$};
    \node (A5) at (4,0) {$\bullet$};
    \node (A6) at (5,0) {$\bullet$};
    \node (A7) at (6,0) {$\bullet$};
    \node (A8) at (7,0) {$\cdots$};
    \path (A1) edge [-] node [auto] {$\scriptstyle{}$} (A2);
    \path (A2) edge [-] node [auto] {$\scriptstyle{}$} (A3);
    \path (A3) edge [-] node [auto] {$\scriptstyle{}$} (A4);
    \path (A4) edge [-] node [auto] {$\scriptstyle{}$} (A5);
    \path (A5) edge [-] node [auto] {$\scriptstyle{}$} (A6);
    \path (A6) edge [-] node [auto] {$\scriptstyle{}$} (A7);
    \path (A7) edge [-] node [auto] {$\scriptstyle{}$} (A8);
  \end{tikzpicture}.$$
Following Figure~\ref{fig:c3}, we see that locally in the matrix representing the intersection form, we replace $$\begin{pmatrix}
a_2 & 1 & 0 & 0 & 0 \\
1 & a_1 & 1 & 0 & 0 \\
0 & 1 & -3 & 1 & 0 \\
0 & 0 & 1 & -3 & 1 \\
0 & 0 & 0 & 1 & b_1
\end{pmatrix} \qquad\text{ with }\qquad \begin{pmatrix}
a_2 & 2 & 2 & -1 \\
2 & 1+4a_1 & 4a_1 & -2a_1 \\
2 & 4a_1 & -3+4a_1 & 1-2a_1 \\
-1 & -2a_1 & 1-2a_1 & a_1+b_1 
\end{pmatrix}.$$

\begin{figure}[htbp]
    \centering
    \includegraphics[width=0.8\linewidth]{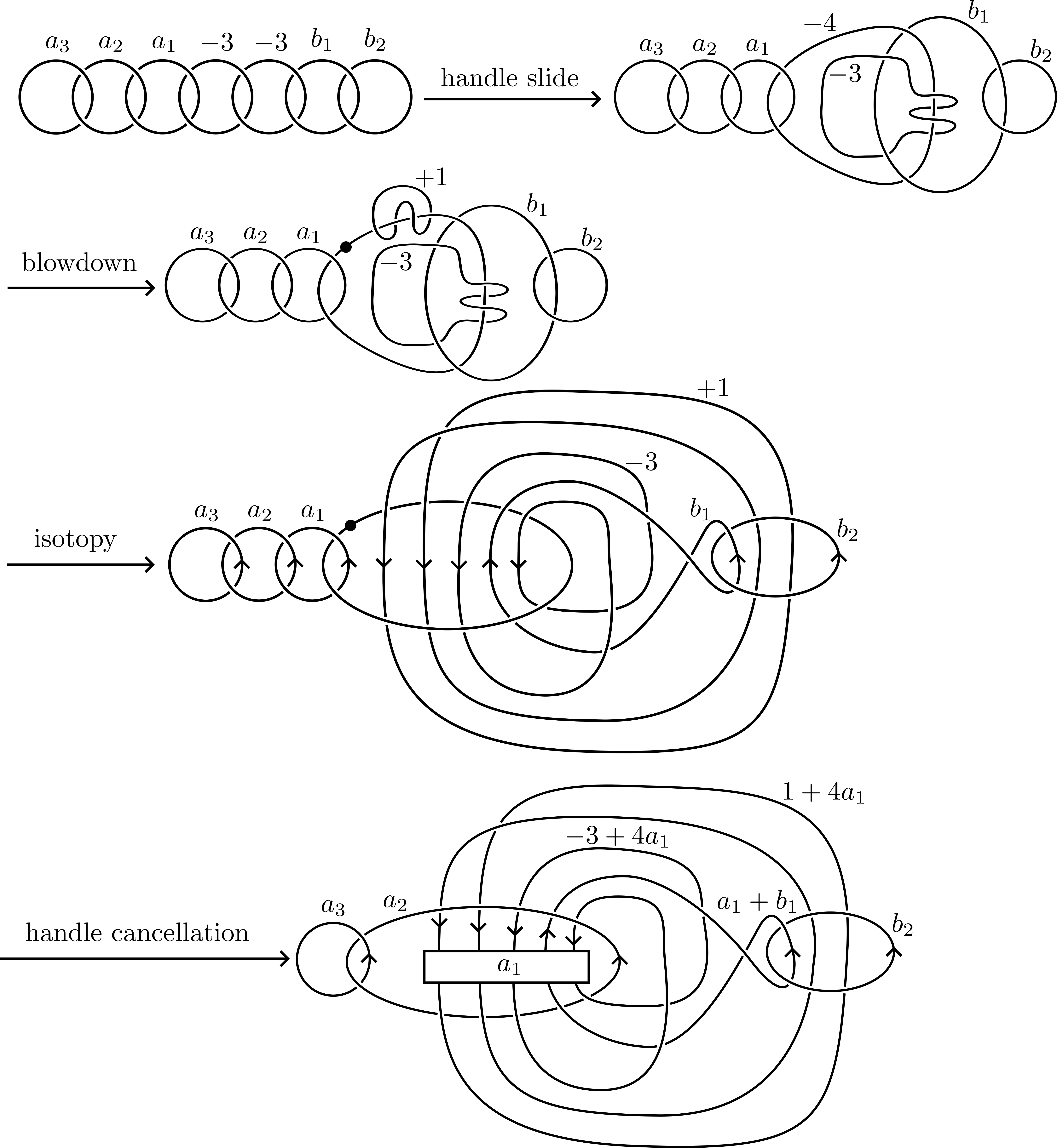}
    \vspace{.2cm}
    \caption{Blowdown for $[-3,-3]$.}
    \label{fig:c3}
\end{figure}

We then compare it with the orthogonal complement of the semi-standard embedding.
\begin{center}
\begin{tabular}{ccc|cccccc|cccc|cccccc|ccc} 
 & & & & & & & & & & & & & & & & & & & & & \\ 
 \hline
 & $\cdots$ & 1 & 1 & -1 & & & & & & & & & & & & & & & & & \tikzmark{e} \\ 
 & & & & 1 & -1 & & & & & & & & & & & & & & & & \\ 
 & & & & & & $\ddots$ & & & & & & & & & & & & & & & \\ 
 & & & & & & & 1 & -1 & & & & & & & & & & & & & \tikzmark{f} \\ 
 \hline
 & & & & & & & & 1 & 1 & -1 & & & & & & & & & & & \\ 
 & & & & & & & & & & 1 & 1 & -1 & & & & & & & & & \\
 & & & & & & & & & -1 & -1 & & & 1 & & & & & & & & \\
 \hline
 & & & & & & & & & & & & & -1 & 1 & & & & & & & \tikzmark{g} \\
 & & & & & & & & & & & & & & & $\ddots$ & & & & & & \\
 & & & & & & & & & & & & & & & & -1 & 1 & & & & \\
 & & & & & & & & & & & & & & & & & -1 & 1 & 1 & $\cdots$ & \tikzmark{h} \\
 & & & \tikzmark{a} & & & & & \tikzmark{b} & & & & & \tikzmark{c} & & & & & \tikzmark{d} & & & \\
 \hline
 \hline
 & $\cdots$ & 1 & -1 & & & & & & & & & & & & & & & & & & \\
$v_1:$ & & & 2 & 2 & 2 & $\cdots$ & 2 & 2 & -1 & 1 & & 1 & & & & & & & & & \\
$v_2:$ & & & 2 & 2 & 2 & $\cdots$ & 2 & 2 & -1 & 1 & 1 & 2 & & & & & & & & & \\
$v_3:$ & & & -1 & -1 & -1 & $\cdots$ & -1 & -1 & & -1 & & -1 & -1 & -1 & $\cdots$ & -1 & -1 & -1 & & & \\
 & & & & & & & & & & & & & & & & & & 1 & -1 & $\cdots$ & \\
\begin{tikzpicture}[overlay,remember picture]
\draw[<->,black,thick] (pic cs:a) -- (pic cs:b) node[midway,above]{$-a_1-1$};
\draw[<->,black,thick] (pic cs:c) -- (pic cs:d) node[midway,above]{$-b_1-1$};
\draw[<->,black,thick] (pic cs:e) -- (pic cs:f) node[midway,right]{$-a_1-2$};
\draw[<->,black,thick] (pic cs:g) -- (pic cs:h) node[midway,right]{$-b_1-2$};
\end{tikzpicture}
\end{tabular}
\end{center}
We compute the intersection pairings of the elements labeled as $v_1,v_2$ in the table: 
\begin{align*}
\langle v_1, v_1 \rangle &= -4(-a_1 - 1) - 3 = 1 + 4a_1,\\
\langle v_1, v_2 \rangle &= -4(-a_1 - 1) - 4 = 4a_1,\\
\langle v_1, v_3 \rangle &= 2(-a_1 - 1) + 2 = -2a_1,\\
\langle v_2, v_2 \rangle &= -4(-a_1 - 1) - 7 = -3 + 4a_1,\\
\langle v_2, v_3 \rangle &= 2(-a_1 - 1) + 3 = 1 - 2a_1,\\
\langle v_3, v_3 \rangle &= -(-a_1 - 1) - 2 - (-b_1 - 1) = a_1 + b_1.
\end{align*}
We see that the intersection pairing coincides with the one obtained from rational blowdown.

Finally, we work on the case of

$$\begin{tikzpicture}[xscale=1.0,yscale=1,baseline={(0,0)}]
    \node at (1-.1,0.4) {$-3$};
    \node at (2-.1,0.4) {$-2$};
    \node at (3-.1,0.4) {$-3$};
    \node (A1) at (0,0) {$\cdots$};
    \node (A2) at (1,0) {$\bullet$};
    \node (A3) at (2,0) {$\bullet$};
    \node (A4) at (3,0) {$\bullet$};
    \node (A5) at (4,0) {$\cdots$};
    \path (A1) edge [-] node [auto] {$\scriptstyle{}$} (A2);
    \path (A2) edge [-] node [auto] {$\scriptstyle{}$} (A3);
    \path (A3) edge [-] node [auto] {$\scriptstyle{}$} (A4);
    \path (A4) edge [-] node [auto] {$\scriptstyle{}$} (A5);
  \end{tikzpicture}.$$

After the handle slide as shown in Figure~\ref{fig:c4}, we proceed as in the previous case. Notice that the $-2$-sphere is not involved in the rest of the operations. Hence, the relevant matrix part is
$$\begin{pmatrix}
a_2 & 2 & 2 & 0 & -1 \\
2 & 1+4a_1 & 4a_1 & 0 & -2a_1 \\
2 & 4a_1 & -3+4a_1 & 1 & 1-2a_1 \\
0 & 0 & 1 & -2 & 0 \\
-1 & -2a_1 & 1-2a_1 & 0 & a_1+b_1
\end{pmatrix}.$$

\begin{figure}
    \centering
    \includegraphics[width=0.5\linewidth]{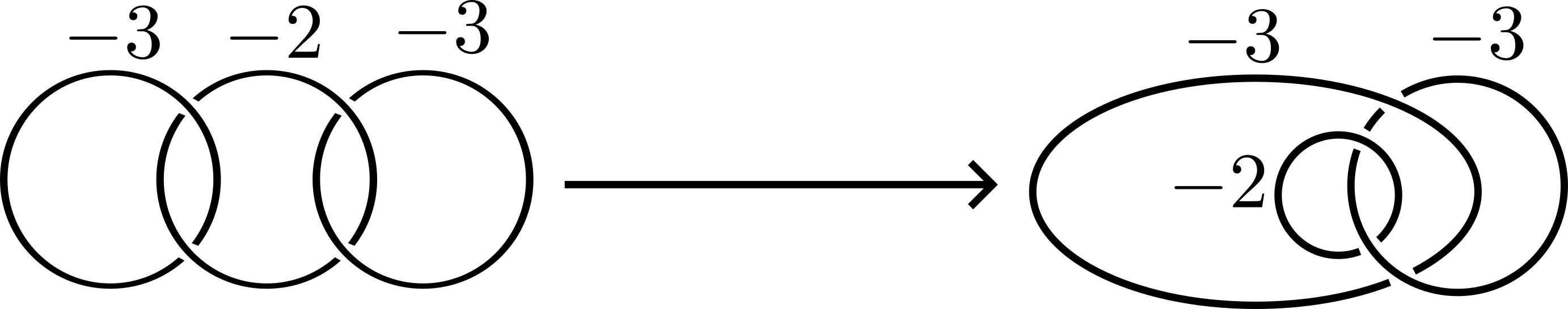}
    \caption{Handle slide for the $[-3,-2,-3]$ case.}
    \label{fig:c4}
\end{figure}

Regarding the orthogonal complement of the semi-standard embedding, we slightly modify the middle part of the previous case into the following:

\begin{center}
\begin{tabular}{c|c|ccccc|c|c} 
 & & & & & & & & \\ 
 \hline
& & & & & & & & \\
 \hline
& & 1 & -1 & & & & & \\
& & & 1 & 1 & 1 & -1 & & \\
& & -1 & -1 & & & & & \\
 \hline
& & & & & & & & \\
 \hline
 \hline
$v_1:$ & & -1 & 1 & & & 1 & & \\
$v_2:$ & & -1 & 1 & 1 & & 2 & & \\
$v_3:$ & & & -1 & & & -1 & & \\
$v_4:$ & & & & -1& 1 & & & \\
\end{tabular}.
\end{center}
Once again, the two intersection forms coincide.
\end{proof}

\section{Forbidden configurations}\label{sec:forbidden_config}

In this section, we prove Theorem~\ref{thm:all k}. Statement (\ref{mainbodyA1}) of Theorem~\ref{thm:all k} corresponds to Lemma~\ref{thm:induced property xk}. Statement (\ref{mainbodyA2}) corresponds to Lemma~\ref{lem:1.3(2)}. To prove statement (\ref{mainbodyA3}), we label the vertices as \textit{shallow}, \textit{deep}, and \textit{neither}, and put an upper bound on the number of vertices in each category. The hard part is producing an upper bound for the number of deep vertices, which is done in Lemma~\ref{lem:deep}. Producing an upper bound for non-deep vertices follows easily.

We first recall the statement:
\setcounter{maintheorem}{0}
\renewcommand{\thethm}{A}
\begin{maintheorem}\label{thm:all kbody}
Let $k$ be a positive integer. There exists a finite set of minimal configurations $S_k$ satisfying the following conditions:
\begin{enumerate}[font=\upshape]
    \item\label{mainbodyA1} A connected sum of lens spaces $L$ satisfies Property $X_k$ if and only if the plumbing graph associated to $L$ does not contain a configuration from $S_k$ as an induced subgraph;
    \item\label{mainbodyA2} the absolute value of the weight of any vertex in any configuration in $S_k$ is at most $2k+7$;
    \item\label{mainbodyA3} the number of vertices in any configuration in $S_k$ is at most $13k$.
\end{enumerate}
\end{maintheorem}
\renewcommand{\thecor}{\thesection.\arabic{thm}}

Fix a positive integer $k$. Let $S'_k$ be the set of all canonical plumbing graphs that do not satisfy Property $X_k$. Let $S_k$ be the subset of $S'_k$ consisting of those graphs that do not contain any other element of $S'_k$ as an induced subgraph. We first show that $S_k$ satisfies \eqref{mainbodyA1} of Theorem~\ref{thm:all kbody}. For this, it suffices to prove the following lemma.

\setcounter{thm}{0}
\begin{lem}\label{thm:induced property xk}
Suppose the canonical plumbing graph of $L'=\#_i L(p'_i,q'_i)$ contains the canonical plumbing graph of $L=\#_i L(p_i,q_i)$ as an induced subgraph. If $L$ does not satisfy Property $X_k$, then $L'$ does not satisfy Property $X_k$.
\end{lem}

\begin{proof}
For notational convenience, define
$$n_L := b_2(\natural_i X(p_i,q_i)) + b_2(\natural_i X(p_i,p_i-q_i))$$
and
$$n_{L'} := b_2(\natural_i X(p'_i,q'_i)) + b_2(\natural_i X(p'_i,p'_i-q'_i)).$$

By an inductive argument on the number of vertices, it suffices to prove Lemma \ref{thm:induced property xk} for the situation where the canonical plumbing graph of $L'=\#_i L(p'_i,q'_i)$ contains exactly one more vertex than that of $L=\#_i L(p_i,q_i)$. In this case, we have
$$b_2(\natural_i X(p'_i,q'_i)) = b_2(\natural_i X(p_i,q_i)) + 1.$$

Let $\Gamma$ and $\Gamma'$ denote the dual graphs for $L$ and $L'$, respectively. Let $v$ be the vertex in the canonical plumbing graph of $L'$ that is not in $L$. By assumption, there exists an embedding $\phi\colon V(\Gamma) \to \Z^{n_L-k}$. We distinguish three cases according to the connectivity of $v$:
\begin{enumerate}
    \item $v$ has no neighbors;
    \item $v$ has exactly one neighbor;
    \item $v$ has exactly two neighbors.
\end{enumerate}

\medskip\noindent\textbf{Case 1:}{ $v$ has no neighbors.} In the dual graph, $v$ corresponds to a component with $|w(v)|-1$ connected vertices with weight $2$. Let those corresponding vertices in $\Gamma$ be $x_1,\dots ,x_{|w(v)|-1}.$

Now we have
$$b_2(\natural_i X(p'_i,p'_i-q'_i))=b_2(\natural_i X(p_i,q_i))+|w(v)|-1$$
Therefore, we have
$$n_{L'}-k=n_L-k+|w(v)|$$
Therefore, $L'$ does not satisfy Property $X_k$ because we can define an embedding $\phi'$ by adjoining standard basis vectors $e'_1,\dots,e'_{|w(v)|}$ to $\Z^{n_L-k}$ and mapping $x_i$ to $e'_i-e'_{i+1}$ for each $1\le i\le |w(v)|-1$.

\medskip\noindent\textbf{Case 2:}{ $v$ has one neighbor.} In the dual graph, it corresponds to increasing the weight of some leaf $u$ by 1, and then attach a chain of $|w(v)|-2$ vertices with weight 2 on $u$ (so, $u$ is no longer a leaf unless $w(v)=-2$). Let the new chain of $|w(v)|-2$ vertices with weight 2 in the dual graph be $x_1,\dots ,x_{|w(v)|-2}$.

We have
$$b_2(\natural_i X(p'_i,p'_i-q'_i))=b_2(\natural_i X(p_i,q_i))+|w(v)|-2$$
and therefore
$$n_{L'}-k=n_L-k+|w(v)|-1.$$
Hence, $L'$ does not satisfy Property $X_k$ because we have an embedding $\phi'$ defined by adding basis vectors $e'_1,e'_2,\dots ,e'_{|w(v)|-1}$ to $\Z^{n_L-k}$, and setting $\phi'(u)=\phi(u)-e'_1$ and $\phi'(x_i)=e'_i-e'_{i+1}$ for each $i$.

\medskip\noindent\textbf{Case 3.1:}{ $v$ has two neighbors and $w(v)\neq -2$.} In the dual graph, it corresponds to increasing the weight of a leaf $u_1$ by 1, increasing the weight of another leaf $u_2$ by 1, create a chain of $|w(v)|-3$ vertices with weight 2 and then glue it to $u_1$ and $u_2$ at the two ends. Let $x_1,\dots ,x_{|w(v)|-3}$ be those new vertices in the dual graph with weight 2.

Similar to the previous cases, we have 
$$n_{L'}-k=n_L-k+|w(v)|-2$$
and we can define $\phi'\colon V(\Gamma')\rightarrow\Z^{n_{L'}-k}$ by adding basis vectors $e'_1,\dots ,e'_{|w(v)|-2}$ to $\Z^{n_L-k}$, and setting $\phi'(u_1)=\phi(u_1)-e'_1$, $\phi'(u_2)=\phi(u_2)+e'_{|w(v)|-2}$, and $\phi'(x_i)=e'_i-e'_{i+1}$ for each $i$.

\medskip\noindent\textbf{Case 3.2:}{ $v$ has two neighbors and $w(v)=-2$.} 
In terms of the dual graph, this operation corresponds to identifying two leaves $u_1, u_2 \in V(\Gamma)$ from distinct connected components to form a single vertex $u \in V(\Gamma')$, with the weight defined by the sum $w(u) = w(u_1) + w(u_2)$.

Similar to previous cases, we have 
$$n_{L'}-k=n_L-k.$$ For any other vertex $z$, to have the correct intersection pairing, we need $$\phi'(z)\cdot\phi'(u)=\phi(z)\cdot\phi(u_1)+\phi(z)\cdot\phi(u_2)$$ This can be achieved by setting $\phi'(u)=\phi(u_1)+\phi(u_2)$. Now, we check that $\phi'(u)\cdot\phi'(u)$ is indeed $w(u)$:
\begin{align*} 
\phi'(u)\cdot\phi'(u)&=(\phi(u_1)+\phi(u_2))\cdot(\phi(u_1)+\phi(u_2)) \\  &=\phi(u_1)\cdot\phi(u_1)+\phi(u_2)\cdot\phi(u_2)+2\phi(u_1)\cdot\phi(u_2) \\
&=w(u_1)+w(u_2)+2\phi(u_1)\cdot\phi(u_2) \\
&=w(u)+2\phi(u_1)\cdot\phi(u_2)
\end{align*}
Since $u_1$ and $u_2$ come from two different components in $\Gamma$, we have $\phi(u_1)\cdot\phi(u_2)=0$. This concludes the proof.
\end{proof}

Next, we show that $S_k$ satisfies conditions \eqref{mainbodyA2} and \eqref{mainbodyA3} of Theorem~\ref{thm:all kbody}. Let $G\in S_k$ be the canonical plumbing graph of a connected sum of lens spaces. Let $\Gamma$ be the dual graph of $G$. Similar to the proof of Theorem \ref{thm:induced property xk}, we let $n_G=|V(G)|+|V(\Gamma)|$. Let $\phi\colon V(\Gamma)\rightarrow\Z^{n_G-k}$ be an embedding. We define
$$k':=\begin{cases}k\text{ if }k\text{ if even;}\\k+1\text{ if }k\text{ is odd.}\end{cases}$$ Consider the following configuration.
$$\begin{tikzpicture}[xscale=1.0,yscale=1,baseline={(0,0)}]
    \node at (0.9,0.4) {$-2$};
    \node at (1.9,0.4) {$-2$};
	\node at (3.9,0.4) {$-2$};
	\node at (4.9,0.4) {$-2$};
    \node (A1) at (1,0) {$\bullet$};
    \node (A2) at (2,0) {$\bullet$};
	\node (A3) at (3,0) {$\cdots$};
	\node (A4) at (4,0) {$\bullet$};
	\node (A5) at (5,0) {$\bullet$};
	\draw [decorate,decoration={brace,amplitude=5pt,mirror,raise=2ex},line width=1pt] (1,0) -- (5,0) node[midway,yshift=-2em]{$k'$};
  \end{tikzpicture}$$
Let this configuration be $C_{k'}$. The configuration $C_{k'}$ is in $S'_k$ because we can embed the vertices of its dual (which has $k'$ disconnected vertices with weight $2$) by mapping them to $e_1+e_2, e_1-e_2, e_3+e_4, e_3-e_4$, etc. Hence, due to the construction of $S_k$, the graph $G$ cannot contain $C_{k'}$ as an induced subgraph unless $G=C_{k'}$.

In particular, $G$ cannot contain a chain of $2k'-1$ vertices of weight $-2$, as such a chain contains $C_{k'}$ as an induced subgraph. This translates to the condition that $\Gamma$ cannot contain a vertex $v$ with adjusted weight $w'(v)\geq 2k'$. Since $C_{k'}$ satisfies conditions \eqref{mainbodyA2} and \eqref{mainbodyA3} of Theorem \ref{thm:all kbody}, we assume $G\neq C_{k'}$.

By the construction of $S_k$, to show that $G\in S_k$ satisfies condition \eqref{mainbodyA2} of Theorem \ref{thm:all kbody}, it suffices to prove the following lemma:

\begin{lem}\label{lem:1.3(2)}
Suppose $G \in S_k$ has a vertex $g$ with $|w(g)| \ge 2k'+6$ and that $G$ does not contain $C_{k'}$ as an induced subgraph. Let $G'$ be the subgraph of $G$ induced by $V(G)\smallsetminus\{g\}$. Then $G'$ does not satisfy Property $X_k$. In particular, the absolute value of the weight of any vertex in $G$ is at most $2k+7$.
\end{lem}

\begin{proof}
Let $\Gamma'$ be the dual graph of $G'$. Let $n_{G'} = |V(G')| + |V(\Gamma')|$.

In $\Gamma$, the vertex $g$ corresponds to a chain of $|w'(g)| - 1$ vertices of weight $2$, where $w'(g)$ is the adjusted weight. Let these vertices be $x_1,\dots,x_{|w'(g)|-1}$. If $x_1$ is not a leaf, let $y_0 \neq x_2$ be a vertex adjacent to $x_1$. If $x_{|w'(g)|-1}$ is not a leaf, let $y_1 \neq x_{|w'(g)|-2}$ be a vertex adjacent to $x_{|w'(g)|-1}$. By the definition of the dual graph, removing $g$ from $G$ corresponds to removing $x_1,\dots,x_{|w'(g)|-1}$ from $\Gamma$ and decreasing the weights of $y_0$ and $y_1$ (if they exist) by $1$.
Hence,
\[
n_{G'} = n_G - 1 - (|w'(g)| - 1) = n_G - |w'(g)|.
\]

Since $|w(g)| \geq 2k' + 6$, we have $|w'(g)| - 1 \geq 2k' + 3 \geq 5$, so it satisfies the Working Conditions of \cite{definite}. Then, by \cite[Lemma 4.7]{definite}, up to isomorphism we have $\phi(x_i) = e_i - e_{i+1}$ for $1 \leq i \leq |w'(g)| - 1$ for some basis vectors $e_1,\dots,e_{|w'(g)|}$. We now show that for any vertex $v \neq x_1,\dots,x_{|w'(g)|-1}$, we must have $e_2,\dots,e_{|w'(g)|-1} \notin \supp(\phi(v))$. Suppose not, and assume there exists a vertex $v \neq x_1,\dots,x_{|w'(g)|-1}$ such that $\supp(\phi(v))$ contains at least one of $e_2,\dots,e_{|w'(g)|-1}$. Since $\phi(v) \cdot \phi(x_i)$ must be $0$ for all $2 \leq i \leq |w'(g)| - 2$, it follows that $e_2,\dots,e_{|w'(g)|-1}$ are all contained in $\supp(\phi(v))$. Hence $w(v) \geq |w'(g)| - 2 \geq 2k' + 2$. This implies $w'(v) \geq 2k'$, which is impossible as discussed above when $G$ does not contain $C_{k'}$ as an induced subgraph.

Similarly, by considering the intersection pairing with $x_1$, we see that if $y_0$ exists then it is the unique vertex other than $x_1$ to have $e_1$ in its support, and if $y_0$ does not exist then there is no vertex other than $x_1$ having $e_1$ in its support. Likewise, we know that if $y_1$ exists then it is the unique vertex other than $x_{|w'(g)|-1}$ to have $e_{|w'(g)|}$ in its support, and if $y_1$ does not exist then there is no vertex other than $x_{|w'(g)|-1}$ having $e_{|w'(g)|}$ in its support.

Thus we obtain an embedding $\phi'\colon V(G')\rightarrow \Z^{n_{G'}-k}$ by removing $e_1,\dots,e_{|w'(g)|}$ from $\Z^{n_G-k}$, removing $x_1,\dots,x_{|w'(g)|-1}$ from $G$, removing the $e_1$-component of $\phi(y_0)$ (if $y_0$ exists), and removing the $e_{|w'(g)|}$-component of $\phi(y_1)$ (if $y_1$ exists).
\end{proof}

It remains to show that $G\in S_K$ satisfies condition \eqref{mainbodyA3} of Theorem~\ref{thm:all kbody}. To this end, we first introduce the following definitions.

\begin{defn}
A vertex $v$ of $G$ is called \textit{shallow} if it satisfies at least one of the following:
\begin{enumerate}
    \item $w(v) = -2$;
    \item $w(v) = -4$;
    \item $w(v) = -3$ and $v$ is adjacent to a vertex of weight $-2$;
    \item $w(v) = -3$ and $v$ is adjacent to another vertex of weight $-3$.
\end{enumerate}
\end{defn}

\begin{defn}
A vertex $v$ of $G$ is called \textit{deep} if it satisfies one of the following:
\begin{enumerate}
    \item $|w(v)| \geq 5$ and $v$ is not adjacent to any shallow vertex;
    \item $w(v) = -3$ and $v$ has no neighbors.
\end{enumerate}
\end{defn}

Our strategy is to obtain upper bounds for the number of shallow vertices, the number of deep vertices, and the number of vertices that are neither shallow nor deep.

To obtain an upper bound on the number of deep vertices of $G$, we prove the following lemma:

\begin{lem}\label{lem:deep}
If $G \in S_k$ has at least $k'$ deep vertices and $G$ does not contain $C_{k'}$ as an induced subgraph, then there exists a deep vertex $g \in V(G)$ such that, letting $G'$ be the subgraph of $G$ induced by $V(G)\smallsetminus\{g\}$, the graph $G'$ does not satisfy Property $X_k$. In particular, $G$ has at most $k'-1$ deep vertices.
\end{lem}

\begin{proof}
Let $x_1,\dots,x_{k'}$ be deep vertices. For every $1\leq i\leq k'$, if $w(x_i)\neq -3$, then since $|w(x_i)|\geq 5$ and $x_i$ is not adjacent to any vertex of weight $-2$, the vertex $x_i$ in $G$ corresponds to a configuration of the form
\[
\begin{tikzpicture}[xscale=1.0,yscale=1,baseline={(0,0)}]
    \node at (0,0.4) {$1$};
    \node at (1,0.4) {$0$};
    \node at (3,0.4) {$0$};
    \node at (4,0.4) {$1$};
    \node (A1) at (0,0) {$\circ$};
    \node (A2) at (1,0) {$\circ$};
    \node (A3) at (2,0) {$\cdots$};
    \node (A4) at (3,0) {$\circ$};
    \node (A5) at (4,0) {$\circ$};
    \path (A1) edge [-] (A2);
    \path (A2) edge [-] (A3);
    \path (A3) edge [-] (A4);
    \path (A4) edge [-] (A5);
    \draw [decorate,decoration={brace,amplitude=5pt,mirror,raise=2ex},line width=1pt]
        (1,0) -- (3,0)
        node[midway,yshift=-2.4em]{$|w(x_i)|-3$};
\end{tikzpicture}
\]
in $\Gamma$. Note that $|w(x_i)| - 3 \geq 2$. In this case, we define $O_i$ to be the set of vertices with adjusted weight $0$ in the diagram above, and $I_i$ to be the set of vertices with adjusted weight $1$ in the diagram above.

If $w(x_i) = -3$ and $x_i$ has no neighbors, then it corresponds to a component of $\Gamma$ of the form
\[
\begin{tikzpicture}[xscale=1.0,yscale=1,baseline={(0,0)}]
    \node at (0,0.4) {$2$};
    \node at (1,0.4) {$2$};
    \node (A1) at (0,0) {$\bullet$};
    \node (A2) at (1,0) {$\bullet$};
    \path (A1) edge [-] node [auto] {$\scriptstyle{}$} (A2);
\end{tikzpicture}.
\]
In this case, we define $O_i$ to be the set consisting of these two vertices, and $I_i$ to be the empty set.

Let
\[
\supp(O_i) := \bigcup_{v\in O_i} \supp(\phi(v)).
\]
Since $|w(x_i)| \neq 4$, it satisfies the Working Conditions of \cite{definite}, and we can apply \cite[Lemma 4.7]{definite} to conclude that the subgraph of $\Gamma$ induced by $\bigcup_i O_i$ embeds standardly under $\phi$. Note that the condition $|w(x_i)| \neq 4$ is needed, since the proof of \cite[Lemma 4.7]{definite} requires ruling out the case of having a component with exactly three vertices, each of weight $2$. Hence, whenever $w(x_i)\neq -3$, we have
\[
|\supp(O_i)| = |w(x_i)| - 2.
\]
When $w(x_i) = -3$, we have $|\supp(O_i)| = 3$. Moreover, whenever $i \neq j$, we have
\[
|\supp(O_i)\cap\supp(O_j)| = 0.
\]

We now prove that the remaining vertices embed standardly, so that the subgraph of $\Gamma$ induced by $\bigcup_i (O_i \cup I_i)$ embeds standardly under $\phi$. We use \cite[Lemma~4.8]{definite} together with its proof. In Cases~(1)--(3) of that proof, the Working Conditions are only used to rule out the
\[
\begin{tikzpicture}[xscale=1.0,yscale=1,baseline={(0,0)}]
    \node at (0,0.4) {$1$};
    \node at (1,0.4) {$0$};
    \node at (2,0.4) {$1$};
    \node (A1) at (0,0) {$\circ$};
    \node (A2) at (1,0) {$\circ$};
    \node (A3) at (2,0) {$\circ$};
    \path (A1) edge [-] node [auto] {$\scriptstyle{}$} (A2);
    \path (A2) edge [-] node [auto] {$\scriptstyle{}$} (A3);
\end{tikzpicture}
\]
configuration. Hence, the conclusion of \cite[Lemma~4.8]{definite} applies when $x_i$ is a leaf or when $|w(x_i)| \geq 6$. The case when $x_i$ is not a leaf and $|w(x_i)| = 5$ corresponds to Case~4 in the proof of \cite[Lemma 4.8]{definite}. In this case, let $u_1,u_2$ be the neighbors of $x_i$. Without loss of generality, assume that $|w(u_1)| \geq |w(u_2)|$. If $|w(u_2)| \geq 4$, then we obtain configuration~(1) in Case~4 of the proof of \cite[Lemma 4.8]{definite}. If $|w(u_1)| \geq 4$ and $w(u_2) = -3$, then, since $u_2$ is not shallow, it is either a leaf or is adjacent to some vertex $u_3 \neq x_i$ with $|w(u_3)| \geq 4$. If $u_2$ is a leaf, we obtain configuration~(1) in the proof of \cite[Lemma 4.8]{definite}. If $u_2$ is adjacent to some vertex $u_3 \neq x_i$ with $|w(u_3)| \geq 4$, we obtain configuration~(2) in the proof of \cite[Lemma 4.8]{definite}. If $w(u_1) = w(u_2) = -3$, then, similarly to the case above, either $u_1$ is a leaf or there is some $u_0 \neq x_i$ adjacent to $u_2$ satisfying $|w(u_0)| \geq 4$, and either $u_2$ is a leaf or there is some $u_3 \neq x_i$ adjacent to $u_2$ satisfying $|w(u_3)| \geq 4$. In all cases, we obtain configuration~(1), (2), or (5) in the proof of \cite[Lemma 4.8]{definite}. By the definition of shallow, these are all the possible cases.
Hence, the subgraph of $\Gamma$
\[
\begin{tikzpicture}[xscale=1.0,yscale=1,baseline={(0,0)}]
    \node at (0,0.4) {$1$};
    \node at (1,0.4) {$0$};
    \node at (3,0.4) {$0$};
    \node at (4,0.4) {$1$};
    \node (A1) at (0,0) {$\circ$};
    \node (A2) at (1,0) {$\circ$};
    \node (A3) at (2,0) {$\cdots$};
    \node (A4) at (3,0) {$\circ$};
    \node (A5) at (4,0) {$\circ$};
    \path (A1) edge [-] node [auto] {$\scriptstyle{}$} (A2);
    \path (A2) edge [-] node [auto] {$\scriptstyle{}$} (A3);
    \path (A3) edge [-] node [auto] {$\scriptstyle{}$} (A4);
    \path (A4) edge [-] node [auto] {$\scriptstyle{}$} (A5);
    \draw [decorate,decoration={brace,amplitude=5pt,mirror,raise=2ex},line width=1pt] (1,0) -- (3,0) node[midway,yshift=-2.2em]{$|w(x_i)|-3$};
\end{tikzpicture}
\]
that corresponds to $x_i$ must embed standardly under $\phi$.

For every $v \in V(\Gamma)$, we say that $v$ \textit{obstructs} $i$ if $v \notin I_i$ and 
$\bigl|\supp(O_i)\cap\supp(\phi(v))\bigr| \neq 0$.
Set
\[
\mathrm{obs}(v) := \bigl|\{\, i \mid v \text{ obstructs } i \}\bigr|.
\]
By considering the intersection pairing between $v$ and the vertices in $O_i$, we see that whenever $v$ obstructs $i$ we must have
\[
\supp(O_i)\subseteq\supp(\phi(v)).
\]
Since $\bigl|\supp(O_i)\cap \supp(O_j)\bigr| = 0$ whenever $i \neq j$, we obtain
\[
w(v) \geq \sum_{\substack{i \\ v\ \text{obstructs } i}} \bigl|\supp(O_i)\bigr|
\geq 3\,\mathrm{obs}(v).
\]

Now we argue that whenever $\mathrm{obs}(v)=1$, we must have $w'(v)\geq 2$. This is immediate when $w(v)\geq 4$, so it suffices to consider the case $w(v)=3$. Since $\mathrm{obs}(v)=1$, the vertex $v$ obstructs $i$ for some $i$. Suppose that $w(v)=3$. Then, since $|\supp(O_i)|\geq 3$ for all possible values of $w(x_i)$, we must have $\supp(v)=\supp(O_i)$. If $w(x_i)\neq -3$, this is impossible: the condition $\supp(v)=\supp(O_i)$ forces $v$ to have non-zero intersection pairing with both vertices in $I_i$, which would imply that there is a cycle in $\Gamma$. Hence $w(x_i)=-3$, and therefore $x_i$ has no neighbors.
We now show that $v$ also has no neighbors, which implies $w'(v)=3$. Up to isomorphism and relabeling of basis vectors, the two vertices in $O_i$ are mapped under $\phi$ to $e_1-e_2$ and $e_2-e_3$, respectively. By considering intersection pairing, we have $\phi(v)=\pm(e_1+e_2+e_3)$. To see that $v$ has no neighbors, let $u\in V(\Gamma)\smallsetminus O_i$. By considering the intersection pairing with the vertices in $O_i$, we obtain
\[
\phi(u)\cdot e_1=\phi(u)\cdot e_2=\phi(u)\cdot e_3.
\]
It follows that $\phi(v)\cdot\phi(u)$ is a multiple of $3$, which cannot be $-1$. Therefore, $u$ cannot be a neighbor of $v$. Consequently, in all cases, whenever $\mathrm{obs}(v)=1$, we must have $w'(v)\geq 2$.

Together with the inequality $w(v)\geq 3\mathrm{obs}(v)$, we conclude that in all cases
\[
w'(v)\geq 2\mathrm{obs}(v).
\]
Hence, if $\mathrm{obs}(v)\neq 0$, the vertex $v$ corresponds to a chain of at least $2\mathrm{obs}(v)-1$ vertices of weight $-2$ in $G$. Such a chain contains $\mathrm{obs}(v)$ pairwise disconnected vertices of weight $-2$. Since $G$ does not contain $C_{k'}$ as an induced subgraph, we conclude that
\[
k' > \sum_{v\in V(\Gamma)} \mathrm{obs}(v).
\]
Therefore, there exists some $1\leq i\leq k'$ such that $i$ is not obstructed by any $v\in V(\Gamma)$. Fix such an index $i$ and set $g:=x_i$. Let $\Gamma'$ be the dual graph of $G'$, where $G'$ is the subgraph of $G$ induced by $V(G)\smallsetminus\{g\}$. Finally, define
\[
n_{G'} = |V(G')| + |V(\Gamma')|.
\]
Since $G'$ is obtained by removing one vertex from $G$, we have $|V(G')|=|V(G)|-1$.

Remark: In the case $w(x_i)\neq -3$, for any $z\in I_i$, if $z$ is a leaf, then $w(z)=2$. Among the two elements of $\supp(z)$, one lies in $\supp(O_i)$. Moreover, for any $v$ with $v\neq z$ and $v\notin O_i$, considering the intersection pairing with $z$ shows that, since $v$ does not obstruct $i$, we must have $|\supp(v)\cap\supp(z)|=0$.

\medskip\noindent\textbf{Case 1:}{ $g$ has no neighbors.} In this case, removing $g$ from $G$ removes all vertices in $O_i$ and $I_i$ from $\Gamma$. Hence,
\[
|V(\Gamma')| = |V(\Gamma)| - (|w(g)|-1),
\]
and therefore
\[
n_{G'}-k = n_G-k-|w(g)|.
\]
However, since $i$ is not obstructed by any $v$, and also by the remark above, the embedding no longer needs to use the $|w(g)|$ basis vectors used by the vertices in $O_i\cup I_i$. Hence, $G'$ does not satisfy Property $X_k$.

\medskip\noindent\textbf{Case 2:}{ $g$ has one neighbor.} In this case, removing $g$ from $G$ removes all vertices in $O_i$ and one vertex in $I_i$, and decreases the weight of the other vertex in $I_i$ by $1$. Hence,
\[
|V(\Gamma')| = |V(\Gamma)|-(|w(g)|-2),
\]
and therefore
\[
n_{G'}-k = n_G-k-(|w(g)|-1).
\]
Now we remove the $|w(g)|-1$ basis vectors used by the vertices in $O_i$ and by the vertex in $I_i$ that is removed. Since $i$ is not obstructed by any $v$, and also by the remark above, the only effect on the remainder of the dual graph is that the weight of the remaining vertex in $I_i$ decreases by $1$, due to having a component on a basis vector in $\supp(O_i)$. This produces the desired graph $\Gamma'$. Hence, $G'$ does not satisfy Property $X_k$.

\medskip\noindent\textbf{Case 3:}{ $g$ has two neighbors.} In this case, removing $g$ from $G$ removes all vertices in $O_i$, and decreases the weights of the vertices in $I_i$ by $1$. Hence,
\[
|V(\Gamma')| = |V(\Gamma)|-(|w(g)|-3),
\]
and therefore
\[
n_{G'}-k = n_G-k-(|w(g)|-2).
\]
Now we remove the $|w(g)|-2$ basis vectors in $\supp(O_i)$. Since $i$ is not obstructed by any $v$, the only effect on the remainder of the dual graph is that the weights of the vertices in $I_i$ decrease by $1$, due to having a component on a basis vector in $\supp(O_i)$. This produces the desired graph $\Gamma'$. Hence, $G'$ does not satisfy Property $X_k$.
\end{proof}

% By the construction of $S_k$, Lemma~\ref{lem:deep} implies that $G$ has at most $k'-1$ deep vertices.

To obtain an upper bound on the number of shallow vertices, we first generalize the construction of $C_{k'}$. Let $D_{k'}$ denote the set of configurations that have $k'$ components, each component being one of the following three:
\[
\begin{tikzpicture}[xscale=1,yscale=1,baseline={(0,0)}]
    \node at (0.9,0.4) {$-2$};
    \node (A1) at (1,0) {$\bullet$};
\end{tikzpicture},\quad
\begin{tikzpicture}[xscale=1,yscale=1,baseline={(0,0)}]
    \node at (0.9,0.4) {$-3$};
    \node at (1.9,0.4) {$-3$};
    \node (A1) at (1,0) {$\bullet$};
    \node (A2) at (2,0) {$\bullet$};
    \path (A1) edge [-] node [auto] {$\scriptstyle{}$} (A2);
\end{tikzpicture},\quad
\begin{tikzpicture}[xscale=1,yscale=1,baseline={(0,0)}]
    \node at (0.9,0.4) {$-4$};
    \node (A1) at (1,0) {$\bullet$};
\end{tikzpicture}.
\]
The two configurations
\[
\begin{tikzpicture}[xscale=1,yscale=1,baseline={(0,0)}]
    \node at (0.9,0.4) {$-3$};
    \node at (1.9,0.4) {$-3$};
    \node (A1) at (1,0) {$\bullet$};
    \node (A2) at (2,0) {$\bullet$};
    \path (A1) edge [-] node [auto] {$\scriptstyle{}$} (A2);
\end{tikzpicture},\quad
\begin{tikzpicture}[xscale=1,yscale=1,baseline={(0,0)}]
    \node at (0.9,0.4) {$-4$};
    \node (A1) at (1,0) {$\bullet$};
\end{tikzpicture}
\]
do not satisfy Property $X_1$, and
\[
\begin{tikzpicture}[xscale=1,yscale=1,baseline={(0,0)}]
    \node at (0.9,0.4) {$-2$};
    \node at (1.9,0.4) {$-2$};
    \node (A1) at (1,0) {$\bullet$};
    \node (A2) at (2,0) {$\bullet$};
\end{tikzpicture}
\]
does not satisfy Property $X_2$. By combining the corresponding embeddings, we see that any configuration in $D_{k'}$ does not satisfy Property $X_k$.

Since every configuration in $D_{k'}$ satisfies condition~(3) of Theorem~\ref{thm:all kbody}, we may assume that $G \notin D_{k'}$. In particular, by the definition of $S_k$, the graph $G$ cannot contain any configuration in $D_{k'}$ as an induced subgraph.

The following lemma gives an upper bound on the number of shallow vertices of $G$.

\begin{lem}\label{lem:shallow}
If $G \in S_k$ and $G$ does not contain $D_{k'}$ as an induced subgraph, then $G$ has at most $6k'-6$ shallow vertices.
\end{lem}

\begin{proof}
Consider the following procedure, in which we label each shallow vertex in $G$ with either ``want'', ``discard'', or ``leftover'' one by one. Whenever there is at least one vertex in $G$ with weight $-2$ or $-4$ that has not been labeled, we pick one such vertex and label it ``want'', and label all shallow vertices adjacent to it as ``discard''. After all vertices with weight $-2$ or $-4$ have been labeled, we turn our attention to the shallow vertices with weight $-3$. Whenever there is a pair of adjacent vertices with weight $-3$, both unlabeled, we label both of them ``want'', and label the vertices adjacent to them as ``discard''. We keep going until every pair of adjacent vertices with weight $-3$ has at least one of them labeled. Finally, we label all remaining unlabeled shallow vertices as ``leftover''.

Consider the subgraph induced by the vertices labeled ``want''. Since $G$ does not contain any configuration in $D_{k'}$ as an induced subgraph, we know that there are at most $k'-1$ vertices with weight $-2$ or $-4$, or pairs of vertices with weight $-3$, labeled ``want''. Hence, there are at most $2k'-2$ vertices labeled ``want''. Whenever we apply the label ``want'', we label at most two vertices ``discard''. Hence, there are at most $2k'-2$ vertices labeled ``discard''.

For a shallow vertex $v$ to receive the label ``leftover'', by the definition of shallow and by how the labeling procedure works, we know that $v$ must have weight $-3$ and be adjacent to some vertex $u$ of weight $-2$ or $-3$, where $u$ is labeled ``discard''.

Since there are at most $2k'-2$ vertices labeled ``discard'', and each vertex labeled ``discard'' must be adjacent to a shallow vertex labeled ``want'', we conclude that there are at most $2k'-2$ vertices labeled ``leftover''.

In total, there are at most
\[
(2k'-2)+(2k'-2)+(2k'-2)=6k'-6
\]
shallow vertices.
\end{proof}

% Finally, we consider the number of vertices that are neither shallow nor deep. By definition, a vertex $v$ that is neither shallow nor deep must satisfy one of the following:
% \begin{enumerate}
%     \item\label{non-shallow non-deep 1} $|w(v)|\geq 5$ and $v$ is adjacent to a shallow vertex;
%     \item\label{non-shallow non-deep 2} $w(v)=-3$ and $v$ is adjacent to some vertex $u$ with $|w(u)|\geq 4$.
% \end{enumerate}
% If a vertex is neither shallow nor deep and satisfies~(\ref{non-shallow non-deep 1}), we call it a type~\ref{non-shallow non-deep 1} vertex. If a vertex is neither shallow nor deep and satisfies~(\ref{non-shallow non-deep 2}), we call it a type~\ref{non-shallow non-deep 2} vertex.

We are now ready to complete the proof of Theorem~\ref{thm:all kbody} by accounting for the remaining vertices.

\begin{proof}[Proof of Theorem~\ref{thm:all kbody}]
By Lemma~\ref{thm:induced property xk}, Lemma~\ref{lem:1.3(2)}, Lemma~\ref{lem:deep}, and Lemma~\ref{lem:shallow}, it remains to consider the number of vertices that are neither shallow nor deep. By definition, a vertex $v$ that is neither shallow nor deep must satisfy one of the following:
\begin{enumerate}
    \item\label{non-shallow non-deep 1} $|w(v)|\geq 5$ and $v$ is adjacent to a shallow vertex;
    \item\label{non-shallow non-deep 2} $w(v)=-3$ and $v$ is adjacent to some vertex $u$ with $|w(u)|\geq 4$.
\end{enumerate}
If a vertex is neither shallow nor deep and satisfies~(\ref{non-shallow non-deep 1}), we call it a type~\ref{non-shallow non-deep 1} vertex. If a vertex is neither shallow nor deep and satisfies~(\ref{non-shallow non-deep 2}), we call it a type~\ref{non-shallow non-deep 2} vertex.

Consider the subgraph $\mathrm{Shallow}_G$ induced by the shallow vertices in $G$. Suppose that we have labeled each shallow vertex in $G\in S_k$ with either ``want'', ``discard'', or ``leftover'', as in Lemma~\ref{lem:shallow}. Since each shallow vertex labeled ``discard'' must be adjacent to at least one shallow vertex labeled ``want'', and each shallow vertex labeled ``leftover'' must be adjacent to at least one shallow vertex labeled ``discard'', we conclude that every connected component of $\mathrm{Shallow}_G$ contains at least one shallow vertex labeled ``want''. Hence, $\mathrm{Shallow}_G$ has at most $k'-1$ connected components. Therefore, there are at most $2k'-2$ type~\ref{non-shallow non-deep 1} vertices.

It remains to consider type~\ref{non-shallow non-deep 2} vertices. In case~(\ref{non-shallow non-deep 2}), since $|w(u)|\geq 4$, the vertex $u$ must be either shallow, deep, or of type~\ref{non-shallow non-deep 1}.
Consider the subgraph of $G$ induced by vertices that are either shallow or of type~\ref{non-shallow non-deep 1}. Using similar arguments regarding the ``want'' label as above, we conclude that this subgraph has at most $k'-1$ connected components. Hence, there are at most $2k'-2$ type~\ref{non-shallow non-deep 2} vertices coming from the case where $u$ is either shallow or of type~\ref{non-shallow non-deep 1}.
Since there are at most $k'-1$ deep vertices, there are at most $2k'-2$ type~\ref{non-shallow non-deep 2} vertices coming from the case where $u$ is deep.
Hence, there are at most $4k'-4$ type~\ref{non-shallow non-deep 2} vertices.

In total, there are at most $6k'-6$ vertices that are neither deep nor shallow. Combined with Lemma~\ref{lem:deep} and Lemma~\ref{lem:shallow}, we conclude that overall there are at most $k'-1$ deep vertices, $6k'-6$ shallow vertices, and $6k'-6$ vertices that are neither deep nor shallow. Hence, $|V(G)|\leq 13k'-13$. This concludes the proof of Theorem~\ref{thm:all kbody}.\end{proof}

\section{Symplectic fillings and Milnor fibers}\label{sec:symplectic}

In this section we prove Corollary~\ref{cor:symplectic} and Corollary~\ref{cor:symplecticunique}. We begin by recalling the first statement.

\setcounter{cor}{0}
\renewcommand{\thecor}{1.5}
\begin{cor}\label{cor:symplecticbody}
Let $L = L(p,q)$ be a lens space, and let $X(p,q)$ denote the corresponding canonical negative-definite plumbing. Suppose that every smooth negative-definite filling $X$ of $L$ satisfies $b_2(X) \ge b_2(X(p,q)) - 1$. If $\xi_{\mathrm{st}}$ is the standard contact structure on $L$, then there exists a minimal symplectic filling $W$ of $(L,\xi_{\mathrm{st}})$ such that
\[
b_2(W)=\min_X b_2(X),
\]
where the minimum is taken across all smooth negative-definite fillings $X$ of $L$.

Furthermore, every such $W$ is simply connected unless $L \in \{L(4,1), L(8,3), L(12,5)\}$, in which case it has fundamental group $\Z/2$.
\end{cor}
\renewcommand{\thecor}{\thesection.\arabic{cor}}

\begin{proof}
Let $L=L(p,q)$ be a lens space. Suppose that every smooth negative-definite filling $X$ of $L$ satisfies $b_2(X)\geq b_2(X(p,q))-1$.

If every smooth negative-definite filling $X$ of $L$ satisfies
$b_2(X)\geq b_2(X(p,q))$, the result follows from the fact that
the symplectic filling obtained from the canonical resolution is
diffeomorphic to $X(p,q)$~\cite[Corollary~1.7]{definite}. Thus we may assume that there exists some $X$ with $b_2(X)=b_2(X(p,q))-1$. In this case, by our discussion in Remark~\ref{rem:comparintworkingcondition}, the dual plumbing graph $\Gamma$ has at least one bad vertex.

% In \cite[Thm.~1.1(2)]{Lisca:2008-1}, Lisca provides a method to construct all symplectic fillings of lens spaces equipped with the standard contact structure. We use the notation of \cite{Lisca:2008-1}. Let $\Gamma$ have vertices $v_1,\ldots,v_k$, and for each $i$, let $v_i$ have weight $b_i$. Let $v_j$ be the bad vertex. By definition of a bad vertex, we have $1<j<k$. We apologize for using $b_i$ for the weights of $\Gamma$ (to be consistent with \cite{Lisca:2008-1}) while using $b_2$ to denote the second Betti number.

In \cite[Thm.~1.1]{Lisca:2008-1}, Lisca provides a method to construct all symplectic fillings of lens spaces with the standard contact structure. We use the notation of \cite{Lisca:2008-1}. Given a lens space $L$, consider the continued
fraction expansion
\[
\frac{p}{p-q} = [b_1,\ldots,b_k]^-,
\]
where $b_i$ are integers with $b_i \ge 2$.\footnote{Here we follow Lisca’s notation and use $b_i$ for the weights of $\Gamma$ and use $n_i$ for the entries of the admissible tuple; note that $b_2$ still denotes the second Betti number and the $n_k$ here is unrelated to the $n_k$ in the previous section.}
Let
\[
\mathbf{n} = (n_1,\ldots,n_k)
\]
be an admissible tuple (see \cite{Lisca:2008-1} for the precise definition) with
\begin{equation}\label{eq:admissible}
[n_1,\ldots,n_k]^- = 0
\qquad\text{and}\qquad
b_i \ge n_i\quad \text{ for } i = 1,\ldots,k.
\end{equation}
Note that this condition implies that the surgery diagram given by a chain of $k$ components, where the $i$–th component is framed by $n_i$, describes $S^1 \times S^2$, which we may identify with $\partial (S^1 \times D^3)$. Lisca proves that every symplectic filling of $(L,\xi_\mathrm{st})$ is diffeomorphic to a $4$–manifold obtained from $S^1 \times D^3$ by attaching $-1$–framed $2$–handles along its boundary via the above identification, where for some admissible tuple $\mathbf{n} = (n_1,\ldots,n_k)$ satisfying \eqref{eq:admissible} the number of $2$–handles attached along the meridian of the $i$–th component of the surgery link is $b_i - n_i$.

Let $\Gamma$ be a connected weighted linear graph with vertices $v_1,\ldots,v_k$, and for each $i$, let $v_i$ have weight $b_i$. Let $v_j$ be the bad vertex. By definition of a bad vertex, we have $1<j<k$. As mentioned in \cite[\S 1]{Lisca:2008-1}, the canonical negative-definite plumbing corresponds to the admissible $k$-tuple $(n_1,\ldots,n_k):=(1,2,\ldots,2,1)$. Consider another admissible $k$-tuple $(n'_1,\ldots,n'_k)$ defined as follows:

$$n'_i=\begin{cases}1\text{ if }i=j;\\
n_i+1\text{ if }|i-j|=1;\\
n_i\text{ otherwise.}
\end{cases}$$
By the definition of bad vertex, if $|i-j|=1$, we have $w'(n_i)\geq 1$. Hence, we have $n'_i\leq b_i$ for all $i$. Moreover, blowing down on $n'_j$ produces the admissible $k-1$-tuple $(1,2,\ldots,2,1)$. Hence, we have $[n'_1,\ldots,n'_k]=0$, and hence the admissible $k$-tuple $(n'_1,\ldots,n'_k)$ represents a symplectic filling of $L$. Let $W$ be this symplectic filling.

The sum of the entries of $(n'_1,\ldots,n'_k)$ is one more than the sum of the entries of $(n_1,\ldots,n_k)$. Hence the corresponding symplectic filling uses one fewer $2$--handle than the canonical negative-definite plumbing, and therefore $b_2(W) = b_2(X(p,q)) - 1$. This proves the first part of the corollary.

% As mentioned in \cite{Lisca:2008-1}, surgery along a chain of unknots with framings $n'_1,\ldots,n'_k$ produces $S^1\times S^2$.

For the fundamental group computation, first note that if $n'_1 < b_1$ or $n'_k < b_k$, then in the construction of $W$ we attach at least one $2$--handle along an unknot at one of the two ends of the chain of unknots. This kills the generator of $\pi_1(S^1 \times S^2)$ (see, e.g., \cite[Lemma~3.1]{Aceto-McCoy-Park:2020-1}), so the resulting manifold $W$ is simply connected.

Therefore, if $W$ is not simply connected, we must have $n'_1 = b_1$ and $n'_k = b_k$. Since $b_1,b_k \ge 2$ and $n_1 = n_k = 1$, it follows that $b_1 = b_k = 2$ and $k = 3$. Working Condition~\ref{it:bad_part} gives three possible values for $w'(v_2)$, and these three possibilities correspond to the cases $L(4,1)$, $L(8,3)$, and $L(12,5)$, respectively. In these three cases, $S^1 \times S^2$ is represented by surgery on a chain of unknots with surgery slopes $2,1,2$, respectively, and a $2$--handle is attached along the middle unknot. Since the framing of the first unknot is $2$, the meridian of the middle unknot represents twice a generator of the fundamental group of $S^1 \times S^2$. Consequently, we have $\pi_1(W) \cong \Z/2$.

% Namely, such a filling is obtained from an admissible $k$–tuple $(n'_1,\ldots,n'_k)$, defined from the $k$–tuple $(n_1,\ldots,n_k) := (1,2,\ldots,2,1)$ by
% \[
% n'_i =
% \begin{cases}
% 1 & \text{if } i = j,\\
% n_i + 1 & \text{if } |i-j| = 1,\\
% n_i & \text{otherwise,}
% \end{cases}
% \]
% where $v_j$ is the bad vertex.

% By the definition of bad vertex, if $|i-j|=1$, we have $w'(n_i)\geq 1$. Hence, we have $n'_i\leq b_i$ for all $i$. Moreover, blowing down on $n'_j$ produces the admissible $k-1$-tuple $(1,2,\ldots,2,1)$. Hence, we have $[n'_1,\ldots,n'_k]=0$, and hence the admissible $k$-tuple $(n'_1,\ldots,n'_k)$ represents a symplectic filling of $L$. Let $W$ be this symplectic filling.

To complete the proof, we show that every minimal symplectic filling $W$ of $(L,\xi_{\mathrm{st}})$ with $b_2(W) = b_2(X(p,q)) - 1$ arises from this construction. As mentioned earlier, \cite[Theorem~1.1(1)]{Lisca:2008-1} implies that every minimal symplectic filling is orientation-preserving diffeomorphic to a manifold constructed from an admissible $k$-tuple $(n_1,\ldots,n_k)$ satisfying \eqref{eq:admissible}. Moreover, \cite[Lemma~2.2]{Lisca:2008-1} shows that such a $k$-tuple is obtained from $(0)$ by a sequence of blowups, with no blowup performed at the left end.

Observe that each blowup at the right end increases the sum of the entries of the tuple by $2$, while each blowup away from the ends increases this sum by $3$. The standard $k$-tuple $(1,2,\ldots,2,1)$ is obtained by performing only blowups at the right end. Hence any symplectic filling $W$ with $b_2(W) = b_2(X(p,q)) - 1$ must correspond to a $k$-tuple obtained by performing exactly one blowup away from the ends, and all remaining blowups at the right end. Starting from $(0)$, performing a sequence of blowups at the right end gives $(1,2,\ldots,2,1)$. Then a single blowup not at an end yields one of
\[
(2,1,3,2,\ldots,2,1),\quad
(1,2,\ldots,2,3,1,3,2,\ldots,2,1),\quad\text{or}\quad
(1,2,\ldots,2,3,1,2).
\]
(Here any chain of $2$’s written as ``$2,\ldots,2$'' may be empty.) Performing the remaining blowups at the right end then produces a $k$-tuple $(n'_1,\ldots,n'_k)$ of the form defined above: the neighbors of the position of the interior blowup have nonzero adjusted weights, so the unique vertex corresponding to this interior blowup must be a bad vertex. This completes the proof.
\end{proof}

We now prove Corollary~\ref{cor:symplecticunique}, whose statement we recall below.  

\setcounter{cor}{0}
\renewcommand{\thecor}{1.6}
\begin{cor}\label{cor:symplecticuniquebody}
Let $L = L(p,q)$ be a lens space satisfying one of the items in Theorem~\ref{thm:main}, and let $\xi_{\mathrm{st}}$ denote the standard contact structure on $L$. Let $n(L)$ be the number of induced subgraphs in the canonical plumbing graph associated to $L$ that are isomorphic to one of the following configurations:
\[
\begin{tikzpicture}[xscale=1.0,yscale=1,baseline={(0,0)}]
    \node at (1-0.1,0.4) {$-4$};
    \node (A1) at (1,0) {$\bullet$};
\end{tikzpicture}
\qquad \qquad \qquad
\begin{tikzpicture}[xscale=1.0,yscale=1,baseline={(0,0)}]
    \node at (1-0.1, .4) {$-3$};
    \node at (2-0.1, .4) {$-3$};
    \node (A_1) at (1, 0) {$\bullet$};
    \node (A_2) at (2, 0) {$\bullet$};
    \path (A_1) edge [-] node [auto] {$\scriptstyle{}$} (A_2);
\end{tikzpicture}
\qquad  \qquad \qquad
\begin{tikzpicture}[xscale=1.0,yscale=1,baseline={(0,0)}]
    \node at (1-0.1, .4) {$-3$};
    \node at (2-0.1, .4) {$-2$};
    \node at (3-0.1, .4) {$-3$};
    \node (A_1) at (1, 0) {$\bullet$};
    \node (A_2) at (2, 0) {$\bullet$};
    \node (A_3) at (3, 0) {$\bullet$};
    \path (A_1) edge [-] node [auto] {$\scriptstyle{}$} (A_2);
    \path (A_2) edge [-] node [auto] {$\scriptstyle{}$} (A_3);
\end{tikzpicture}.
\]
Then the number of minimal symplectic fillings of $(L,\xi_{\mathrm{st}})$, up to orientation-preserving diffeomorphism, is either $n(L)+1$ or $n(L)$. Moreover, this number is $n(L)$ if and only if $q^2\equiv 1 \pmod p$ and the canonical plumbing graph associated to $L$ contains one of the following as an induced subgraph:
\[
\begin{tikzpicture}[xscale=1.0,yscale=1,baseline={(0,0)}]
    \node at (1-0.1, .4) {$-4$};
    \node at (2-0.1, .4) {$-4$};
    \node (A_1) at (1, 0) {$\bullet$};
    \node (A_2) at (2, 0) {$\bullet$};
    \path (A_1) edge [-] node [auto] {$\scriptstyle{}$} (A_2);
\end{tikzpicture}
\qquad  \qquad \qquad
\begin{tikzpicture}[xscale=1.0,yscale=1,baseline={(0,0)}]
    \node at (1-0.1, .4) {$-3$};
    \node at (2-0.1, .4) {$-3$};
    \node at (3-0.1, .4) {$-3$};
    \node (A_1) at (1, 0) {$\bullet$};
    \node (A_2) at (2, 0) {$\bullet$};
    \node (A_3) at (3, 0) {$\bullet$};
    \path (A_1) edge [-] node [auto] {$\scriptstyle{}$} (A_2);
    \path (A_2) edge [-] node [auto] {$\scriptstyle{}$} (A_3);
\end{tikzpicture}.
\]
Furthermore, an analogous statement holds for the number of Milnor fibers arising from the irreducible components of the reduced miniversal base space of the cyclic quotient singularity corresponding to $L$.
\end{cor}
\renewcommand{\thecor}{\thesection.\arabic{cor}}

\begin{proof}
Let $L=L(p,q)$ be a lens space as in Theorem~\ref{thm:main}, and let $X(p,q)$ denote its canonical negative-definite plumbing. Then every smooth negative-definite filling $X$ of $L$ satisfies $b_2(X) \ge b_2(X(p,q)) - 1$. Moreover, since each bad vertex corresponds to an induced subgraph isomorphic to one of
\[
\begin{tikzpicture}[xscale=1.0,yscale=1,baseline={(0,0)}]
    \node at (1-0.1,0.4) {$-4$};
    \node (A1) at (1,0) {$\bullet$};
\end{tikzpicture}
\qquad \qquad \qquad
\begin{tikzpicture}[xscale=1.0,yscale=1,baseline={(0,0)}]
    \node at (1-0.1, .4) {$-3$};
    \node at (2-0.1, .4) {$-3$};
    \node (A_1) at (1, 0) {$\bullet$};
    \node (A_2) at (2, 0) {$\bullet$};
    \path (A_1) edge [-] node [auto] {$\scriptstyle{}$} (A_2);
\end{tikzpicture}
\qquad  \qquad \qquad
\begin{tikzpicture}[xscale=1.0,yscale=1,baseline={(0,0)}]
    \node at (1-0.1, .4) {$-3$};
    \node at (2-0.1, .4) {$-2$};
    \node at (3-0.1, .4) {$-3$};
    \node (A_1) at (1, 0) {$\bullet$};
    \node (A_2) at (2, 0) {$\bullet$};
    \node (A_3) at (3, 0) {$\bullet$};
    \path (A_1) edge [-] node [auto] {$\scriptstyle{}$} (A_2);
    \path (A_2) edge [-] node [auto] {$\scriptstyle{}$} (A_3);
\end{tikzpicture},
\]
it follows that $n(L)$ is precisely the number of bad vertices in $\Gamma$.

% If every smooth negative-definite filling $X$ of $L$ satisfies
% $b_2(X)\geq b_2(X(p,q))$, the result follows from the fact that
% the symplectic filling obtained from the canonical resolution is
% diffeomorphic to $X(p,q)$~\cite[Corollary~1.7]{definite}. Thus we may assume that there exists some $X$ with $b_2(X)=b_2(X(p,q))-1$.

% As in the proof of Corollary~\ref{cor:symplecticbody}, we may assume that there exists some $X$ with $b_2(X)=b_2(X(p,q))-1$.

As in the proof of Corollary~\ref{cor:symplecticbody}, if every smooth negative-definite filling $X$ of $L$ satisfies
$b_2(X) \ge b_2(X(p,q))$, then $L$, equipped with the standard contact structure, admits a unique minimal symplectic filling. In this case we have $n(L)+1=1$, and the canonical plumbing graph associated to $L$ contains neither
\[
\begin{tikzpicture}[xscale=1.0,yscale=1,baseline={(0,0)}]
    \node at (1-0.1, .4) {$-4$};
    \node at (2-0.1, .4) {$-4$};
    \node (A_1) at (1, 0) {$\bullet$};
    \node (A_2) at (2, 0) {$\bullet$};
    \path (A_1) edge [-] node [auto] {$\scriptstyle{}$} (A_2);
\end{tikzpicture}
\qquad  \qquad \qquad
\begin{tikzpicture}[xscale=1.0,yscale=1,baseline={(0,0)}]
    \node at (1-0.1, .4) {$-3$};
    \node at (2-0.1, .4) {$-3$};
    \node at (3-0.1, .4) {$-3$};
    \node (A_1) at (1, 0) {$\bullet$};
    \node (A_2) at (2, 0) {$\bullet$};
    \node (A_3) at (3, 0) {$\bullet$};
    \path (A_1) edge [-] node [auto] {$\scriptstyle{}$} (A_2);
    \path (A_2) edge [-] node [auto] {$\scriptstyle{}$} (A_3);
\end{tikzpicture}
\]
as an induced subgraph, by \cite[Theorem~1.1]{definite}. Therefore we may assume that there exists a filling $X$ with $b_2(X) = b_2(X(p,q)) - 1$, and the dual plumbing graph $\Gamma$ then has at least one bad vertex.

% Let $v_1,\ldots,v_k$ denote the vertices of $\Gamma$, and for each $i$ let $v_i$ have weight $b_i$.

From the proof of Corollary~\ref{cor:symplecticbody}, we can conclude that up to orientation-diffeomorphism, there are at most $1+n(L)$ minimal symplectic fillings, with the extra 1 coming from the standard $(1,2,\dots ,2,1)$ tuple. From the classification of Lisca~\cite[Theorem~1.1(3)]{Lisca:2008-1}, the number of minimal symplectic fillings equals $1+n(L)$ unless we are in the $q=\overline{q}$ and $\textbf{n}=\overline{\textbf{n}}$, where $\overline{q}$ is defined to be the multiplicative inverse of $q$ modulo $p$, and $\overline{\textbf{n}}$ is defined to be $\textbf{n}$ with the order of the entries reversed.

The case $q=\overline{q}$ is equivalent to $q^2\equiv 1 \pmod p$. Since the canonical plumbing graph of $L(p,\overline{q})$ is obtained by reflecting the canonical plumbing graph of $L(p,q)$ (see, e.g., \cite{HNK:1971}), it follows that the canonical plumbing graph of $L(p,q)$ is symmetric, and hence $\Gamma$ is symmetric. The desired statement follows by inspecting which of the five configurations in the ``If there is more than one bad vertex'' case in the proof of Lemma~\ref{lem:working conditions proof} are symmetric. The conclusion about the Milnor fibers then follows from the one-to-one correspondence of~\cite{NemethiPopescu}.
\end{proof}

\bibliographystyle{alpha}
\def\MR#1{}
\bibliography{bib}
\appendix

\section{Computer data}\label{computer data}
Similar to \cite{definite}, we use the \texttt{OrthogonalEmbeddings} command in GAP to define a function \texttt{ListEmbeddings} that find all possible embeddings up to automorphisms of $\Z^n$ \cite{GAP}. However, different from \cite{definite}, there are a lot more cases we need to check. Not only working out all embeddings by hand would take an extreme amount of time, even manually checking whether the extended bad part in each embedding is standard or semi-standard after the GAP code lists out the possible embeddings is still very tedious. Therefore, we modify the \texttt{ListEmbeddings} function from \cite{definite} such that it automatically classifies whether the extended bad part is standard, semi-standard, or neither. To further simplify the task, we define a function \texttt{AllConfig} such that given a bad part configuration in Working Condition \ref{it:bad_part}, a description on how the graph can extend on the left side, and a description on how the graph can extend on the right side, it finds all possible configurations of the extended bad part and feed them all into \texttt{ListEmbeddings} all at once.

The code for defining those functions can be found in
\url{https://sites.google.com/view/antony-fung/codes}.

\medskip\noindent\textbf{Case 1:} When the bad part is
\[
\begin{tikzpicture}[xscale=1.0,yscale=1,baseline={(0,0)}]
    \node at (0,0.4) {$1$};
    \node at (1,0.4) {{$\red{\circled{0}}$}};
    \node at (2,0.4) {$1$};
    \node (A1) at (0,0) {$\circ$};
    \node (A2) at (1,0) {$\circ$};
    \node (A3) at (2,0) {$\circ$};
    \path (A1) edge [-] node [auto] {$\scriptstyle{}$} (A2);
    \path (A2) edge [-] node [auto] {$\scriptstyle{}$} (A3);
  \end{tikzpicture}.
\]
We investigate how can the graph extend on the right hand side. We denote the first vertex being extended to the right as $r_1$, and the second vertex $r_2$, etc. Throughout this appendix, we will keep implicitly using Working Condition \ref{it:largeweight}. So, keep in mind about Working Condition \ref{it:largeweight} when reading through the arguments below.

While extending the graph, we can stop as soon as it includes the end of the extended bad part, but we are also allowed to continue going to further restrict the embedding possibilities.

If $r_1$ does not exist, we have
\[
\begin{tikzpicture}[xscale=1.0,yscale=1,baseline={(0,0)}]
    \node at (0,0.4) {$2$};
    \node at (1,0.4) {$2$};
    \node at (2,0.4) {$2$};
    \node (A1) at (0,0) {$\bullet$};
    \node (A2) at (1,0) {$\bullet$};
    \node (A3) at (2,0) {$\bullet$};
    \path (A1) edge [-] node [auto] {$\scriptstyle{}$} (A2);
    \path (A2) edge [-] node [auto] {$\scriptstyle{}$} (A3);
  \end{tikzpicture}.
\]
If $r_1$ exists, by Working Condition \ref{it:forbidden_configs}\ref{it:1012}\ref{it:1013}, $w'(r_1)=0$ or $1$.

If $w'(r_1)=0$, since vertices with vanishing adjusted weight cannot be a leaf, $r_2$ exists. Since $r_1$ is not bad, $w'(r_2)=0$, so we have $r_3$. By Working Condition \ref{it:forbidden_configs}\ref{it:1002}\ref{it:1003}, $w'(r_3)=0$ or $1$. So, either $w(r_3)=2$ or $w(r_3)=3$, and if $w(r_3)=3$ then $w(r_4)$ must exists. By Working Condition \ref{it:forbidden_configs}\ref{it:31001}, $w'(r_4)\leq 2$. So, either $w(r_4)=2$, or $w(r_4)=3$, or $w(r_4)=4$ and $r_5$ exists. If $w(r_4)=4$ and $r_5$ exists, then since $r_4$ is not bad, $w(r_5)=2$.

Back to $w'(r_1)$. If $w'(r_1)=1$ and there is no $r_2$, then $w(r_1)=2$. If $w'(r_1)=1$ and $r_2$ exists, then $w(r_1)=3$, and since $r_1$ is not bad we must have $w'(r_2)=0$. So, similar as before, we have $w(r_2)=w(r_3)=2$ and $r_4$ exists. By Working Condition \ref{it:forbidden_configs}\ref{it:1002}\ref{it:1003}, $w'(r_4)=0$ or $1$. So, $w(r_4)=2$, or $w(r_4)=3$ and $r_5$ exists. If $w(r_4)=3$ and $r_5$ exists, by Working Condition \ref{it:weight3condition} and \ref{it:forbidden_configs}\ref{it:110012}, $w'(r_5)=0$ or $1$. So, $w(r_5)=2$, or $w(r_5)=3$ and $r_6$ exists. If $w(r_5)=3$ and $r_6$ exists, since $r_5$ is not bad, we must have $w'(r_6)=0$. Together, the weights of these extensions on the right are
\[
222,\ 2232,\ 2233,\ 22342,\ 2,\ 3222,\ 32232,\ 322332.
\]
Similarly, the weights of these extensions on the left are
\[
222,\ 2322,\ 3322,\ 24322,\ 2,\ 2223,\ 23223,\ 233223.
\]

We first run the code in \url{https://sites.google.com/view/antony-fung/codes} to define the function \texttt{AllConfig}. Then, we run the following command.

\texttt{AllConfig([2,2,2],[2],[[2,2,2],[2,3,2,2],[3,3,2,2],[2,4,3,2,2],[2],[2,2,2,3],}\\
\texttt{[2,3,2,2,3],[2,3,3,2,2,3]],[[2,2,2],[2,2,3,2],[2,2,3,3],[2,2,3,4,2],[2],[3,2,2,}\\
\texttt{2],[3,2,2,3,2],[3,2,2,3,3,2]],false);}

The first argument of the function is are weights of the bad part if there are no extensions. The second argument records the location of the bad vertices. In this case, it is in the second position (note that GAP array indices start at 1, not 0). The third and fourth arguments are the left and right extensions respectively. For the fifth entry, changing it to ``true'' will display the embeddings explicitly.\\

The output is
\begin{alltt}
Total number of embeddings:386
Total number of embeddings that are standard in the extended bad part:203
Total number of embeddings that are semi-standard in the extended bad part:183
Total number of embeddings that are neither in the extended bad part:0.
\end{alltt}

This confirms that in all these embeddings, the extended bad part is embedded either standardly or semi-standardly.

\medskip\noindent\textbf{Case 2:} When the bad part is
\[
\begin{tikzpicture}[xscale=1.0,yscale=1,baseline={(0,0)}]
    \node at (0,0.4) {$1$};
    \node at (1,0.4) {{$\red{\circled{0}}$}};
    \node at (2,0.4) {$2$};
    \node (A1) at (0,0) {$\circ$};
    \node (A2) at (1,0) {$\circ$};
    \node (A3) at (2,0) {$\circ$};
    \path (A1) edge [-] node [auto] {$\scriptstyle{}$} (A2);
    \path (A2) edge [-] node [auto] {$\scriptstyle{}$} (A3);
  \end{tikzpicture}.
\]
Similar to how we denote the right extension vertices as $r_1,r_2,\ldots$, we denote the left extension vertices as $l_1,l_2,\dots$.

By Working Conditions \ref{it:largeweight} and \ref{it:forbidden_configs}\ref{it:1102}, if $l_1$ exists, we must have $w'(l_1)=0$. Since $l_1$ is not bad and vertices with vanishing adjusted weight cannot be a leaf, that implies $l_2$ exists and $w'(l_2)$=0, and therefore $l_3$ exists. By Working Condition \ref{it:largeweight}, $w(l_3)=2$ or $3$. If $w(l_3)=3$, then $w'(l_3)=1$ and $l_4$ exists. By Working Condition \ref{it:largeweight}, $w(l_4)=2$ or $3$. If $w(l_4)=3$, then $w'(l_4)=1$ and $l_5$ exists. Since $l_4$ is not bad, we must have $w'(l_5)=0$, and hence $w(l_5)=2$. So, we input the left extensions 222,2322,23322.

By Working Conditions \ref{it:largeweight}, if $r_1$ exists, then $w'(r_1)=0$ or $1$.

If $w'(r_1)=0$, then similar as before, we have $w'(r_2)=0$ and $r_3$ exists. By Working Condition \ref{it:largeweight} and \ref{it:forbidden_configs}\ref{it:1002}, we have $w'(r_3)=0$. So, $r_4$ exists. By Working Condition \ref{it:largeweight}, $w(r_4)=2$ or $3$. So we get 2222, 2223.

If $w'(r_1)=1$, either $w(r_1)=2$ and $r_2$ does not exist, or $w(r_1)=3$ and $r_2$ exists. In the later case, since $r_1$ is not bad, we have $w'(r_2)=0$. Similar as before, then $w'(r_3)=0$. By Working Conditions \ref{it:largeweight}, $w(r_4)=2$ or $3$, and if $w(r_4)=3$ then $r_5$ exists. By Working Conditions \ref{it:forbidden_configs}\ref{it:110012}, $w'(r_5)=0$, and since $r_5$ is not bad we must have $w'(r_6)=0$. So, we get 2, 3222, 322322.

We run the following command.

\texttt{AllConfig([2,2,3],[2],[[2,2,2],[2,3,2,2],[2,3,3,2,2]],[[2,2,2,2],[2,2,2,3],}\\
\texttt{[2],[3,2,2,2],[3,2,2,3,2,2]],false);}\\

The output is
\begin{alltt}
Total number of embeddings:84
Total number of embeddings that are standard in the extended bad part:48
Total number of embeddings that are semi-standard in the extended bad part:36
Total number of embeddings that are neither in the extended bad part:0.
\end{alltt}

\medskip\noindent\textbf{Case 3:} When the bad part is
\[
\begin{tikzpicture}[xscale=1.0,yscale=1,baseline={(0,0)}]
    \node at (0,0.4) {$1$};
    \node at (1,0.4) {{$\red{\circled{0}}$}};
    \node at (2,0.4) {$3$};
    \node (A1) at (0,0) {$\circ$};
    \node (A2) at (1,0) {$\circ$};
    \node (A3) at (2,0) {$\circ$};
    \path (A1) edge [-] node [auto] {$\scriptstyle{}$} (A2);
    \path (A2) edge [-] node [auto] {$\scriptstyle{}$} (A3);
  \end{tikzpicture}.
\]
By Working Condition \ref{it:weight3condition}, if $l_1$ exists, we must have $w'(l_1)=0$. As before, in that case, $w'(l_2)=0$ and $l_3$ exists. By Working Condition \ref{it:largeweight}, $w'(l_3)=0$ or $1$. So, either $w(l_3)=2$, or $w(l_3)=3$ and $l_4$ exists. By Working Condition \ref{it:weight3condition}, if $l_4$ exists, then $w'(l_4)=0$. As $l_4$ is not bad, in that case we have $w'(l_5)=0$. So, we input 222 and 22322 as left extensions.

If $r_1$ exists and $w'(r_1)=0$, then since $r_1$ is not bad, and also by Working Condition \ref{it:forbidden_configs}\ref{it:1003}\ref{it:10003}, we have $w'(r_2)=w'(r_3)=w'(r_4)=0$. Then $r_5$ exists and $w(r_5)=2$ or $3$. Hence, we have 22222, 22223.

If $r_1$ exists and $w'(r_1)=1$, then either $w(r_1)=2$, or $w(r_1)=3$ and $r_2$ exists. In the later case, as $r_1$ is not bad, we have $w'(r_2)=0$. As $r_2$ is not bad, we have $w'(r_3)=0$. By Working Condition \ref{it:forbidden_configs}\ref{it:31001}, we have $w'(r_4)=0$. By considering $r_5$ as well, we get 2, 32222, 32223.

We run the following command.

\texttt{AllConfig([2,2,4],[2],[[2,2,2],[2,2,3,2,2]],[[2,2,2,2,2],[2,2,2,2,3],[2],[3,2,}\\
\texttt{2,2,2],[3,2,2,2,3]],false);}\\

The output is
\begin{alltt}
Total number of embeddings:64
Total number of embeddings that are standard in the extended bad part:32
Total number of embeddings that are semi-standard in the extended bad part:32
Total number of embeddings that are neither in the extended bad part:0.
\end{alltt}

\medskip\noindent\textbf{Case 4:} When the bad part is
\[
\begin{tikzpicture}[xscale=1.0,yscale=1,baseline={(0,0)}]
    \node at (0,0.4) {$1$};
    \node at (1,0.4) {{$\red{\circled{1}}$}};
    \node at (2,0.4) {$1$};
    \node (A1) at (0,0) {$\circ$};
    \node (A2) at (1,0) {$\circ$};
    \node (A3) at (2,0) {$\circ$};
    \path (A1) edge [-] node [auto] {$\scriptstyle{}$} (A2);
    \path (A2) edge [-] node [auto] {$\scriptstyle{}$} (A3);
  \end{tikzpicture}.
\]
Since only the middle ``1'' is bad, if $r_1$ exists, we must have $w'(r_1)=w'(r_2)=0$. By Working Condition \ref{it:bad1}, we have $w'(r_3)\leq 1$. So, either $w(r_3)=2$, or $w(r_3)=3$ and $r_4$ exists. Similarly, if $w(r_4)$ exists, then either $w(r_4)=2$, or $w(r_4)=3$ and $r_5$ exists. If $w(r_3)=w(r_4)=3$, then since $r_4,r_5$ are not bad, we have $w'(r_5)=w'(r_6)=0$. Hence, we input 222, 2232, 223322 as right extensions. Similarly, we input 222, 2322, 223322 as left extensions.

We run the following command.

\texttt{AllConfig([2,3,2],[2],[[2,2,2],[2,3,2,2],[2,2,3,3,2,2]],[[2,2,2],[2,2,3,2],[2,}\\
\texttt{2,3,3,2,2]],false);}\\

The output is
\begin{alltt}
Total number of embeddings:50
Total number of embeddings that are standard in the extended bad part:25
Total number of embeddings that are semi-standard in the extended bad part:25
Total number of embeddings that are neither in the extended bad part:0.
\end{alltt}

\medskip\noindent\textbf{Case 5:} When the bad part is
\[
\begin{tikzpicture}[xscale=1.0,yscale=1,baseline={(0,0)}]
    \node at (0,0.4) {$1$};
    \node at (1,0.4) {{$\red{\circled{1}}$}};
    \node at (2,0.4) {$2$};
    \node (A1) at (0,0) {$\circ$};
    \node (A2) at (1,0) {$\circ$};
    \node (A3) at (2,0) {$\circ$};
    \path (A1) edge [-] node [auto] {$\scriptstyle{}$} (A2);
    \path (A2) edge [-] node [auto] {$\scriptstyle{}$} (A3);
  \end{tikzpicture}.
\]

Since only the middle ``1'' is bad, if $r_1$ exists, we must have $w'(r_1)=w'(r_2)=0$. By Working Condition \ref{it:forbidden_configs}\ref{it:1002}, $w'(r_3)=0$. Hence, by considering $r_4,r_5,r_6$ and the fact that $r_5$ (if exists) cannot be bad, we input the right extensions 2222, 22232, 222332.

The left side is the same as Case 4. 

We run the following command.

\texttt{AllConfig([2,3,3],[2],[[2,2,2],[2,3,2,2],[2,2,3,3,2,2]],[[2,2,2,2],[2,2,2,3,}\\
\texttt{2],[2,2,2,3,3,2]],false);}\\

The output is
\begin{alltt}
Total number of embeddings:60
Total number of embeddings that are standard in the extended bad part:30
Total number of embeddings that are semi-standard in the extended bad part:30
Total number of embeddings that are neither in the extended bad part:0.
\end{alltt}

\medskip\noindent\textbf{Case 6:} When the bad part is
\[
\begin{tikzpicture}[xscale=1.0,yscale=1,baseline={(0,0)}]
    \node at (0,0.4) {$1$};
    \node at (1,0.4) {{$\red{\circled{2}}$}};
    \node at (2,0.4) {$1$};
    \node (A1) at (0,0) {$\circ$};
    \node (A2) at (1,0) {$\circ$};
    \node (A3) at (2,0) {$\circ$};
    \path (A1) edge [-] node [auto] {$\scriptstyle{}$} (A2);
    \path (A2) edge [-] node [auto] {$\scriptstyle{}$} (A3);
  \end{tikzpicture}.
\]
Since only the middle ``1'' is bad, if $r_1$ exists, we must have $w'(r_1)=w'(r_2)=0$. By Working Condition \ref{it:largeweight}, we have $w(r_3)=2$ or $3$, and if $w(r_3)=3$ then $r_4$ exists. In the later case, by Working Condition \ref{it:bad2}, we have $w'(r_4)=0$. Since $r_4$ is not bad, we have $w'(r_5)=0$. Hence, we input 222, 22322 as right extensions. Same for left extensions.

We run the following command.

\texttt{AllConfig([2,4,2],[2],[[2,2,2],[2,2,3,2,2]],[[2,2,2],[2,2,3,2,2]],false);}\\

The output is
\begin{alltt}
Total number of embeddings:32
Total number of embeddings that are standard in the extended bad part:16
Total number of embeddings that are semi-standard in the extended bad part:16
Total number of embeddings that are neither in the extended bad part:0.
\end{alltt}

\medskip\noindent\textbf{Case 7:} When the bad part is
\[
\begin{tikzpicture}[xscale=1.0,yscale=1,baseline={(0,0)}]
    \node at (0,0.4) {$1$};
    \node at (1,0.4) {{$\red{\circled{0}}$}};
    \node at (2,0.4) {$1$};
    \node at (3,0.4) {{$\red{\circled{0}}$}};
    \node at (4,0.4) {$1$};
    \node (A1) at (0,0) {$\circ$};
    \node (A2) at (1,0) {$\circ$};
    \node (A3) at (2,0) {$\circ$};
    \node (A4) at (3,0) {$\circ$};
    \node (A5) at (4,0) {$\circ$};
    \path (A1) edge [-] node [auto] {$\scriptstyle{}$} (A2);
    \path (A2) edge [-] node [auto] {$\scriptstyle{}$} (A3);
    \path (A3) edge [-] node [auto] {$\scriptstyle{}$} (A4);
    \path (A4) edge [-] node [auto] {$\scriptstyle{}$} (A5);
  \end{tikzpicture}.
\]
The left side and the right side are the same as Case 1. Therefore, we only change the first two arguments of the function. We run the following command.

\texttt{AllConfig([2,2,3,2,2],[2,4],[[2,2,2],[2,3,2,2],[3,3,2,2],[2,4,3,2,2],[2],[2,}\\
\texttt{2,2,3],[2,3,2,2,3],[2,3,3,2,2,3]],[[2,2,2],[2,2,3,2],[2,2,3,3],[2,2,3,4,2],[2]}\\
\texttt{,[3,2,2,2],[3,2,2,3,2],[3,2,2,3,3,2]],false);}\\

The output is
\begin{alltt}
Total number of embeddings:589
Total number of embeddings that are standard in the extended bad part:203
Total number of embeddings that are semi-standard in the extended bad part:386
Total number of embeddings that are neither in the extended bad part:0.
\end{alltt}

\medskip\noindent\textbf{Case 8:} When the bad part is
\[
\begin{tikzpicture}[xscale=1.0,yscale=1,baseline={(0,0)}]
    \node at (0,0.4) {$1$};
    \node at (1,0.4) {{$\red{\circled{0}}$}};
    \node at (2,0.4) {$1$};
    \node at (3,0.4) {{$\red{\circled{0}}$}};
    \node at (4,0.4) {$2$};
    \node (A1) at (0,0) {$\circ$};
    \node (A2) at (1,0) {$\circ$};
    \node (A3) at (2,0) {$\circ$};
    \node (A4) at (3,0) {$\circ$};
    \node (A5) at (4,0) {$\circ$};
    \path (A1) edge [-] node [auto] {$\scriptstyle{}$} (A2);
    \path (A2) edge [-] node [auto] {$\scriptstyle{}$} (A3);
    \path (A3) edge [-] node [auto] {$\scriptstyle{}$} (A4);
    \path (A4) edge [-] node [auto] {$\scriptstyle{}$} (A5);
  \end{tikzpicture}.
\]
For the left side, it is the same as Case 1, except that we can remove 24322 because of Working Condition \ref{it:largeweight}, and we can replace 3322 by 23322 because by Working Condition \ref{it:largeweight} if $w(l_4)=3$ then it cannot be a leaf, and since $l_4$ is not bad, we have $w'(l_5)=0$.

The right side is the same as Case 2.

We run the following command.

\texttt{AllConfig([2,2,3,2,3],[2,4],[[2,2,2],[2,3,2,2],[2,3,3,2,2],[2],[2,2,2,3],[2,3,}\\
\texttt{2,2,3],[2,3,3,2,2,3]],[[2,2,2,2],[2,2,2,3],[2],[3,2,2,2],[3,2,2,3,2,2]],false);}\\

The output is
\begin{alltt}
Total number of embeddings:264
Total number of embeddings that are standard in the extended bad part:96
Total number of embeddings that are semi-standard in the extended bad part:168
Total number of embeddings that are neither in the extended bad part:0.
\end{alltt}

\medskip\noindent\textbf{Case 9:} When the bad part is
\[
\begin{tikzpicture}[xscale=1.0,yscale=1,baseline={(0,0)}]
    \node at (0,0.4) {$1$};
    \node at (1,0.4) {{$\red{\circled{0}}$}};
    \node at (2,0.4) {$1$};
    \node at (3,0.4) {{$\red{\circled{1}}$}};
    \node at (4,0.4) {$1$};
    \node (A1) at (0,0) {$\circ$};
    \node (A2) at (1,0) {$\circ$};
    \node (A3) at (2,0) {$\circ$};
    \node (A4) at (3,0) {$\circ$};
    \node (A5) at (4,0) {$\circ$};
    \path (A1) edge [-] node [auto] {$\scriptstyle{}$} (A2);
    \path (A2) edge [-] node [auto] {$\scriptstyle{}$} (A3);
    \path (A3) edge [-] node [auto] {$\scriptstyle{}$} (A4);
    \path (A4) edge [-] node [auto] {$\scriptstyle{}$} (A5);
  \end{tikzpicture}.
\]
For the left side, we use the same argument as in Case 8, but using Working Condition \ref{it:bad1} instead of Working Condition \ref{it:largeweight}.

The right side is the same as Case 4.

We run the following command.

\texttt{AllConfig([2,2,3,3,2],[2,4],[[2,2,2],[2,3,2,2],[2,3,3,2,2],[2],[2,2,2,3],[2,3,}\\
\texttt{2,2,3],[2,3,3,2,2,3]],[[2,2,2],[2,2,3,2],[2,2,3,3,2,2]],false);}\\

The output is
\begin{alltt}
Total number of embeddings:160
Total number of embeddings that are standard in the extended bad part:50
Total number of embeddings that are semi-standard in the extended bad part:110
Total number of embeddings that are neither in the extended bad part:0.
\end{alltt}

\medskip\noindent\textbf{Case 10:} When the bad part is
\[
\begin{tikzpicture}[xscale=1.0,yscale=1,baseline={(0,0)}]
    \node at (0,0.4) {$1$};
    \node at (1,0.4) {{$\red{\circled{0}}$}};
    \node at (2,0.4) {$1$};
    \node at (3,0.4) {{$\red{\circled{1}}$}};
    \node at (4,0.4) {$2$};
    \node (A1) at (0,0) {$\circ$};
    \node (A2) at (1,0) {$\circ$};
    \node (A3) at (2,0) {$\circ$};
    \node (A4) at (3,0) {$\circ$};
    \node (A5) at (4,0) {$\circ$};
    \path (A1) edge [-] node [auto] {$\scriptstyle{}$} (A2);
    \path (A2) edge [-] node [auto] {$\scriptstyle{}$} (A3);
    \path (A3) edge [-] node [auto] {$\scriptstyle{}$} (A4);
    \path (A4) edge [-] node [auto] {$\scriptstyle{}$} (A5);
  \end{tikzpicture}.
\]
For the left side, if $l_1$ exists, by Working Condition \ref{it:forbidden_configs}\ref{it:110112}, we have $w'(l_1)=0$. Since $l_1$ is not bad, we have $w'(l_2)=0$. By considering $l_3,l_4,l_5$ and the fact that $l_4$ (if exists) cannot be bad, we input the left extensions 222, 2322, 23322.

The right side is the same as Case 5.
	
We run the following command.

\texttt{AllConfig([2,2,3,3,3],[2,4],[[2,2,2],[2,3,2,2],[2,3,3,2,2]],[[2,2,2,2],[2,2,2,}\\
\texttt{3,2],[2,2,2,3,3,2]],false);}\\

The output is
\begin{alltt}
Total number of embeddings:108
Total number of embeddings that are standard in the extended bad part:36
Total number of embeddings that are semi-standard in the extended bad part:72
Total number of embeddings that are neither in the extended bad part:0.
\end{alltt}

\medskip\noindent\textbf{Case 11:} When the bad part is
\[
\begin{tikzpicture}[xscale=1.0,yscale=1,baseline={(0,0)}]
    \node at (0,0.4) {$1$};
    \node at (1,0.4) {{$\red{\circled{1}}$}};
    \node at (2,0.4) {{$\red{\circled{1}}$}};
    \node at (3,0.4) {$1$};
    \node (A1) at (0,0) {$\circ$};
    \node (A2) at (1,0) {$\circ$};
    \node (A3) at (2,0) {$\circ$};
    \node (A4) at (3,0) {$\circ$};
    \path (A1) edge [-] node [auto] {$\scriptstyle{}$} (A2);
    \path (A2) edge [-] node [auto] {$\scriptstyle{}$} (A3);
    \path (A3) edge [-] node [auto] {$\scriptstyle{}$} (A4);
  \end{tikzpicture}.
\]
Both the left side and the right side are the same as Case 4.
	
We run the following command.

\texttt{AllConfig([2,3,3,2],[2,3],[[2,2,2],[2,3,2,2],[2,2,3,3,2,2]],[[2,2,2],[2,2,3,}\\
\texttt{2],[2,2,3,3,2,2]],false);}\\

The output is
\begin{alltt}
Total number of embeddings:75
Total number of embeddings that are standard in the extended bad part:25
Total number of embeddings that are semi-standard in the extended bad part:50
Total number of embeddings that are neither in the extended bad part:0.
\end{alltt}

\medskip\noindent\textbf{Case 12:} When the bad part is
\[
\begin{tikzpicture}[xscale=1.0,yscale=1,baseline={(0,0)}]
    \node at (0,0.4) {$1$};
    \node at (1,0.4) {{$\red{\circled{1}}$}};
    \node at (2,0.4) {$1$};
    \node at (3,0.4) {{$\red{\circled{1}}$}};
    \node at (4,0.4) {$1$};
    \node (A1) at (0,0) {$\circ$};
    \node (A2) at (1,0) {$\circ$};
    \node (A3) at (2,0) {$\circ$};
    \node (A4) at (3,0) {$\circ$};
    \node (A5) at (4,0) {$\circ$};
    \path (A1) edge [-] node [auto] {$\scriptstyle{}$} (A2);
    \path (A2) edge [-] node [auto] {$\scriptstyle{}$} (A3);
    \path (A3) edge [-] node [auto] {$\scriptstyle{}$} (A4);
    \path (A4) edge [-] node [auto] {$\scriptstyle{}$} (A5);
  \end{tikzpicture}.
\]
Both the left side and the right side are the same as Case 4.

We run the following command.

\texttt{AllConfig([2,3,3,3,2],[2,3,4],[[2,2,2],[2,3,2,2],[2,2,3,3,2,2]],[[2,2,2],}\\
\texttt{[2,2,3,2],[2,2,3,3,2,2]],false);}\\

The output is
\begin{alltt}
Total number of embeddings:100
Total number of embeddings that are standard in the extended bad part:25
Total number of embeddings that are semi-standard in the extended bad part:75
Total number of embeddings that are neither in the extended bad part:0.
\end{alltt}

\medskip\noindent\textbf{Case 13:} When the bad part is
\[
\begin{tikzpicture}[xscale=1.0,yscale=1,baseline={(0,0)}]
    \node at (0,0.4) {$1$};
    \node at (1,0.4) {{$\red{\circled{1}}$}};
    \node at (2,0.4) {{$\red{\circled{2}}$}};
    \node at (3,0.4) {$1$};
    \node (A1) at (0,0) {$\circ$};
    \node (A2) at (1,0) {$\circ$};
    \node (A3) at (2,0) {$\circ$};
    \node (A4) at (3,0) {$\circ$};
    \path (A1) edge [-] node [auto] {$\scriptstyle{}$} (A2);
    \path (A2) edge [-] node [auto] {$\scriptstyle{}$} (A3);
    \path (A3) edge [-] node [auto] {$\scriptstyle{}$} (A4);
  \end{tikzpicture}.
\]
For the left side, same as Case 4, we get 222, 2322, 223322. However, in the 2322 case, by using Working Condition \ref{it:bad2}, we know that $l_4$ cannot be a leaf. Since $l_4$ is not bad, we have $w'(l_5)=0$ as well. Also by Working Condition \ref{it:bad2}, we can rule out the 223322 case. Hence, we input 222, 22322 as left extensions.

The right side is the same as Case 6.

We run the following command.

\texttt{AllConfig([2,3,4,2],[2,3],[[2,2,2],[2,2,3,2,2]],[[2,2,2],[2,2,3,2,2]],false);}\\

The output is
\begin{alltt}
Total number of embeddings:48
Total number of embeddings that are standard in the extended bad part:16
Total number of embeddings that are semi-standard in the extended bad part:32
Total number of embeddings that are neither in the extended bad part:0.
\end{alltt}

\end{document}